\documentclass{amsart}
\usepackage[utf8]{inputenc}
\usepackage{lmodern}
\usepackage{mathrsfs} 
\usepackage{wrapfig}
\usepackage{multirow}
\usepackage{tabu}
\usepackage{amssymb}
\usepackage{bm}
\usepackage{graphicx}
\usepackage[centertags]{amsmath}
\usepackage{amsfonts}
\usepackage{amsthm}
\usepackage{enumerate}
\usepackage{tocvsec2}
\usepackage{xcolor}
\usepackage[margin=1in]{geometry}

\newtheorem{theorem}{Theorem}[section]
\newtheorem*{theorem*}{Theorem}

\newtheorem{corollary}[theorem]{Corollary}
\newtheorem{lemma}[theorem]{Lemma}
\newtheorem{rem}[theorem]{Remark}

\newtheorem{proposition}[theorem]{Proposition}

\newtheorem{fact}[theorem]{Fact}

\theoremstyle{definition}

\newcommand{\ee}{\varepsilon}
\newcommand{\nn}{\mathbb{N}}

\begin{document}

\title{The $\xi,\zeta$-Dunford Pettis property}

\begin{abstract} Using the hierarchy of weakly null sequences introduced in \cite{AMT}, we introduce two new families of operator classes. The first family simultaneously generalizes the completely continuous operators and the weak Banach-Saks operators. The second family generalizes the class $\mathfrak{DP}$. We study the distinctness of these classes, and prove that each class is an operator ideal. We also investigate the properties possessed by each class, such as injectivity, surjectivity, and identification of the dual class.  We produce a number of examples, including the higher ordinal Schreier and Baernstein spaces.  We prove ordinal analogues of several known results for Banach spaces with the Dunford-Pettis, hereditary Dunford-Pettis property, and hereditary by quotients Dunford-Pettis property. For example, we prove that for any $0\leqslant \xi, \zeta<\omega_1$, a Banach space $X$ has the hereditary $\omega^\xi, \omega^\zeta$-Dunford Pettis property if and only if every seminormalized, weakly null sequence either has a subsequence which is an $\ell_1^{\omega^\xi}$-spreading model or a $c_0^{\omega^\zeta}$-spreading model.

\end{abstract}

\author{R.M. Causey}
\address{Department of Mathematics, Miami University, Oxford, OH 45056, USA}
\email{causeyrm@miamioh.edu}


\maketitle



\section{Introduction}

In \cite{DP}, Dunford and Pettis showed that any weakly compact operator defined on an $L_1(\mu)$ space must be completely continuous (sometimes also called a Dunford-Pettis operator).  In \cite{Grothendieck}, Grothendieck showed that $C(K)$ spaces enjoy the same property. That is, any weakly compact operator defined on a $C(K)$ domain is also completely continuous.   Now, we say a Banach space $X$ has the \textbf{Dunford-Pettis property} provided that for any Banach space $Y$ and any weakly compact operator $A:X\to Y$, $A$ is completely continuous.     A standard characterization of this property is as follows: $X$ has the Dunford-Pettis Property if for any weakly null sequences $(x_n)_{n=1}^\infty\subset X$, $(x^*_n)_{n=1}^\infty\subset X^*$, $\lim_n x^*_n(x_n)=0$.  Generalizing this, one can study the class of operators $A:X\to Y$ such that for any weakly null sequences $(x_n)_{n=1}^\infty\subset X$ and $(y^*_n)_{n=1}^\infty\subset Y^*$, $\lim_n y^*_n(Ax_n)=0$.   This class of operators has been denoted in the literature by $\mathfrak{DP}$, although it is not to be confused with the class of Dunford-Pettis operators, $\mathfrak{V}$. Then the Banach space $X$ has the Dunford-Pettis property if and only if $I_X\in \mathfrak{DP}$.

By the well-known Mazur lemma, if $X$ is a Banach space and $(x_n)_{n=1}^\infty$ is a weakly null sequence in $X$, then $(x_n)_{n=1}^\infty$ admits a norm null convex block sequence. Of course, the simplest form of convex block sequences would be one in which all coefficients are equal to $1$, in which case the convex block sequence of $(x_n)_{n=1}^\infty$ is actually a subsequence. The next simplest form of a convex block sequence is a sequence of Cesaro means.  A property of significant interest is whether the sequence $(x_n)_{n=1}^\infty$ has a subsequence (or whether every subsequence of $(x_n)_{n=1}^\infty$ has a further subsequence) whose Cesaro means converge to zero in norm.  A weakly null sequence having the property that every subsequence has a further subsesquence whose Cesaro means converge to zero in norm is sometimes called \emph{uniformly weakly null}. Schreier \cite{Schreier} produced an example of a weakly null sequence which is not uniformly weakly null.    More generally, there is a hierarchy of weak nullity fully elucidated by Argyros, Merkourakis, and Tsarpalilas \cite{AMT} indexed by countable ordinals.   As described above, norm null sequences are $0$-weakly null, uniformly weakly null sequences are $1$-weakly null, and for every countable ordinal $\xi$ there exists a weakly null sequence which is $\xi$-weakly null and not $\zeta$-weakly null for any $\zeta<\xi$. By convention, we establish that a sequence is said to be $\omega_1$-weakly null if it is weakly null.    Consistent with this convention is the fact that for any $0\leqslant \xi\leqslant \zeta\leqslant \omega_1$, every sequence which is $\xi$-weakly null is $\zeta$-weakly null.    The ordinal quantification assigns to a given weakly null sequence some measure of how complex the convex coefficients of a norm null convex block sequence must be.

This yields a natural generalization of the class $\mathfrak{DP}$.  Given an operator $A:X\to Y$, rather than asking that every weakly null sequence in $(x_n)_{n=1}^\infty\subset X$ and any weakly null sequence $(y^*_n)_{n=1}^\infty \subset Y^*$, $\lim_n y^*_n(Ax_n)=0$, we may instead only require the weaker condition that every pair of sequences $(x_n)_{n=1}^\infty\subset X$, $(y^*_n)_{n=1}^\infty \subset Y^*$ which are ``very'' weakly null, $\lim_n y^*_n(Ax_n)=0$.    Formally, for any $0\leqslant \xi, \zeta\leqslant \omega_1$, we let $\mathfrak{M}_{\xi, \zeta}$ denote the class of all operators $A:X\to Y$ such that for every $\xi$-weakly null $(x_n)_{n=1}^\infty\subset X$ and every $\zeta$-weakly null $(y^*_n)_{n=1}^\infty \subset Y^*$, $\lim_n y^*_n(Ax_n)=0$. We let $\textsf{M}_{\xi, \zeta}$ denote the class of all Banach spaces $X$ such that $I_X\in \mathfrak{M}_{\xi, \zeta}$. Then $\mathfrak{DP}=\mathfrak{M}_{\omega_1, \omega_1}$ and $\textsf{M}_{\omega_1, \omega_1}$ is the class of all Banach spaces with the Dunford-Pettis property.    Note that  every operator lies in $\mathfrak{DP}_{\xi, \zeta}$ when $\min \{\xi, \zeta\}=0$, since $0$-weakly null sequences are norm null. Thus we are interested in studying  the classes $\mathfrak{M}_{\xi, \eta}$ only for $0<\xi, \zeta$.  Furthermore, one may ask for a characterization, as one does with the Dunford-Pettis property, of Banach spaces all of whose subspaces, or all of whose quotients, enjoy a given property (in our case, membership in $\textsf{M}_{\xi, \zeta}$).    We note that the classes $\textsf{M}_{1, \omega_1}$ were introduced and studied in \cite{GG}, while the classes $\textsf{M}_{\omega_1, \xi}$, were introduced and studied in \cite{AG}. The study of classes of operators with these weakened Dunford-Pettis conditions  rather than spaces with these conditions is new to this work.  Along these lines, we have the following results.

\begin{theorem} For every $0<\xi, \zeta\leqslant \omega_1$, $\mathfrak{M}_{\xi, \zeta}$ is a closed ideal which is  not injective, surjective, or symmetric. Moreover, the ideals $(\mathfrak{M}_{\xi, \zeta})_{0<\xi, \zeta\leqslant \omega_1}$ are distinct.

\end{theorem}

In addition to generalizations of the Dunford-Pettis property, one may use the quantified weak nullity to generalize other classes of operators.  Two classes of interest are the classes $\mathfrak{V}$ of completely continuous operators and $\mathfrak{wBS}$ of weak Banach-Saks operators. Also of interest are the associated space ideals $\textsf{V}$ of Schur spaces and $\textsf{wBS}$ of weak Banach-Saks spaces.  The concepts behind these classes are that weakly null sequences are mapped by the operator to sequences which are ``very'' weakly null (completely continuous operators send weakly null sequences to $0$-weakly null sequences, and weak Banach-Saks operators send weakly null sequences to $1$-weakly null sequences).   In \cite{CN}, the notions of $\xi$-completely continuous operators and $\xi$-Schur Banach spaces were introduced. These notions are weakenings of the notions of completely continuous operators and Schur Banach spaces, respectively.    An operator  is $\xi$-completely continuous if it sends $\xi$-weakly null sequences to norm null ($0$-weakly null) sequences.  Heuristically, this is an operator which sends sequences which are ``not too bad'' to sequences which are ``good.''  In \cite{BC}, the notion of $\xi$-weak Banach-Saks was introduced.  An operator  is $\xi$-weak Banach-Saks if it sends  weakly null sequences to  $\xi$-weakly null sequences. Heuristically, this is an operator which sends any weakly null sequence, regardless of how ``bad'' it is,  to sequences which are ``not too bad.''    Of course, there is a simultaneous  generalization of both of these notions.  For $0\leqslant \zeta<\xi\leqslant \omega_1$, we let $\mathfrak{G}_{\xi, \zeta}$ denote the class of operators which send $\xi$-weakly null sequences to $\zeta$-weakly null sequences. Along these lines, we prove the following.

\begin{theorem} For every $0\leqslant \zeta<\xi\leqslant \omega_1$, $\mathfrak{G}_{\xi, \zeta}$ is a closed, injective ideal which fails to be surjective or symmetric. These ideals are distinct. 

\end{theorem}

We also recall the stratification $(\mathfrak{W}_\xi)_{0\leqslant \xi\leqslant \omega_1}$ of the weakly compact operators. The class $\mathfrak{W}_\xi$ is called the class of $\xi$-\emph{Banach-Saks operators}.    We recall the basic facts of these classes and basic facts about operator classes, including the quotients $\mathfrak{A}\circ \mathfrak{B}^{-1}$ and $\mathfrak{B}^{-1}\circ \mathfrak{A}$,  in Section $3$.      We note that $\mathfrak{W}_0$ is the class of compact operators, also denoted by $\mathfrak{K}$. The class of weakly compact operators is denoted by $\mathfrak{W}$ and $\mathfrak{W}_{\omega_1}$, and $\mathfrak{W}_1$ denotes the class of Banach-Saks operators.    It is a well-known identity regarding completely continuous operators that $\mathfrak{V}=\mathfrak{K}\circ \mathfrak{W}^{-1}$. It is also standard that $\mathfrak{DP}=\mathfrak{W}^{-1}\circ \mathfrak{V}=\mathfrak{W}^{-1}\circ \mathfrak{K}\circ \mathfrak{W}^{-1}$.  Rewriting theses identities using the ordinal notation for these classes gives $$\mathfrak{V}_{\omega_1}= \mathfrak{W}_0\circ \mathfrak{W}_{\omega_1}^{-1} $$ and $$\mathfrak{M}_{\omega_1, \omega_1}= \mathfrak{W}^{-1}_{\omega_1}\circ \mathfrak{K}\circ \mathfrak{W}_{\omega_1}^{-1}.$$       We generalize these identities in the following theorem.

\begin{theorem} For $0\leqslant \zeta<\xi\leqslant \omega_1$, $$\mathfrak{G}_{\xi, \zeta}=\mathfrak{W}_\zeta\circ \mathfrak{W}_\xi^{-1}$$ and $$\mathfrak{G}_{\xi, \zeta}^\text{\emph{dual}}= (\mathfrak{W}_\xi^\text{\emph{dual}})^{-1}\circ \mathfrak{W}_\zeta^\text{\emph{dual}}.$$

For $0<\zeta, \xi\leqslant \omega_1$, $$\mathfrak{M}_{\xi, \zeta}= (\mathfrak{W}^\text{\emph{dual}}_\zeta)^{-1}\circ \mathfrak{V}_\xi = (\mathfrak{W}^\text{\emph{dual}}_\zeta)^{-1}\circ \mathfrak{K}\circ \mathfrak{W}_\xi^{-1}.$$

\end{theorem}

The appearance of $\mathfrak{W}_\xi^\text{dual}$, rather than simply $\mathfrak{W}_\xi$ as it appeared in the identities preceding the theorem are due to the fact that $\mathfrak{W}_0=\mathfrak{K}=\mathfrak{K}^\text{dual}=\mathfrak{W}_0^\text{dual}$ and $\mathfrak{W}_{\omega_1}=\mathfrak{W}=\mathfrak{W}^\text{dual}=\mathfrak{W}_{\omega_1}$, while $\mathfrak{W}_\xi\neq \mathfrak{W}_\xi^\text{dual}$ for $0<\xi<\omega_1$.  This duality is known to fail for all $0<\xi<\omega_1$. The failure for $\xi=1$ is the classical fact that the Banach-Saks property is not a self-dual property, while the $1<\xi<\omega_1$ cases are generalizations of this.

We say  Banach space $X$ is \emph{hereditarily} $\textsf{M}_{\xi, \zeta}$ if for every every closed subspace $Y$ of $X$, $Y\in \textsf{M}_{\xi, \zeta}$.       We say $X$ is \emph{hereditary by quotients} $\textsf{M}_{\xi, \zeta}$ if for every closed subspace $Y$ of $X$, $X/Y\in \textsf{M}_{\xi, \zeta}$.  In Section $2$, we define the relevant notions regarding $\ell_1^\xi$ and $c_0^\zeta$-spreading models. We also adopt the convention that a sequence which is equivalent to the canonical $c_0$ basis will be called a $c_0^{\omega_1}$-spreading model.  We summarize our results regarding these hereditary and spatial notions in the following theorem. We note that item $(i)$ of the following theorem generalizes a characterization of the hereditary Dunford-Pettis property due to Elton, as well as a characterization of the hereditary $\zeta$-Dunford-Pettis property defined by Argyros and Gasparis.

\begin{theorem} Fix $0<\xi, \zeta\leqslant \omega_1$. \begin{enumerate}[(i)]\item $X$ is hereditarily $\textsf{\emph{M}}_{\xi, \zeta}$ if every $\xi$-weakly null sequence has a subsequence which is a $c_0^\zeta$-spreading model. \item $X$ is hereditary by quotients $\textsf{\emph{M}}_{\omega_1, \zeta}$ if and only if $X^*$ is hereditarily $\textsf{\emph{M}}_{\zeta, \omega_1}$. \item     If $\xi<\omega_1$, then $X$ is hereditarily $\textsf{\emph{M}}_{\gamma, \zeta}$ for some $\omega^\xi<\gamma<\omega^{\xi+1}$ if and only if $X$ is hereditarily $\textsf{\emph{M}}_{\gamma, \zeta}$ for every $\omega^\xi<\gamma<\omega^{\xi+1}$.   \item If $\zeta<\omega_1$, then $X$ is hereditarily $\textsf{\emph{M}}_{\xi, \gamma}$ for some $\omega^\zeta<\gamma<\omega^{\zeta+1}$ if and only if $X$ is hereditarily $\textsf{\emph{M}}_{\xi, \gamma}$ for every $\omega^\zeta<\gamma<\omega^{\zeta+1}$.  \end{enumerate}

\end{theorem}

We also study three space properties related to the $\xi$-weak Banach-Saks property, modifying a method of Ostrovskii \cite{Ostrovskii}. In \cite{Ostrovskii}, it was shown that the weak Banach-Saks property is not a three-space property. Our final theorem generalizes this. In our final theorem, $\textsf{wBS}_\xi$ denotes the class of Banach spaces $X$ such that $I_X\in \mathfrak{wBS}_\xi$. 

\begin{theorem} For $0\leqslant \zeta, \xi<\omega_1$, if $X$ is a Banach space and $Y$ is a closed subspace such that $Y\in \textsf{\emph{wBS}}_\zeta$ and $X/Y\in \textsf{\emph{wBS}}_\xi$, then $X\in \textsf{\emph{wBS}}_{\xi+\zeta}$.    

For every $0\leqslant \zeta, \xi<\omega_1$, there exists a Banach space $X$ with a closed subspace $Y$ such that $Y\in \textsf{\emph{wBS}}_\zeta$, $X/Y\in \textsf{\emph{wBS}}_\xi$, and for each $\gamma<\xi+\zeta$, $X$ fails to lie in $\textsf{\emph{wBS}}_\gamma$.

\end{theorem}

\section{Combinatorics}

\subsection{Regular families}

Througout, we let $2^\nn$ denote the power set of $\nn$ and topologize this set with the Cantor topology. Given a subset $M$ of $\nn$, we  let $[M]$ (resp. $[M]^{<\nn}$) denote set of infinite (resp. finite) subsets of $M$.  For convenience, we often write subsets of $\nn$ as sequences, where a set $E$ is identified with the (possibly empty) sequence obtained by listing the members of $E$ in strictly increasing order.  Henceforth, if we write $(m_i)_{i=1}^r\in [\nn]^{<\nn}$ (resp. $(m_i)_{i=1}^\infty\in [\nn]$), it will be assumed that $m_1<\ldots <m_r$ (resp. $m_1<m_2<\ldots$).  Given $M=(m_n)_{n=1}^\infty\in [\nn]$ and $\mathcal{F}\subset [\nn]^{<\nn}$, we define $$\mathcal{F}(M)=\{(m_n)_{n\in E}: E\in \mathcal{F}\}$$ and $$\mathcal{F}(M^{-1})=\{E: (m_n)_{n\in E}\in \mathcal{F}\}.$$  

Given $(m_i)_{i=1}^r, (n_i)_{i=1}^r\in [\nn]^{<\nn}$, we say $(n_i)_{i=1}^r$ is a \emph{spread} of $(m_i)_{i=1}^r$ if $m_i\leqslant n_i$ for each $1\leqslant i\leqslant r$.   We agree that $\varnothing$ is a spread of $\varnothing$.    We write $E\preceq F$ if either $E=\varnothing$ or $E=(m_i)_{i=1}^r$ and $F=(m_i)_{i=1}^s$ for some $r\leqslant s$.   In this case, we say $E$ is an \emph{initial segment} of $F$.  For $E,F\subset \nn$, we write $E<F$ to mean that either $E=\varnothing$, $F=\varnothing$, or $\max E<\min F$. Given $n\in \nn$ and $E\subset \nn$, we write $n\leqslant E$ (resp. $n<E$) to mean that $n\leqslant \min E$ (resp. $n<\min E$).

We say $\mathcal{G}\subset [\nn]^{<\nn}$ is \begin{enumerate}[(i)]\item \emph{compact} if it is compact in the Cantor topology, \item \emph{hereditary} if $E\subset F\in \mathcal{G}$ implies $E\in \mathcal{G}$,  \item \emph{spreading} if whenever $E\in \mathcal{G}$ and $F$ is a spread of $E$, $F\in \mathcal{G}$,  \item \emph{regular} if it is compact, hereditary, and spreading. \end{enumerate}

Let us also say that $\mathcal{G}$ is \emph{nice} if \begin{enumerate}[(i)]\item $\mathcal{G}$ is regular, \item $(1)\in \mathcal{G}$,  \item for any $\varnothing\neq E\in \mathcal{G}$, either $E\in MAX(\mathcal{G})$ or $E\cup (1+\max E)\in \mathcal{G}$.  \end{enumerate} Let us briefly explain why these last two properties are desirable. We wish to create norms on $c_{00}$ of the form $$\|\sum_{n=1}^\infty a_ne_n\|_\mathcal{F}=\sup\{\sum_{n\in F}|a_n|: F\in \mathcal{F}\}.$$  In order for this to be a norm and not just a seminorm, we require that $(1)\in \mathcal{F}$. The last condition is because we wish to have the property that any $M\in[\nn]$ can be uniquely decomposed into sets $F_1<F_2<\ldots$, where each $F_n\in MAX(\mathcal{F})$. If $\mathcal{F}$ is compact and $M\in[\nn]$, then there exists a largest (with respect to inclusion) $F$ which is an initial segment of $M$ and which lies in $\mathcal{F}$, but this $F$ need not be a maximal member of $\mathcal{F}$.  To see why, let $$\mathcal{F}=\{E\subset \nn: |E|\leqslant 2\}\setminus\{(1,2)\}.$$   This is compact, spreading, and hereditary, but the largest initial segment of the set $M=(1, 3,4,\ldots)$ which lies in $\mathcal{F}$ is $(1)$, which is not a maximal member of $\mathcal{F}$.

If $M\in[\nn]$ and if $\mathcal{F}$ is nice, then there exists a unique, finite, non-empty initial segment of $M$ which lies in $MAX(\mathcal{F})$.  We let $M_\mathcal{F}$ denote this initial segment.    We now define recursively $M_{\mathcal{F},1}=M_\mathcal{F}$ and $M_{\mathcal{F}, n+1}= (M\setminus \cup_{i=1}^n M_{\mathcal{F},i})_{\mathcal{F}}$.    An alternate description of $M_{\mathcal{F},1}, M_{\mathcal{F}, 2}, \ldots$ is that the sequence $M_{\mathcal{F},1}, M_{\mathcal{F},2}, \ldots$ is the unique partition of $M$ into successive sets which are maximal members of $\mathcal{F}$.

If $\mathcal{F}$ is nice and $M\in [\nn]$, then there exists a partition $E_1<E_2<\ldots$ of $\nn$ such that $M_{\mathcal{F},n}=(m_i)_{i\in E_n}$ for all $n\in\nn$.  We define $M^{-1}_{\mathcal{F},n}=E_n$.

Given a topological space $K$ and a subset $L$ of $K$, $L'$ denotes the \emph{Cantor Bendixson derivative} of $L$ consists of those members of $L$ which are not relatively isolated in $L$.    We define by transfinite induction the higher order transfinite derivatives of $L$ by $$L^0=L,$$ $$L^{\xi+1}=(L^\xi)',$$ and if $\xi$ is a limit ordinal, $$L^\xi=\bigcap_{\zeta<\xi}L^\zeta.$$ We recall that $K$ is said to be \emph{scattered} if there exists an ordinal $\xi$ such that $K^\xi=\varnothing$. In this case, we define the \emph{Cantor Bendixson index} of $K$ by $CB(K)=\min\{\xi: K^\xi=\varnothing\}$.    If $K^\xi\neq \varnothing$ for all ordinals $\xi$, we write $CB(K)=\infty$.   We agree to the convention that $\xi<\infty$ for all ordinals $\xi$, and therefore $CB(K)<\infty$ simply means that $CB(K)$ is an ordinal, and $K$ is scattered.

For each $n\in\nn\cup \{0\}$, we let $\mathcal{A}_n=\{E\in [\nn]^{<\nn}: |E|\leqslant n\}$.    It is clear that $\mathcal{A}_n$ is regular.    Also of importance are the Schreier families, $(\mathcal{S}_\xi)_{\xi<\omega_1}$.  We recall these families.  We let $$\mathcal{S}_0=\mathcal{A}_1,$$ $$\mathcal{S}_{\xi+1}=\{\varnothing\} \cup \Bigl\{\bigcup_{i=1}^n E_i: \varnothing\neq E_i \in \mathcal{S}_\xi, n\leqslant E_1, E_1<\ldots <E_n\Bigr\},$$ and if $\xi<\omega_1$ is a limit ordinal, there exists a sequence $\xi_n\uparrow \xi$ such that $$\mathcal{S}_\xi=\{E\in [\nn]^{<\nn}: \exists n\leqslant E\in \mathcal{S}_{\xi_n+1}\},$$  and $(\xi_n)_{n=1}^\infty$ has the property that for any $n\in\nn$, $\mathcal{S}_{\xi_n+1}\subset \mathcal{S}_{\xi_{n+1}}$.    The existence of such families with the last indicated property is discussed, for example, in \cite{Concerning}.  With the fact that $\mathcal{S}_{\xi_n+1}\subset \mathcal{S}_{\xi_{n+1}}\subset \mathcal{S}_{\xi_{n+1}+1}$, and equivalent, useful way of representing these sets is $$\mathcal{S}_\xi=\{\varnothing\}\cup \{E\in [\nn]^{<\nn}: \varnothing\neq E\in \mathcal{S}_{\xi_{\min E}+1}\}.$$  Sometimes for convenience, we simply represent $$\mathcal{S}_\xi=\{E\in [\nn]^{<\nn}: \exists n\leqslant E\in \mathcal{S}_{\zeta_n}\},$$ where $\zeta_n=\xi_n+1$.    In each instance, we use the notation which is most convenient. 

Given two non-empty regular families $\mathcal{F}, \mathcal{G}$, we let $$\mathcal{F}[\mathcal{G}]=\{\varnothing\}\cup \Bigl\{\bigcup_{i=1}^n E_i: \varnothing \neq E_i\in \mathcal{G}, E_1<\ldots <E_n, (\min E_i)_{i=1}^n\in \mathcal{F}\Bigr\}.$$  We let $\mathcal{F}[\mathcal{G}]=\varnothing$ if either $\mathcal{F}=\varnothing$ or $\mathcal{G}=\varnothing$. 

The following facts are collected in \cite{Concerning}.

\begin{proposition} \begin{enumerate}[(i)]\item For any non-empty regular families $\mathcal{F}, \mathcal{G}$, $\mathcal{F}[\mathcal{G}]$ is regular. Furthermore, if $CB(\mathcal{F})=\beta+1$ and $CB(\mathcal{G})=\alpha+1$, then $CB(\mathcal{F}[\mathcal{G}])=\alpha\beta+1$. \item For any $n\in\nn$, $CB(\mathcal{A}_n)=n+1$. \item For any $\xi<\omega_1$, $CB(\mathcal{S}_\xi)=\omega^\xi+1$. \item If $\mathcal{F}$ is regular and $M\in [\nn]$, then $\mathcal{F}(M^{-1})$ is regular and $CB(\mathcal{F})=CB(\mathcal{F}(M^{-1}))$. \item For regular families $\mathcal{F}, \mathcal{G}$, there exists $M\in [\nn]$ such that $\mathcal{F}(M)\subset \mathcal{G}$ if and only if there exists $M\in [\nn]$ such that $\mathcal{F}\subset \mathcal{G}(M^{-1})$ if and only if $CB(\mathcal{F})\leqslant CB(\mathcal{G})$. \item For $\xi\leqslant \zeta<\omega_1$, there exists $n\in\nn$ such that $n\leqslant E\in \mathcal{S}_\xi$ implies $E\in \mathcal{S}_\zeta$. \item For all $1\leqslant \xi<\omega_1$, $\mathcal{S}_1\subset \mathcal{S}_\xi$. \end{enumerate} 

\label{deep facts}

\end{proposition}

Item $(vi)$ is sometimes referred to as the \emph{almost monotone property}.

\begin{lemma} Fix a countable ordinal  $\gamma$. \begin{enumerate}[(i)]\item For any $L\in[\nn]$ and $\delta<\omega_1$, there exists $M\in[L]$ such that for all $(n_i)_{i=1}^\infty\in [M]$, $G\in \mathcal{S}_\delta$, and $E_1<E_2<\ldots$, $\varnothing\neq E_i\in \mathcal{S}_\gamma$, $$\bigcup_{i\in G} E_{n_i}\in \mathcal{S}_{\gamma+\delta}.$$  \item For any $L\in[\nn]$, there exists $M\in[L]$ such that for all $(n_i)_{i=1}^\infty\in [M]$ and any $E\in \mathcal{S}_{\gamma+\delta}$, there exist $E_1<\ldots <E_d$, $\varnothing\neq E_i\in \mathcal{S}_\gamma$, such that $(n_{\min E_i})_{i=1}^d\in \mathcal{S}_\delta$ and $E=\cup_{i=1}^d E_i$. \end{enumerate}

\label{tech}
\end{lemma}

\begin{rem}\upshape Both parts of Lemma \ref{tech} are strengthenings of  Proposition \ref{deep facts}.

\end{rem}

\begin{proof} For both $(i)$ and $(ii)$, we induct on $\delta$. 

$(i)$ For $\delta=0$, we can simply take $M=L$.  Now suppose that the result holds for $\delta$ and $L\in[\nn]$ is fixed. By the inductive hypothesis, there exists $M\in[L]$ such that for any $(n_i)_{i=1}^\infty\in [M]$, $E_1<E_2<\ldots$, $\varnothing\neq E_i\in \mathcal{S}_\gamma$, and $G\in \mathcal{S}_\delta$, $\cup_{i\in E} E_{n_i}\in \mathcal{S}_{\gamma+\delta}$.    Now fix  $(n_i)_{i=1}^\infty \in [M]$, $E_1<E_2<\ldots$, $\varnothing\neq E_i\in \mathcal{S}_\gamma$, and $\varnothing \neq G\in \mathcal{S}_{\gamma+1}$.   Let $k=\min G$ and note that we may write $G=\cup_{i=1}^d G_i$ for some $G_1<\ldots <G_d$, $\varnothing\neq G_i\in \mathcal{S}_\delta$, nd $d\leqslant k$.    By the choice of $M$, for each $1\leqslant j\leqslant d$, $F_j:=\cup_{i\in G_j} E_{n_i}\in \mathcal{S}_{\gamma+\delta}$.        Since $F_1<\ldots <F_d$ and $\min F_1 = \min E_k \geqslant k\geqslant d$, $$\bigcup_{i\in G} E_{n_i}=\bigcup_{j=1}^d F_j\in \mathcal{S}_{\gamma+\delta+1}.$$  

Now suppose that $\delta<\omega_1$ is a limit ordinal. Let $(\delta_n)_{n=1}^\infty$, $(\beta_n)_{n=1}^\infty$ be the sequences such that $$\mathcal{S}_{\gamma+\delta}=\{\varnothing\}\cup \{E: \varnothing\neq E \in \mathcal{S}_{\beta_{\min E}}\}$$ and $$\mathcal{S}_\delta=\{\varnothing\}\cup \{E: \varnothing\neq E\in \mathcal{S}_{\delta_n}\}.$$     Now let us choose natural numbers $p_1<p_2<\ldots$ and $q_1<q_2<\ldots$ such that $$\gamma+\delta_n<\beta_{p_n}$$ and if $q_n\leqslant E\in \mathcal{S}_{\gamma+\delta_n}$, $E\in \mathcal{S}_{\beta_{p_n}}$.  By the inductive hypothesis, we may fix $$M_0:=L\supset M_1\supset M_2\supset \ldots,$$ $M_n\in [\nn]$, such that for each $n\in \nn$, each $(n_i)_{i=1}^\infty \in [M_n]$, each $E_1<E_2<\ldots$ with $\varnothing \neq E_i\in \mathcal{S}_\gamma$, and each $G\in \mathcal{S}_{\delta_n}$, $\cup_{i\in G}E_{n_i}\in \mathcal{S}_{\gamma+\delta_n}$.   Since each $M_n$ may be taken to lie in any infinite subset of $M_{n-1}$, we may also assume that $\min M_n \geqslant \max\{p_n, q_n\}$ for all $n\in\nn$.       Now write $M_n=(m^n_i)_{i=1}^\infty$ and let $m_n=m^n_n$. Note that $m_1<m_2<\ldots$. Let $M=(m_i)_{i=1}^\infty$.     Fix $(n_i)_{i=1}^\infty\in [M]$, $E_1<E_2<\ldots$ with $\varnothing\neq E\in \mathcal{S}_\gamma$, and $\varnothing\neq G\in \mathcal{S}_\delta$.      Let $k=\min G$ and note that $G\in \mathcal{S}_{\delta_k}$. Let $$S=(m_1^k, m_2^k, \ldots, m^k_{k-1}, n_k, n_{k+1}, n_{k+2}, \ldots) \in [M_k].$$   Write $S=(s_i)_{i=1}^\infty$ and note that since $s_i=n_i$ for all $i\geqslant k$, $H:=\cup_{i\in G}E_{n_i}=\cup_{i\in G}E_{s_i}$.       Since $G\in \mathcal{S}_{\delta_k}$ and $S\in [M_k]$, $H\in \mathcal{S}_{\gamma+\delta_k}$.    Note that $$\min H \geqslant n_k\geqslant \min M_k \geqslant\max\{p_k, q_k\}.$$  Since $q_k\leqslant H \in \mathcal{S}_{\gamma+\delta_k}$, $H\in \mathcal{S}_{\beta_{p_k}}$.   Since $p_k\leqslant H\in \mathcal{S}_{\beta_{p_k}}$, $H\in \mathcal{S}_{\gamma+\delta}$.

$(ii)$ Note that if $M=(m_i)_{i=1}^\infty$ and $N=(n_i)_{i=1}^\infty\in [M]$, then for any $\varnothing\neq E\in [\nn]^{<\nn}$, $(n_i)_{i\in E}$ is a spread of $(m_i)_{i\in E}$.     Thus if we verify the conclusion when $(n_i)_{i=1}^\infty =M$, this implies the result for all $(n_i)_{i=1}^\infty \in [M]$.

For $\delta=0$, we may simply take $M=L$. Suppose the result holds for $\delta$ and fix $L\in[\nn]$. Choose $M=(m_i)_{i=1}^\infty\in [L]$ such that for any $E\in \mathcal{S}_{\gamma+\delta}$, there exist $F_1<\ldots <F_d$ such that  $E=\cup_{i=1}^d F_i$, $\varnothing\neq F_i\in \mathcal{S}_\gamma$ such that $(m_{\min F_i})_{i=1}^d\in \mathcal{S}_\delta$.  Now fix $E\in \mathcal{S}_{\gamma+\delta+1}$ and let $k=\min E$. Write $E=\cup_{j=1}^l E_j$, $E_1<\ldots<E_l$, $\varnothing \neq E_i\in \mathcal{S}_\gamma$, and $l\leqslant k$.     We may recursively select $F_1<\ldots <F_n$, $\varnothing\neq F_i\in \mathcal{S}_\gamma$ and $0=d_0<\ldots <d_l=n$ such that for each $1\leqslant i\leqslant l$, $E_i=\cup_{j=d_{i-1}+1}^{d_i} F_j$ and $H_i:=(m_{\min F_j})_{j=d_{i-1}+1}^{d_i}\in \mathcal{S}_\delta$.    Note that $\min H_1\geqslant \min F_1=\min E=k\geqslant l$.  Therefore $E=\cup_{i=1}^l E_i=\cup_{j=1}^n F_j$ and $$(m_{\min F_j})_{j=1}^n = \bigcup_{i=1}^l (m_{\min F_j})_{j=d_{i-1}+1}^{d_i}=\bigcup_{i=1}^l H_i \in \mathcal{S}_{\delta+1}.$$

Last, let $\delta<\omega_1$ be a limit ordinal. Let $(\delta_n)_{n=1}^\infty$, $(\beta_n)_{n=1}^\infty$ be the sequences such that $$\mathcal{S}_{\gamma+\delta}=\{\varnothing\}\cup \{E: \varnothing\neq E \in \mathcal{S}_{\beta_{\min E}}\}$$ and $$\mathcal{S}_\delta=\{\varnothing\}\cup \{E: \varnothing\neq E\in \mathcal{S}_{\delta_{\min E}+1}\},$$ and recall that $\mathcal{S}_{\delta_n+1}\subset \mathcal{S}_{\delta_{n+1}}$ for all $n\in\nn$.     Choose natural numbers $p_1<p_2<\ldots$, $q_1<q_2<\ldots$ such that for all $n\in\nn$, $\beta_n\leqslant \gamma+\delta_{p_n}$ and $q_n\leqslant E\in \mathcal{S}_{\beta_n}$ implies $E\in \mathcal{S}_{\gamma+\delta_{p_n}}$.       Recursively select $$M_0=L\supset M_1\supset M_2\supset \ldots$$ such that $\min M_n \geqslant \max \{p_n, q_n\}$ and, with $M_n=(m^n_i)_{i=1}^\infty$, if $E\in \mathcal{S}_{\gamma+\delta_{p_n}}$, there exist $F_1<\ldots <F_d$ such that $\varnothing\neq F_i\in \mathcal{S}_\gamma$, $E=\cup_{i=1}^d E_i$, and $(m^n_{\min E_i})_{i=1}^d\in \mathcal{S}_{\delta_{r_n}}$.  Let $m_n=m^n_n$.    Now fix $\varnothing\neq E\in \mathcal{S}_{\gamma+\delta}$ and let $k=\min E$. If $k=1$, then $E=(1)$, and we may write $E=E_1$, $E_1=(1)\in \mathcal{S}_\gamma$, $(m_{\min E_1})\in \mathcal{S}_\delta$. Assume $1<k$.     Then $E\in \mathcal{S}_{\beta_k}$, and $E\cap [q_k, \infty)\in \mathcal{S}_{\gamma+\delta_{p_k}}$.   Let us choose $F_1<F_2<\ldots <F_d$, $\varnothing\neq F_i\in \mathcal{S}_\gamma$ such that $E\cap [q_k, \infty)=\cup_{i=1}^d F_i$ and $J:=(m_{\min F_i}^k)_{i=1}^d\in \mathcal{S}_{\delta_{p_k}}$.   Since $\min F_1\geqslant k$, $p_k\leqslant m_k\leqslant J\in \mathcal{S}_{\delta_{p_k}}\subset \mathcal{S}_{\delta_{p_k}+1}$, $J\in \mathcal{S}_\delta$.   Then since $H:=(m_{\min F_i})_{i=1}^d$ is a spread of $J$, $H\in \mathcal{S}_{\delta_{p_k}}\cap \mathcal{S}_\delta$.     If $E\cap [q_k, \infty)=E$, this is the desired conclusion.  Otherwise enumerate $E\cap (1, q_k)=(b_1, \ldots, b_t)$ and let $G_i=\{b_i\}$ for each $1\leqslant i\leqslant t$.    Note that $G_1<\ldots <G_t<F_1<\ldots <F_d$, $E=\bigl(\cup_{i=1}^t G_i\bigr)\cup \bigl(\cup_{i=1}^d F_i\bigr)$, and $\varnothing\neq G_i, F_i\in \mathcal{S}_\gamma$.   Let $G=(m_{\min G_i})_{i=1}^t$ and note that $m_k\leqslant G$ and $|G|\leqslant q_k\leqslant m_k$, so $G\in \mathcal{S}_1\subset \mathcal{S}_{\delta_{p_k}}$.  Since $2\leqslant G<H$ and $G,H\in \mathcal{S}_{\delta_{p_k}}$, $G\cup H \in \mathcal{S}_{\delta_{p_k}+1}$. Since $p_k\leqslant m_k\leqslant G$, $$(m_{\min G_i})_{i=1}^t \cup (m_{\min F_i})_{i=1}^d = G\cup H\in \mathcal{S}_\delta.$$

\end{proof}

\subsection{$\ell_1^\xi$ and $c_0^\xi$-spreading models}

Given a regular family $\mathcal{F}$, a Banach space $X$, and a seminormalized sequence $(x_n)_{n=1}^\infty\subset X$, we say $(x_n)_{n=1}^\infty$ is an $\ell_1^\mathcal{F}$-\emph{spreading model} provided that $$0<\inf\{\|x\|: F\in \mathcal{F}, x\in \text{abs\ co}(x_n: n\in F)\}.$$  Here, $$\text{abs\ co}(x_n: n\in F)=\Bigl\{\sum_{n\in F}a_n x_n: \sum_{n\in F}|a_n|=1\Bigr\}.$$   We say that a sequence  $(x_n)_{n=1}^\infty$ is a $c_0^\mathcal{F}$-\emph{spreading model} provided that $$0<\inf \{\|\sum_{n\in F}\ee_n x_n\|: F\in \mathcal{F}, \max_{n\in F}|\ee_n|=1\} \leqslant \sup \{\|\sum_{n\in F}\ee_n x_n\|: F\in \mathcal{F}, \max_{n\in F}|\ee_n|= 1\}<\infty.$$    If $\mathcal{F}=\mathcal{S}_\xi$, we write $\ell_1^\xi$ or $c_0^\xi$-spreading model in place of $\ell_1^{\mathcal{S}_\xi}$ or $c_0^{\mathcal{S}_\xi}$.  Note that a weakly null $\ell_1^0$ or $c_0^0$-spreading model is simply a seminormalized, weakly null sequence.   

Note that for a regular family $\mathcal{F}$, the spreading property of $\mathcal{F}$ yields that for any $k_1<k_2<\ldots$, $$\inf\{\|x\|: F\in \mathcal{F}, x\in \text{abs\ co}(x_{k_n}: n\in F)\} \geqslant \inf\{\|x\|: F\in \mathcal{F}, x\in \text{abs\ co}(x_n: n\in F)\},$$ so that any subsequence of an $\ell_1^\mathcal{F}$-spreading model is also an $\ell_1^\mathcal{F}$-spreading model.    Similarly, every subsequence of a $c_0^\mathcal{F}$-spreading model is also a $c_0^\mathcal{F}$-spreading model.

For $\xi<\omega_1$, we say a weakly null sequence is $\xi$-\emph{weakly null} if it has no subsequence which is an $\ell_1^\xi$-spreading model.   From this definition together with the convex unconditionality theorem of \cite{AMT}, it follows that if $(x_n)_{n=1}^\infty$ is a $\xi$-weakly null sequence in the Banach space $X$, then there exist sets $F_1<F_2<\ldots$, $F_n\in \mathcal{S}_\xi$, and positive scalars $(a_i)_{i\in \cup_{n=1}^\infty F_n}$ such that for each $n\in\nn$, $\sum_{i\in F_n}a_i=1$, and such that $\lim_n \|\sum_{i\in F_n}a_ix_i\|=0$. We will use this fact often.    However, we will also often need a technical fact which states that the coefficients $(a_i)_{i\in F_n}$ can come from the repeated averages hierarchy. We make this precise below.

Let $\mathscr{P}$ denote the set of all probability measures on $\nn$.    We treat each member $\mathbb{P}$ of $\mathscr{P}$ as a function from $\nn$ into $[0,1]$, where $\mathbb{P}(n)=\mathbb{P}(\{n\})$.  We let $\text{supp}(\mathbb{P})=\{n\in \nn: \mathbb{P}(n)>0\}$.   Given a nice family $\mathcal{P}$ and a subset $\mathfrak{P}=\{\mathbb{P}_{M,n}: M\in [\nn], n\in \nn\}$ of $ \mathscr{P}$, we say $(\mathfrak{P}, \mathcal{P})$ is a \emph{probability block} provided that \begin{enumerate}[(i)]\item for each $M\in [\nn]$, $\text{supp}(\mathbb{P}_{M,1})=M_{\mathcal{P}, 1}$, and \item for any $M\in [\nn]$ and $r\in \nn$, if $N=M\setminus \cup_{i=1}^{r-1}\text{supp}(\mathbb{P}_{M,i})$, then $\mathbb{P}_{N,1}=\mathbb{P}_{M,r}$.     \end{enumerate}

\begin{rem}\upshape It follows from the definition of probability block that for any $M\in [\nn]$, $(M_{\mathcal{P}, n})_{n=1}^\infty=(\text{supp}(\mathbb{P}_{M,n}))_{n=1}^\infty$ and for any $s\in \nn$ and $M,N\in \nn$, and $r_1<\ldots <r_s$ such that $\cup_{i=1}^s \text{supp}(\mathbb{P}_{M, r_i})$ is an initial segment of $N$, then $\mathbb{P}_{N, i}=\mathbb{P}_{M, r_i}$ for all $1\leqslant i\leqslant s$.   This was proved in \cite{CN}. 

\label{permanence}

\end{rem}

Suppose that $\mathcal{Q}$ is nice.    Given $L=(l_n)_{n=1}^\infty\in[\nn]$, there exists a unique sequence $0=p_0<p_1<\ldots$ such that $(l_i)_{i=p_{n-1}+1}^{p_n}\in MAX(\mathcal{Q})$ for all $n\in \nn$. We then define $L^{-1}_{\mathcal{Q}, n}=\nn\cap (p_{n-1}, p_n]$.   

Suppose we have probability blocks $(\mathfrak{P}, \mathcal{P})$, $(\mathfrak{Q}, \mathcal{Q})$.  We define a collection $\mathfrak{Q}*\mathfrak{P}$ such that $(\mathfrak{Q}*\mathfrak{P}, \mathcal{Q}[\mathcal{P}])$ is a probability block.     Fix $M\in \nn$ and for each $n\in\nn$,  let $l_n=\min \text{supp}(\mathbb{P}_{M,n})$ and $L=(l_n)_{n=1}^\infty$.    We then let $$\mathbb{O}_{M,n}= \sum_{i\in L_{\mathcal{Q},n}^{-1}} \mathbb{Q}_{L,n}(l_i) \mathbb{P}_{M,i}$$   and $\mathfrak{Q}*\mathfrak{P}=\{\mathbb{O}_{M,n}: M\in [\nn], n\in \nn\}$.   

In \cite{AMT}, the \emph{repeated averages hierarchy} was defined. This is a collection $\mathfrak{S}_\xi$, $\xi<\omega_1$, such that $(\mathfrak{S}_\xi, \mathcal{S}_\xi)$ is a probability block for every $\xi<\omega_1$. We will denote the members of $\mathfrak{S}_\xi$ by $\mathbb{S}^\xi_{M,n}$, $M\in [\nn]$, $n\in\nn$.

For $\xi<\omega_1$, we say a probability block $(\mathfrak{P}, \mathcal{P})$ is $\xi$-\emph{sufficient} provided that for any $L\in[\nn]$, any $\ee>0$, and any regular family $\mathcal{G}$ with $CB(\mathcal{G})\leqslant\omega^\xi$, there exists $M\in[\nn]$ such that $$\sup \{\mathbb{P}_{N,1}(E): E\in \mathcal{G}, N\in [M]\}<\ee.$$  It was shown in \cite{AMT} that $(\mathfrak{S}_\xi, \mathcal{S}_\xi)$ is $\xi$-sufficient.

 The following facts were shown in \cite{CN}. Item $(ii)$ was shown in \cite{AMT} in the particular case that $(\mathfrak{P}, \mathcal{P})=(\mathfrak{S}_\xi, \mathcal{S}_\xi)$.

\begin{theorem} \begin{enumerate}[(i)]\item For $\xi, \zeta<\omega_1$, if $(\mathfrak{P}, \mathcal{P})$ is $\xi$-sufficient and $(\mathfrak{Q}, \mathcal{Q})$ is $\zeta$-sufficient, then $(\mathfrak{Q}*\mathfrak{P}, \mathcal{Q}[\mathcal{P}])$ is $\xi+\zeta$-sufficient. \item If $X$ is a Banach space, $\xi<\omega_1$, $(\mathfrak{P}, \mathcal{P})$ is $\xi$-sufficient, and $CB(\mathcal{P})=\omega^\xi+1$, then a weakly null sequence $(x_n)_{n=1}^\infty \subset X$ is $\xi$-weakly null if and only if for any $L\in [\nn]$ and $\ee>0$, there exists $M\in[L]$ such that for all $N\in [M]$, $\|\sum_{i=1}^\infty \mathbb{P}_{N,1}(i)x_i\|<\ee$. \end{enumerate}

\label{7surv}
\end{theorem}

\begin{rem}\upshape
Since for each $\xi<\omega_1$, at least one $\xi$-sufficient probability block $(\mathfrak{P}, \mathcal{P})$ with $CB(\mathcal{P})=\omega^\xi+1$ exists, item $(ii)$ of the preceding theorem yields that if $X$ is a Banach space and $(x_n)_{n=1}^\infty$, $(y_n)_{n=1}^\infty$ are $\xi$-weakly null in $X$, then $(x_n+y_n)_{n=1}^\infty$ is also $\xi$-weakly null.  This generalizes to sums of any number of sequences.  The importance of this fact, which we will use often throughout, is that if for $k=1, \ldots, l$, if $(x^k_n)_{n=1}^\infty\subset X$ is a $\xi$-weakly null sequence, then for any $\ee>0$, there exist $F\in \mathcal{S}_\xi$ and positive scalars $(a_i)_{i\in F}$ such that $\sum_{i\in F}a_i=1$ and for each $1\leqslant k\leqslant l$, $$\|\sum_{i\in F}a_i x_i^k\|\leqslant \ee.$$  That is, there one choice of $F$ and $(a_i)_{i\in F}$ such that the corresponding linear combinations of the $l$ different sequences are simultaneously small.    

Note that the  preceding implies that for two Banach spaces $X,Y$ and $\xi$-weakly null sequences $(x_n)_{n=1}^\infty\subset X$, $(y_n)_{n=1}^\infty\subset Y$, for any $\ee>0$, there exist $F\in \mathcal{S}_\xi$ and positive scalars $(a_i)_{i\in F}$ summing to $1$ such that $$\|\sum_{i\in F}a_i x_i\|_X, \|\sum_{i\in F}a_iy_i\|_Y<\ee.$$  This is because the sequences $(x_n, 0)_{n=1}^\infty \subset X\oplus_\infty Y$ and $(0, y_n)_{n=1}^\infty \subset X\oplus_\infty Y$ are also $\xi$-weakly null, as is their sum in $X\oplus_\infty Y$.

\label{simul}
\end{rem}

\begin{rem}\upshape Let $X$ be a Banach space and let $(x_n)_{n=1}^\infty$ be $\xi$-weakly null.  Let $(\mathfrak{P}, \mathcal{P})$ be $\xi$-sufficient with $CB(\mathcal{P})=\omega^\xi+1$.  Then by Theorem \ref{7surv}$(ii)$, we may recursively select $M_1\supset M_2\supset \ldots$, $M_n\in[\nn]$ such that for each $n\in\nn$, $$\sup\{\|\sum_{i=1}^\infty \mathbb{P}_{N,1}(i)x_i\|: N\in [M_n]\}<1/n.$$   Now choose $m_n\in M_n$ with $m_1<m_2<\ldots$ and let $M=(m_n)_{n=1}^\infty$.  Then for any $N\in [M]$ and $n\in\nn$, if $F_1<F_2<\ldots$ is a partition of $N$ into consecutive, maximal members of $\mathcal{P}$ and $N_j=N\setminus \cup_{i=1}^{j-1}F_i$ for each $j\in\nn$, $N_n\in [M_n]$.  By the permanence property mentioned in Remark \ref{permanence}, $$\|\sum_{i=1}^\infty \mathbb{P}_{N,n}(i)x_i\|=\|\sum_{i=1}^\infty \mathbb{P}_{N_n, 1}(i)x_i\|<1/n.$$

\label{6surv}
\end{rem}

Before proceeding to the following, we recall that for $M\in[\nn]$ and a regular family $\mathcal{F}$, we let $M|_\mathcal{F}$ denote the maximal initial segment of $M$ which lies in $\mathcal{F}$. If $\mathcal{F}$ is nice, then $M|_\mathcal{F}$ lies in $MAX(\mathcal{F})$. 

\begin{lemma} Let $X$ be a Banach space,  $(x_n)_{n=1}^\infty\subset X$ a seminormalized, weakly null sequence, and $\mathcal{F}$ a nice family.  \begin{enumerate}[(i)]\item $(x_n)_{n=1}^\infty$ admits a subsequence which is a $c_0^\mathcal{F}$-spreading model if and only if there exists $L\in [\nn]$ such that $$\sup \{\|\sum_{n\in M|_{\mathcal{F}}}x_n\|: M\in [L]\}<\infty.$$ \item If $(x_n)_{n=1}^\infty$ admits no subsequence which is a $c_0^\mathcal{F}$-spreading model, then there exists $L\in[\nn]$ such that for any $H_1<H_2<\ldots$, $H_n\in MAX(\mathcal{F})\cap [L]^{<\nn}$, $\|\sum_{i\in H_n}x_i\|>n$ for each $n\in\nn$.  \end{enumerate}

\label{c0}

\end{lemma}

\begin{proof}$(i)$ Assume there exists $L\in[\nn]$ such that $$\sup \{\|\sum_{n\in M|_\mathcal{F}}x_n\|: M\in [L]\}=C<\infty.$$   By passing to a subsequence, we may assume $(x_n)_{n\in L}$ is $2$-basic.   If $F\in \mathcal{F}\cap [L]^{<\nn}$, there exists an infinite subset $M$ of $L$ such that $F$ is an initial segment of $M|_\mathcal{F}$, whence $$\|\sum_{n\in F} x_n\|\leqslant 2\|\sum_{n\in M|_\mathcal{F}}x_n\|\leqslant 2C.$$   Thus $$\sup\{\|\sum_{n\in F}x_n\|: F\in \mathcal{F}\cap [L]^{<\nn}\}\leqslant 2C.$$  Then if $\varnothing\neq F\in \mathcal{F}\cap [L]^{<\nn}$, $(a_n)_{n\in F} \in [-1,1]^F$, $$\sum_{n\in F}a_nx_n\in \text{co}\Bigl(\sum_{n\in G}x_n: G\subset F\Bigr)-\text{co}\Bigl(\sum_{n\in G}x_n:G\subset F\Bigr) \subset 4CB_X.$$   Now for any $\varnothing\neq F\in \mathcal{F}\cap [L]^{<\nn}$ and for any scalars $(a_n)_{n\in F}$ with $|a_n|\leqslant 1$, $$\|\sum_{n\in F}a_nx_n\|\leqslant \|\sum_{n\in F}\text{Re\ }(a_n)x_n\|+\|\sum_{n\in F}\text{Im\ }(a_n)x_n\|\leqslant 8C.$$   If $L=(l_n)_{n=1}^\infty$, this yields the appropriate upper estimates to deduce that  $(x_{l_n})_{n=1}^\infty$ is a $c_0^\mathcal{F}$-spreading model.    The lower estimates follow from the fact that $(x_{l_n})_{n=1}^\infty$ is seminormalized basic.

For the converse, suppose that $(x_{r_n})_{n=1}^\infty$ is a $c_0^\mathcal{F}$-spreading model and let $$c=\sup\{\|\sum_{n\in F}x_{r_n}\|: F\in \mathcal{F}\}<\infty.$$    Let us choose $1=s_1<s_2<\ldots$ such that  $s_{n+1}>r_{s_n}$ for all $n\in\nn$. Let $l_n=r_{s_n}$, $L=(l_n)_{n=1}^\infty$, and $S=(s_n)_{n=1}^\infty$.    Fix $M\in[L]$ and note that $M=(r_{t_n})_{n=1}^\infty$ for some $(t_n)_{n=1}^\infty \in [S]$.  Let $M|_\mathcal{F}=(r_{t_n})_{n=1}^k\in \mathcal{F}$ and note that $(t_n)_{n=2}^k\in \mathcal{F}$. Indeed, if $t_{n-1}=s_i$ and $t_n=s_j$, $i<j$, then $$t_n=s_j\geqslant s_{i+1}>r_{s_i}=r_{t_{n-1}}.$$    Thus $E:=(t_n)_{n=2}^k$ is a spread of $(r_{t_n})_{n=1}^{k-1}\subset (r_{t_n})_{n=1}^k\in \mathcal{F}$, and $E\in \mathcal{F}$.    Therefore, with $b=\sup_n \|x_n\|$, $$\|\sum_{n\in M|_\mathcal{F}} x_n\| \leqslant \|x_{r_{t_1}}\|+\|\sum_{n=2}^k x_{r_{t_n}}\|=\|x_{r_{t_1}}\|+\|\sum_{n\in E}x_{r_{t_n}}\|\leqslant b+c=:C.$$    Therefore we have shown that $$\sup\{\|\sum_{n\in M|_\mathcal{F}}x_n\|: M\in [L]\}\leqslant C.$$

$(ii)$ For each $n\in\nn$, let $$\mathcal{V}_n=\{M\in[\nn]: \|\sum_{i\in M|_\mathcal{F}}x_i\|\leqslant n\|\}.$$  It is evident that $\mathcal{V}_n$ is closed, and in fact $M\mapsto \|\sum_{i\in M|_\mathcal{F}}x_i\|$ is locally constant on $[\nn]$.    By the Ramsey theorem, we may select $M_1\supset M_2\supset \ldots$ such that for all $n\in\nn$, either $[M_n]\subset \mathcal{V}_n$ or $\mathcal{V}_n\cap [M_n]=\varnothing$.   By $(i)$ and the hypothesis that $(x_n)_{n=1}^\infty$ admits no subsequence which is a $c_0^\mathcal{F}$-spreading model, for each $n\in\nn$, $\mathcal{V}_n\cap [M_n]=\varnothing$.    Now fix $l_1<l_2<\ldots$, $l_n\in M_n$,  and let $L=(l_n)_{n=1}^\infty$.    Fix $\varnothing\neq H_1<H_2<\ldots$.  For each $n\in\nn$, let  $N_n=\cup_{i=n}^\infty H_i\in [M_n]$ and note that $H_n=N_n|_\mathcal{F}$. Since $N_n\in [M_n]\subset [\nn]\setminus \mathcal{V}_n$, $$\|\sum_{i\in H_n}x_i\|=\|\sum_{i\in N_n|_\mathcal{F}}x_i\|>n.$$

\end{proof}

For ordinals $\xi, \zeta<\omega_1$ and any $M\in[\nn]$, there exists $N\in [M]$ such that $\mathcal{S}_\xi[\mathcal{S}_\zeta](N)\subset \mathcal{S}_{\zeta+\xi}$ and $\mathcal{S}_{\zeta+\xi}(N)\subset \mathcal{S}_\xi[\mathcal{S}_\zeta]$. From this it follows that for a given sequence $(x_n)_{n=1}^\infty$ in a Banach space $X$, there exist  $m_1<m_2<\ldots$ such that $$0< \inf\{\|x\|: F\in \mathcal{S}_{\zeta+\xi}, x\in \text{abs\ co}(x_{m_n}:n\in F)\}$$ if and only if there exist $m_1<m_2<\ldots$ such that $$0<\inf\{\|x\|: F\in \mathcal{S}_\xi[\mathcal{S}_\zeta], x\in \text{abs\ co}(x_{m_n}: n\in F)\}.$$   This fact will be used throughout to deduce that if $(x_n)_{n=1}^\infty$ is an $\ell_1^{\zeta+\xi}$-spreading model (or has a subsequence which is an $\ell_1^{\zeta+\xi}$-spreading model), then there exists a subsequence of $(x_n)_{n=1}^\infty$ which is an $\ell_1^{\mathcal{S}_\xi[\mathcal{S}_\zeta]}$-spreading model.  Similarly, if $(x_n)_{n=1}^\infty$ has a subsequence which is a $c_0^{\zeta+\xi}$-spreading model, then it has a subsequence which is a $c_0^{\mathcal{S}_\xi[\mathcal{S}_\zeta]}$-spreading model.

\begin{corollary} Fix $\alpha, \beta, \gamma<\omega_1$. Let $X,Y$ be  Banach spaces, $A:X\to Y$ an operator,  and let $(x_n)_{n=1}^\infty$ be a seminormalized, weakly null sequence in $X$.     \begin{enumerate}[(i)]\item If $(Ax_n)_{n=1}^\infty$ has a subsequence which is an $\ell_1^{\alpha+\beta}$-spreading model and $(x_n)_{n=1}^\infty$ has no subsequence which is an $\ell_1^{\alpha+\gamma}$-spreading model, then there exists a convex block sequence $(z_n)_{n=1}^\infty$ of $(x_n)_{n=1}^\infty$ which has no subsequence which is an $\ell_1^\gamma$-spreading model and such that $(Az_n)_{n=1}^\infty$ is an $\ell_1^\beta$-spreading model. \item If $(x_n)_{n=1}^\infty$ has a subsequence which is a $c_0^{\alpha+\beta}$-spreading model but no subsequence which is a $c_0^{\alpha+\gamma}$-spreading model, then there exists a  block sequence of $(x_n)_{n=1}^\infty$ which is a $c_0^\beta$-spreading model and has no subsequence which is a $c_0^\gamma$-spreading model. If $0<\beta$, the block sequence is also weakly null.   \end{enumerate}

\label{block of block}
\end{corollary}

\begin{proof}$(i)$ We first assume $\sup_n \|x_n\|=1$.  By passing to a subsequence, we may assume without loss of generalilty that $$0<\ee=\inf\{\|Ax\|: F\in \mathcal{S}_\beta[\mathcal{S}_\alpha], x\in \text{abs\ co}(x_n: n\in F)\}.$$  Let $\mathcal{P}=\mathcal{S}_\gamma[\mathcal{S}_\alpha]$, $\mathfrak{P}=\mathfrak{S}_\gamma*\mathfrak{S}_\alpha=\{\mathbb{P}_{M,n}:M\in[\nn], n\in\nn\}$.     As mentioned in Remark \ref{6surv}, we may also fix $L\in[\nn]$ such that for all $M\in[L]$ and $n\in\nn$, $$\|\sum_{i=1}^\infty \mathbb{P}_{M,n}(i)x_i\|\leqslant 1/n.$$   Now fix $F_1<F_2<\ldots$, $F_n\in MAX(\mathcal{S}_\alpha)$, $L=\cup_{n=1}^\infty F_n$ and let $y_n=\sum_{i=1}^\infty \mathbb{S}^\alpha_{L,n}(i)x_i=\sum_{i\in F_n} \mathbb{S}^\alpha_{L,n}(i)x_i$.   It follows from the second sentence that $$\ee\leqslant \inf\{\|Ay\|: F\in \mathcal{S}_\beta, y\in \text{abs\ co}(y_n: n\in F)\}.$$   That is, $(Ay_n)_{n=1}^\infty$ is an $\ell_1^\beta$-spreading model.  It remains to show that $(y_n)_{n=1}^\infty$ has no subsequence which is an $\ell_1^\gamma$-spreading model. To that end, assume $R=(r_n)_{n=1}^\infty$, $\delta>0$ are such that $$\delta\leqslant \inf\{\|\sum_{n\in F}a_ny_{r_n}\|: F\in \mathcal{S}_\gamma, \sum_{n\in F}|a_n|=1\}.$$   Now let $E_n=F_{r_n}$, $N=\cup_{n=1}^\infty E_n$, $S=(s_n)_{n=1}^\infty=(\min E_n)_{n=1}^\infty$ and note that $$z_n:=y_{r_n}=\sum_{i=1}^\infty \mathbb{S}^\alpha_{N,n}(i)x_i$$ for all $n\in\nn$.     Now fix $1=q_1<q_2<\ldots$ such that $q_{n+1}>s_{q_n}$.  Let $M=\cup_{n=1}^\infty E_{q_n}$ and note that there exist $0=k_0<k_1<\ldots$ such that for all $n\in\nn$, $$\mathbb{P}_{M,n}=\sum_{j=k_{n-1}+1}^{k_n} \mathbb{S}^\gamma_{T,n}(s_{q_j})\mathbb{S}^\alpha_{M,j}$$ and $(s_{q_j})_{j=k_{n-1}+1}^{k_n}\in \mathcal{S}_\gamma$, where $T=(s_{q_j})_{j=1}^\infty$.   Moreover, since $0<\gamma$, $\mathbb{S}^\gamma_{T,n}(s_{q_{k_{n-1}+1}})\to 0$.  We now observe that since $s_{q_j}<q_{j+1}$, $G_n:=(q_j)_{j=k_{n-1}+2}^{k_n}$ is a spread of $(q_j)_{j=k_{n-1}+1}^{k_n-1}$, which is a subset of a member of $\mathcal{S}_\gamma$.    Therefore, for any $n\in\nn$, \begin{align*} \delta(1-\mathbb{S}^\gamma_{T,n}(s_{q_{k_{n-1}+1}})) & \leqslant \|\sum_{j=k_{n-1}+2}^{k_n} \mathbb{S}^\gamma_{T,n}(s_{q_j})z_j\| \leqslant \|\sum_{j=k_{n-1}+1}^{k_n}\mathbb{S}^\gamma_{T,n}(s_{q_j})z_j\|-\mathbb{S}^\gamma_{T,n}(s_{q_{k_{n-1}+1}}) \\ & \|\sum_{i=1}^\infty \mathbb{P}_{M,n}(i)x_i\| - \mathbb{S}^\gamma_{T,n}(s_{q_{k_{n-1}+1}})\leqslant 1/n + \mathbb{S}^\gamma_{T,n}(s_{q_{k_{n-1}+1}}) .\end{align*}  Since $\lim_n \mathbb{S}^\gamma_{T,n}(s_{q_{k_{n-1}+1}})=0$, these inequalities yield a contradiction for sufficiently large $n$.

$(ii)$ We may assume without loss of generality that $$\sup\{\|\sum_{n\in F}\ee_nx_n\|: F\in \mathcal{S}_\beta[\mathcal{S}_\alpha], |\ee_n|=1\}=C<\infty$$ and that $(x_n)_{n=1}^\infty$ is basic.  By Lemma \ref{c0} applied with $\mathcal{F}=\mathcal{S}_\gamma[\mathcal{S}_\alpha]$,  there exists $L\in [\nn]$ such that for all $H_1<H_2<\ldots$, $H_n\in MAX(\mathcal{S}_\gamma[\mathcal{S}_\alpha])\cap [L]^{<\nn}$, $\|\sum_{i\in H_n}x_i\|>n$.   We claim that for any $F_1<F_2<\ldots$, $F_n\in MAX(\mathcal{S}_\alpha)\cap [L]^{<\nn}$, $(\sum_{i\in F_n}x_i)_{n=1}^\infty$ fails to have a subsequence which is a $c_0^\gamma$-spreading model. In order to prove this, it is sufficient to prove that $(\sum_{i\in F_n} x_i)_{n=1}^\infty$ is not a $c_0^\gamma$-spreading model. To see this, simply observe that if $F_1<F_2<\ldots$, $F_n\in MAX(\mathcal{S}_\alpha)\cap [L]^{<\nn}$ and $(\sum_{i\in F_{r_n}}x_i)_{n=1}^\infty$ is a $c_0^\gamma$-spreading model, this contradicts the previous sentence, since $F_{r_1}<F_{r_2}<\ldots$ also lie in $MAX(\mathcal{S}_\alpha)\cap [L]^{<\nn}$.  Seeking a contradiction, suppose that $$\sup \{\|\sum_{n\in E}\sum_{i\in F_n}x_i\|: E\in \mathcal{S}_\gamma\}=D<\infty.$$    Now fix $1=s_1<s_2<\ldots$ such that for all $n\in\nn$, $s_{n+1}>\min F_{s_n}$.    Let $T=\cup_{n=1}^\infty F_{s_n}$   and let $H_1<H_2<\ldots$ be such that $H_n\in MAX(\mathcal{S}_\gamma[\mathcal{S}_\alpha])$ and $T=\cup_{n=1}^\infty H_n$.   Note that $\|\sum_{i\in H_n}x_i\|>n$ for all $n\in\nn$.    Note also that there exist $0=k_0<k_1<\ldots$ such that $H_n=\cup_{j=k_{n-1}+1}^{k_n} F_{s_j}$, and these numbers are uniquely determined by the property that $(\min F_{s_j})_{j=k_{n-1}+1}^{k_n}\in MAX(\mathcal{S}_\gamma)$. As is now familiar, we note that for each $n\in \nn$, $E_n:=(s_j)_{k_{n-1}+2}^{k_n}$ is a spread of a subset $(\min F_{s_j})_{j=k_{n-1}+1}^{k_n-1}$, whence $E_n\in \mathcal{S}_\gamma$.    We note that for each $n\in \nn$, $$n<\|\sum_{i\in H_n} x_i\| \leqslant \|\sum_{i\in F_{k_{n-1}+1}}x_i\|+ \|\sum_{j=k_{n-1}+2}^{k_n} \sum_{i\in F_{s_j}}x_i\|\leqslant C+\|\sum_{j\in E_n} x_i\|\leqslant C+D.$$ This is a contradiction for sufficiently large $n$.

\end{proof}

\subsection{Schreier and Baernstein spaces}

If $\mathcal{F}$ is a nice family, we let $X_\mathcal{F}$ denote the completion of $c_{00}$ with respect to the norm $$\|x\|_\mathcal{F}=\sup \{\|Ex\|_{\ell_1}:E\in \mathcal{F}\}.$$    In the case that $\mathcal{F}=\mathcal{S}_\xi$, we write $\|\cdot\|_\xi$ in place of $\|\cdot\|_{\mathcal{S}_\xi}$ and $X_\xi$ in place of $X_{\mathcal{S}_\xi}$. The spaces $X_\xi$ are called the \emph{Schreier spaces}.   Note that $X_0=c_0$ isometrically. 

Given $1<p< \infty$ and a nice family $\mathcal{F}$, we let $X_{\mathcal{F}, p}$ be the completion of $c_{00}$ with respect to the norm $$\|x\|_{\mathcal{F}, p}= \sup \Bigl\{\bigl(\sum_{i=1}^\infty \|E_ix\|_{\ell_1}^p\bigr)^{1/p}: E_1<E_2<\ldots, E_i\in \mathcal{F}\Bigr\}.$$   For convenience, we let $X_{\xi, p}$ and $\|\cdot\|_{\xi, p}$ denote $X_{\mathcal{S}_\xi, p}$ and $\|\cdot\|_{\mathcal{S}_\xi, p}$, respectively. We also let $X_{\xi, \infty}$ denote $X_\xi$.

\begin{rem}\upshape The Schreier families $\mathcal{S}_\xi$, $\xi<\omega_1$, possess the almost monotone property, which means that for any $\zeta<\xi<\omega_1$, there exists $m\in\nn$ such that if $m\leqslant E\in \mathcal{S}_\zeta$, then $E\in \mathcal{S}_\xi$. From this it follows that the formal inclusion $I:X_\xi\to X_\zeta$ is bounded for any $\zeta\leqslant \xi<\omega_1$.  In fact, there exists a tail subspace $[e_i:i\geqslant m]$ of $X_\xi$ such that the restriction of $I:[e_i:i\geqslant m]\to X_\zeta$ is norm $1$. We will use this fact throughout. 

It is also obvious that the formal inclusion from $X_{\xi,p}$ to $X_{\zeta, p}$ is bounded for any $\zeta\leqslant \xi<\omega_1$, as is the inclusion from $X_{\xi,p}$ to $X_{\xi, q}$ whenever $p<q\leqslant \infty$. Combining these facts yields that the formal inclusion from $X_{\xi,p}$ to $X_\zeta$ is bounded whenever $\zeta\leqslant \xi$.  Furthermore, the adjoints of all of these maps are also bounded.

\end{rem}

The following collects known facts about the Schreier and Baernstein spaces. Throughout, we let $\|\cdot\|_{\xi,p}$ denote the norm of $X_{\xi,p}$ as well as its first and second duals.

\begin{theorem} Fix $\xi<\omega_1$ and $1<p\leqslant \infty$. \begin{enumerate}[(i)]\item  $\|\sum_{i=1}^n x_i\|_{\xi,p}=\|\sum_{i=1}^n |x_i|\|_{\xi,p}$ for any disjointly supported $x_1, \ldots, x_n\in X_{\xi,p}$. \item  The canonical basis of $X_{\xi, p}$ is shrinking. \item The basis of $X_{\xi,p}$ is boundedly-complete (and $X_{\xi,p}$ is reflexive) if and only if $p<\infty$. \item   If $p<\infty$ and $1/p+1/q=1$, $$\|\sum_{i=1}^n x_i\|_{\xi,p}\geqslant \bigl(\sum_{i=1}^n \|x_i\|_{\xi,p}^p\bigr)^{1/p}$$ and $$\|\sum_{i=1}^n x^*_i\|_{\xi,p} \leqslant \bigl(\sum_{i=1}^n \|x^*_i\|_{\xi,p}^q\bigr)^{1/q}$$ for any $x_1<\ldots <x_n\in X_{\xi,p}$ and $x_1^*<\ldots <x^*_n$, $x^*_i\in X_{\xi,p}^*$. \item The canonical basis of $X_{\xi,p}$ is a weakly null $\ell_1^\xi$-spreading model, while every normalized, weakly null sequence in $X_{\xi,p}$ is $\xi+1$-weakly null. \item The space $X_\xi$ is isomorphically embeddable into $C(\mathcal{S}_\xi)$.   \end{enumerate}

\label{sbfacts}

\end{theorem}

\begin{rem}\upshape Throughout, if $E\in [\nn]^{<\nn}$, we will use the notation $x^*\sqsubset E$ to mean that $\|x^*\|_{c_0}\leqslant 1$ and $\text{supp}(x^*)= E$.   It is evident that for any regular family $\mathcal{F}$, $$\bigcup_{E\in \mathcal{F}}\{x^*: x^*\sqsubset E\}\subset B_{X^*_\mathcal{F}}.$$ Moreover, a convexity argument yields that for any $y^*\in B_{X^*_\mathcal{F}}$ with $\text{supp}(y^*)\subset F\in [\nn]^{<\nn}$, $$y^*\in \text{co}\Bigl(\bigcup_{F\supset E\in \mathcal{F}}\{x^*: x^*\sqsubset E\}\Bigr).$$

Finally, we note that if there exist $x^*_1<\ldots <x^*_d$ and for each $1\leqslant i\leqslant d$, there exist $l_i\in \nn$, $E_{i,j}\subset \text{supp}(x^*_i)$, and $x^*_{i,j}$, $j=1, \ldots, l_i$, such that $x^*_{i,j}\sqsubset E_{i,j}\subset \text{supp}(x_i^*)$, $x^*_i\in \text{co}(x^*_{i,j}: 1\leqslant j\leqslant l_i)$,  and for each $(j_i)_{i=1}^d\in \prod_{i=1}^d \{1, \ldots, l_i\}$, $\cup_{i=1}^d E_{i, j_i}\in \mathcal{F}$, then $$\|\sum_{i=1}^d x^*_i\|_{X^*_\mathcal{F}}\leqslant 1.$$    Moreover, if we replace $x^*_i$ with $a_ix^*_i$, where $a_1, \ldots, a_d$ are such that $|a_i|\leqslant 1$ for each $1\leqslant i\leqslant d$, the resulting functionals $a_1 x^*_1, \ldots, a_d x^*_d$ also satisfy the hypotheses, so $\|\sum_{i=1}^d a_ix_i^*\|_{X^*_\mathcal{F}}\leqslant 1$ for any $(a_i)_{i=1}^d\in \ell_\infty^d$.

Let us see why $\|\sum_{i=1}^d x^*_i\|_{X^*_\mathcal{F}}\leqslant 1$.    Write $x_i^*=\sum_{j=1}^{l_i} w_{i,j}x^*_{i,j}$ where $w_{i,j}\geqslant 0$ and $\sum_{j=1}^{l_i}w_{i,j}=1$.     Let $I=\prod_{i=1}^d \{1, \ldots, l_i\}$ and for each $t=(j_i)_{i=1}^d\in I$, let $w_t=\prod_{i=1}^d w_{i, j_i}$ and $x^*_t=\sum_{i=1}^d x^*_{i, j_i}$. Then $x^*=\sum_{t\in I}w_t x^*_t$, $w_t\geqslant 0$, and $\sum_{t\in I}w_t=1$. Therefore it suffices to show that $\|x^*_t\|_{X^*_\mathcal{F}}\leqslant 1$ for each $t\in I$.   But $x^*_t\sqsubset \cup_{i=1}^d E_{i, j_i}\in \mathcal{F}$, whence $\|x^*_t\|_{X^*_\mathcal{F}}\leqslant 1$ follows. 

\label{daniel}
\end{rem}

\begin{proposition} Fix $0\leqslant \gamma, \delta<\omega_1$, and $1<p\leqslant \infty$.   \begin{enumerate}[(i)]\item If $(x^*_n)_{n=1}^\infty \subset X^*_\gamma$ is weakly null and satisfies $\underset{n}{\lim\inf} \|x^*_n\|_\gamma^*<C$, then there exists a subsequence $(x^*_{n_i})_{i=1}^\infty$ of $(x^*_n)_{n=1}^\infty$ such that for any $G\in \mathcal{S}_\delta$, $\|\sum_{i\in G} x^*_{n_i}\|_{\gamma+\delta}^*<C$.   \item Suppose $(x_n)_{n=1}^\infty \subset X_{\gamma+\delta,p}$ is weakly null  in $X_{\gamma+\delta,p}$, and for every $\beta<\gamma$, $\lim_n \|x_n\|_\beta=0$. Then  every subsequence of $(x_n)_{n=1}^\infty$ has a further subsequence which is dominated by a subsequence of the $X_{\delta,p}$ basis. \item If $(x_n)_{n=1}^\infty\subset X_{\gamma+\delta,p}$ is a weakly null sequence such that $\lim\sup_n \|x_n\|_\gamma>0$, then $(x_n)_{n=1}^\infty$ has a subsequence which dominates the $X_{\delta,p}$ basis. \end{enumerate}

\label{phenomenal}

\end{proposition}

\begin{proof}$(i)$ By passing to a subsequence, we may assume that $(x^*_n)_{n=1}^\infty$ is a block sequence and $\sup_n \|x^*_n\|<C_1<C$.  By scaling, we may assume $C_1=1$.   For each $n\in\nn$, let $S_n=\text{supp}(x^*_n)$.  For each $n\in\nn$, it follows from convexity and compactness arguments that for each $n\in\nn$, there exist $d_n$,  $(x^*_{n,i})_{i=1}^{d_n}$,  and $(E_{n,i})_{i=1}^{d_n}\subset \mathcal{S}_\gamma\cap [S_n]^{<\nn}$ such that  $x^*_{n,i}\sqsubset E_{n,i}$, and $x^*_n\in \text{co}(x^*_{i,n}: 1\leqslant i\leqslant d_n)$.   By Lemma \ref{tech}, there exist $n_1<n_2<\ldots$ such that for any $G\in \mathcal{S}_\delta$ and $E_1<E_2<\ldots$, $E_i\in \mathcal{S}_\gamma$, $\cup_{i\in G}E_{n_i}\in \mathcal{S}_{\gamma+\delta}$.    Now we conclude that for each $G\in \mathcal{S}_\delta$, $\|\sum_{n\in G}x^*_n\|_{\gamma+\delta}\leqslant C_1=1$ using the facts contained in Remark \ref{daniel}.

$(ii)$ By perturbing and scaling, we may assume $(x_n)_{n=1}^\infty \subset B_{X_{\xi,p}}$ is a block sequence.  If $\gamma$ is a successor, let $\gamma_n+1=\gamma$ for all $n\in \nn$ if $\gamma$ is a successor. If $\gamma$ is a limit ordinal, let $(\gamma_n)_{n=1}^\infty$ be the  sequence defining  $\mathcal{S}_\gamma$.    For each $n\in\nn$, let $\ee_n=2^{-n-2}$.  Let $m_n=\max \text{supp}(x_n)$.    We may recursively choose $1=k_1<k_2<\ldots$ such that for any $n<l$, $$\|x_{k_l}\|_{\gamma_{k_n}}<\ee_n/m_{k_n}.$$    By relabeling, we may assume $k_n=n$.

Now by Lemma \ref{tech}, we may fix $(n_i)_{i=1}^\infty$ such that if $E\in \mathcal{S}_{\gamma+\delta}$, there exist $E_1<\ldots <E_d$, $\varnothing\neq E_i\in \mathcal{S}_\gamma$ such that $(n_{\min E})_{i=1}^d\in \mathcal{S}_\delta$ and $E=\cup_{i=1}^d E_i$.  Now let $r_i=n_{m_i}$.   We first consider the $p=\infty$ case.  We claim that $(x_i)_{i=1}^\infty$ is dominated by $(e_{r_i})_{i=1}^\infty\subset X_\delta$.   Fix $(a_i)_{i=1}^\infty\in c_{00}\cap S_{\ell_\infty}$ and let $x=\sum_{i=1}^\infty a_ix_i$ and $y=\sum_{i=1}^\infty a_ie_{r_i}$.    Fix $E\in \mathcal{S}_{\gamma+\delta}$ and write $E=\cup_{i=1}^d E_i$, where $E_1<\ldots <E_d$, $\varnothing\neq E_i\in \mathcal{S}_\gamma$, and $(n_{\min E_i})_{i=1}^d\in \mathcal{S}_\delta$.   If $\gamma=0$, we can take each $E_i$ to be a singleton.   By omitting any superfluous $E_i$ and relabeling, we may assume that for each $1\leqslant i\leqslant d$, there exists $j$ such that $E_ix_j\neq 0$.

As the following estimates involve many definitions, we say a word before proceeding. For each $E_i$, our choice of the sequence $(x_i)_{i=1}^\infty$ will yield that $\|E_ix_l\|_{\ell_1}$ will be essentially negligible for all vectors except the first one whose support $E_i$ intersects.    Moreover, of all of the sets $E_i$ which intersect the support of $x_l$, since the sets are successive, at most one of the sets can intersect the support of a later vector, so we can control the number of negligible pieces.    For each $1\leqslant i\leqslant d$, let $j_i=\min \{l: E_ix_l\neq 0\}$ and $J=\{j_i: 1\leqslant i\leqslant d\}$.   For each $j\in J$, let $S_j=\{i\leqslant d: j_i=j\}$.      For each $j\in J$, let $s_j=\max S_j$ and let $T_j=S_j\setminus \{s_j\}$.  Note that for each $i\in S_j$, $E_ix_l=0$ for all $l<j$ by the minimality of $j=j_i$.  Note also that for each $i\in T_j$, $E_ix_l=0$ for all $l>j$, since $$\max E_i <\min E_{s_j}\leqslant \max \text{supp}(x_j)<\min \text{supp}(x_l).$$        Furthermore, since $E_{s_j}x_j\neq 0$, $E_{s_j}\in \mathcal{S}_\gamma$ with $\min E_{s_j} \leqslant m_j$.  If $\gamma$ is a limit ordinal, then $E_{s_j}\in \mathcal{S}_{\gamma_{m_j}}$, which means that for any $k>j$, $$\|E_{s_j}x_k\|_{\ell_1} \leqslant  \ee_k/m_j\leqslant \ee_k.$$   If $\gamma$ is a successor, then $\gamma=\gamma_{m_j}+1$ and $\min E_{s_j}\leqslant m_j$ yield that $E_{s_j}=\cup_{i=1}^q F_i$ for some $F_1<\ldots <F_q$, $q\leqslant m_j$, and $F_i\in \mathcal{S}_{\gamma_{m_j}}$.  Then for $k>j$, $$\|E_{s_j}x_k\|_{\ell_1}\leqslant \sum_{i=1}^q\|F_ix_k\|_{\ell_1} \leqslant m_n\|x_k\|_{\gamma_{m_j}}\leqslant \ee_k.$$     In the case $\gamma=0$, each $E_i$ is a singleton, so we have the trivial estimate that for $i\in S_j$ and $l>j$, $E_ix_l=0$.      Therefore in each of the $\gamma=0$, $\gamma$ a successor, and $\gamma$ a limit ordinal cases, $$\sum_{i\in S_j} \|E_ix\|_{\ell_1} \leqslant |a_j|\|Ex_j\|_{\ell_1} +\sum_{k=j+1}^\infty \|E_{s_j}x_k\|_{\ell_1} \leqslant |a_j|+\sum_{k=j+1}^\infty \ee_k.$$    Summing over $i$ yields that $$\|Ex\|_{\ell_1}\leqslant \sum_{j\in J}\sum_{i\in S_j} \|E_ix\|_{\ell_1} \leqslant \sum_{j\in J} |a_j| + \sum_{j\in J}\sum_{k=j+1}^\infty \ee_k \leqslant \sum_{j\in J}|a_j|+ \sum_{j=m(E)}^\infty\sum_{k=j}^\infty \ee_k,$$ where $m(E)=\min \{j: Ex_j\neq 0\}$.  Now for each $j\in J$, fix some $i_j\in \{1, \ldots, d\}$ such that $j=j_{i_j}$. Then $j\mapsto i_j$ is an injection of $J$ into $\{1, \ldots, d\}$, whence $(m_{i_j})_{j\in J}$ is a spread of a subset of $(\min E_i)_{i=1}^d$. Therefore $T(E):=(r_{i_j})_{j\in J}=(n_{m_{i_j}})_{j\in J}$ is a spread of a subset of $(n_{\min E_i})_{i=1}^d\in \mathcal{S}_\delta$, so $T(E)\in \mathcal{S}_\delta$.  Therefore $$\|y\|_\delta \geqslant \|T(E)y\|_{\ell_1} = \sum_{j\in J}|a_j|.$$  Collecting these estimates and recalling our  assumption that $(a_i)_{i=1}^\infty \in S_{\ell_\infty}$, we deduce that $$\|x\|_{\gamma+\delta} \leqslant \sum_{j\in J}|a_i|+\sum_{j=m(E)}^\infty \sum_{k=j}^\infty \ee_k \leqslant 2\|y\|_\delta.$$ This completes the $p=\infty$ case. 

Now assume $1<p<\infty$.  Fix $E_1<E_2<\ldots$, $E_i\in \mathcal{S}_{\gamma+\delta}$.    Let $x=\sum_{i=1}^\infty a_ix_i$, $y=\sum_{i=1}^\infty a_ie_{r_i}$ as in the previous paragraph.     For each $i\in\nn$, let $$J_i=\{j\in \nn: (\forall i\neq k\in\nn)(E_jx_k=0)\}.$$   Let $J=\cup_{i=1}^\infty J_i$ and $I=\nn\setminus J$. Let us rename the sets $(E_i)_{i\in I}$ as $F_1<G_1<F_2<G_2<\ldots$ (ignoring this step if $I$ is empty and with the appropriate notational change if $I$ is finite and non-empty). By the properties of  $I$, for each $i$ such that $F_i$ (resp. $G_i$) is defined, there exist at least two distinct indices $j,k$ such that $F_ix_j, F_ix_k\neq 0$ (resp.  at least two distinct indices $j', k'$ such that $G_ix_{j'}, G_ix_{k'}\neq 0$). From this it follows that, with $$U_i=\{j: F_ix_j\neq 0\}$$ and $$V_i=\{j: G_ix_j\neq 0\},$$ the sets $(U_i)_i$ are successive, as are $(V_i)_i$.  In particular, $F_ix_j=G_ix_j=0$ whenever $j<i$.  Observe that \begin{align*} \bigl(\sum_{i\in J} \|E_ix\|_{\ell_1}^p\bigr)^{1/p} & = \bigl(\sum_{j=1}^\infty |a_j|^p\sum_{i\in J_j}\|E_ix_j\|_{\ell_1}^p\bigr)^{1/p} \leqslant \|(a_j)_{j=1}^\infty\|_{\ell_p} \leqslant \|y\|_{\delta,p}.\end{align*}  Now, arguing as in the $p=\infty$ case, for each $i$ such that $F_i$ is defined, if $m(F_i)=\min \{j: F_ix_j\neq 0\}$, there exists a set $T(F_i)\in \mathcal{S}_\delta$ such that $$\|F_ix\|_{\ell_1}\leqslant \|T(F_i)y\|_{\ell_1}+\sum_{l=m(F_i)}^\infty \sum_{k=l}^\infty \ee_k.$$  Furthermore, $T(F_i)\subset \{n_{m_j}: j\in  U_i\}$, whence the sets $T(F_i)$ are successive, since the sets $U_i$ are.  From this,  the triangle inequality, and the fact that $m(F_i)\geqslant i$ for each appropriate $i$,  it follows that \begin{align*} \Bigl(\sum_i \|F_ix\|_{\ell_1}^p\Bigr)^{1/p} & \leqslant  \Bigl(\sum_{i=1}^\infty \|T(F_i)y\|_{\ell_1}^p\Bigr)^{1/p} +\sum_{i=1}^\infty \sum_{l=m(F_i)}^\infty \sum_{k=l}^\infty \ee_k \\ & \leqslant \|y\|_{\delta, p} +\sum_{i=1}^\infty\sum_{l=i}^\infty \sum_{k=l}^\infty \ee_k \\ & = \|y\|_{\delta, p} +1 \leqslant 2\|y\|_{\delta,p}.\end{align*} A similar argument yields that $$\Bigl(\sum_i \|G_ix\|_{\ell_1}^p\Bigr)^{1/p} \leqslant 2\|y\|_{\delta, p}.$$    Therefore $$\Bigl(\sum_{j=1}^\infty \|E_jx\|_{\ell_1}^p\Bigr)^{1/p} \leqslant 5\|y\|_{\delta, p}.$$  Since $E_1<E_2<\ldots$, $E_i\in \mathcal{S}_{\gamma+\delta}$ were arbitrary, $\|x\|_{\gamma+\delta,p}\leqslant 5 \|y\|_{\delta,p}$.

$(iii)$ By passing to a subsequence and perturbing, we may assume $(x_n)_{n=1}^\infty$ is a block sequence in $X_{\gamma+\delta,p}$ and $\inf_n \|x_n\|_\gamma=\ee>0$. We may fix a block sequence $(x^*_n)\in \ee^{-1} B_{X^*_\gamma}$ biorthogonal to $(x_n)_{n=1}^\infty$. By $(i)$, after passing to a subsequence and using properties of the $X_{\gamma+\delta,p}$ basis, assume that $$\sup \{\|\sum_{n\in G} \ee_n x^*_n\|_{\gamma+\delta}: G\in \mathcal{S}_\delta, |\ee_n|=1\}\leqslant 1/\ee.$$    If $p=\infty$, note that for any $(a_i)_{i=1}^\infty \in c_{00}$, \begin{align*} \|\sum_{i=1}^\infty a_ie_i\|_\gamma & = \sup \{\sum_{n\in G}|a_n|: G\in \mathcal{S}_\delta\}  \\ & \leqslant \sup \Bigl\{\text{Re\ }\bigl(\sum_{n\in G}\ee_n x^*_n\bigr)\bigl(\sum_{n=1}^\infty a_nx_n\bigr): G\in \mathcal{S}_\gamma, |\ee_n|=1\Bigr\} \\ &  \leqslant \ee^{-1}\bigl(\sum_{n=1}^\infty a_nx_n\bigr). \end{align*}

Now suppose that $1<p<\infty$.   Fix $(a_i)_{i=1}^\infty\in c_{00}$ and let $x=\sum_{i=1}^\infty a_ie_i$.  Fix $E_1<E_2<\ldots <E_n$, $E_i\in \mathcal{S}_\delta$ and a sequence $(b_i)_{i=1}^n\in S_{\ell_q^n}$, such that $$\|x\|_{\gamma, p}=\bigl(\sum_{i=1}^n \|E_ix\|_{\ell_1}^p\bigr)^{1/p} = \sum_{i=1}^n b_i\bigl(\sum_{j\in E_i}|a_j|\bigr).$$   Let $y^*_i=\sum_{j\in E_i} \ee_j x^*_j$, where $\ee_j a_j=|a_j|$, and let $y^*=\sum_{i=1}^n b_iy^*_i$. Since $\|y^*_i\|_{\gamma+\delta}\leqslant \ee^{-1}$ and $\sum_{i=1}^n b_i^q=1$, $\|y^*\|_{\gamma+\delta, p}\leqslant \ee^{-1}$.    Indeed, by H\"{o}lder's inequality, for any $x\in c_{00}$, if $I_1<\ldots <I_n$ are such that $\text{supp}(y^*_i)\subset I_i$, then $$|y^*(x)| \leqslant \sum_{i=1}^n b_i|y^*_i(x)|\leqslant \ee^{-1}\sum_{i=1}^n b_i\|I_ix_i\|_\gamma \leqslant \ee^{-1}\bigl(\sum_{i=1}^n b_i^q\bigr)\bigl(\sum_{i=1}^n  \|I_ix\|_\gamma^p\bigr)^{1/p} \leqslant \ee^{-1}\|x\|_{\gamma,p}.$$  Moreover, $$\ee^{-1}\|\sum_{i=1}^\infty a_i x_i\|_{\xi,p} \geqslant y^*(\sum_{i=1}^\infty a_ix_i) =\sum_{i=1}^n b_i\bigl(\sum_{j\in E_i} |a_j|\bigr)=\|x\|_{\gamma, p}.$$

\end{proof}

Let us recall that for any ordinals $\gamma, \xi$ with $\gamma\leqslant \xi$, there exists a unique ordinal $\delta$ such that $\gamma+\delta=\xi$. We denote this ordinal $\delta$ by $\xi-\gamma$.

We also recall that any non-zero ordinal $\xi$ admits a unique representation (called the \emph{Cantor normal form}) as $$\xi=\omega^{\ee_1}n_1+\ldots +\omega^{\ee_k}n_k,$$ where $k, n_1, \ldots, n_k\in\nn$ and $\ee_1>\ldots >\ee_k$.   By writing $\omega^\ee n=\omega^\ee+\ldots +\omega^\ee$, where the summand $\omega^\ee$ appears $n$ times, we may also uniquely represent $\xi$ as $$\xi=\omega^{\delta_1}+\ldots +\omega^{\delta_l},$$ where $l\in\nn$ and $\delta_1\geqslant \ldots \geqslant \delta_l$.  In this case, $\delta_1=\ee_1$. Let us define $\lambda(\xi)=\omega^{\ee_1}=\omega^{\delta_1}$ in this case. For completeness, let us define $\lambda(0)=0$ and note that if $\zeta\leqslant \xi$, $\lambda(\zeta)\leqslant \lambda(\xi)$.

\begin{theorem} Fix $\xi<\omega_1$ and $1<p\leqslant \infty$.   Fix a weakly null sequence $(x_n)_{n=1}^\infty\subset X_{\xi,p}$. Let $\Gamma=\{\zeta\leqslant \xi: \lim\sup_n \|x_n\|_\zeta>0\}$. \begin{enumerate}[(i)]\item If $p=\infty$, then $\Gamma=\varnothing$ if and only if $(x_n)_{n=1}^\infty$ is norm null.   \item If $p<\infty$ and $\Gamma=\varnothing$, then either $(x_n)_{n=1}^\infty$ is norm null or $(x_n)_{n=1}^\infty$ has a subsequence equivalent to the canonical $\ell_p$ basis. \item If $\Gamma\neq \varnothing$ and $\gamma=\min \Gamma$,  then  $(x_n)_{n=1}^\infty$ admits a subsequence which is equivalent to a subsequence of the $X_{\xi-\gamma,p}$ basis.  In particular, $(x_n)_{n=1}^\infty$ is $\xi-\gamma+1$ weakly null and not $\xi-\gamma$ weakly null.  \item If $p=\infty$,  then  every subsequence of $(x_n)_{n=1}^\infty$ has a further WUC subsequence if and only if $\Gamma\subset \{\xi\}$. \item If $0<\xi$, a weakly null sequence $(x_n)_{n=1}^\infty $ is $\xi$-weakly null if and only if for every $\beta<\lambda(\xi)$, $\lim_n \|x_n\|_\beta=0$.  \end{enumerate}

\label{shinsuke}

\end{theorem}

\begin{proof} First note that by the almost monotone property of the Schreier families, if $\zeta\in \Gamma$, then $[\zeta, \xi]\subset \Gamma$.

$(i)$ It is evident that $\lim_n \|x_n\|_\xi=0$ if and only if $\xi\notin \Gamma$.

$(ii)$ If $\xi\notin \Gamma$, then let $\gamma=\xi$ and $\delta=0$. By Proposition \ref{phenomenal}$(ii)$, any subsequence of $(x_n)_{n=1}^\infty$ has a further subsequence which is dominated by a subsequence of the $X_{\delta, p}=\ell_p$ basis.  Then since every seminormalized block sequence in $X_{\xi,p}$ which dominates the $\ell_p$ basis, either $\lim_n \|x_n\|_{\xi,p}=0$, or $(x_n)_{n=1}^\infty$ has a seminormalized subsequence which dominates the $\ell_p$ basis, and this subsequence has a further subsequence equivalent to the $\ell_p$ basis.

$(iii)$ Let $\delta=\xi-\gamma$, so that $\gamma+\delta=\xi$.  Proposition \ref{phenomenal}$(ii)$ yields that every subsequence of $(x_n)_{n=1}^\infty$ has a further subsequence which is dominated by a subsequence of the $X_{\delta, p}$ basis. Since no subsequence of the $X_{\delta,p}$ basis is an $\ell_1^{\delta+1}$-spreading model, this yields that $(x_n)_{n=1}^\infty$ is $\delta+1$-weakly null.   Since $\gamma\in \Gamma$, Proposition \ref{phenomenal}$(iii)$ yields the existence of a subsequence $(y_n)_{n=1}^\infty$ of $(x_n)_{n=1}^\infty$ which dominates the $X_{\delta, p}$ basis, so $(x_n)_{n=1}^\infty$ is not $\delta$-weakly null.     Now note that $\lim_n \|y_n\|_\beta=0$ for all $\beta<\gamma$, so by Proposition \ref{phenomenal}$(ii)$, there exists a subsequence $(z_n)_{n=1}^\infty$ of $(y_n)_{n=1}^\infty$ which is dominated by some subsequence $(x_{n_i})_{i=1}^\infty$ of the canonical $X_{\delta,p}$ basis. This sequence $(z_n)_{n=1}^\infty$ also dominates some subsequence $(x_{m_i})_{i=1}^\infty$ of the canonical $X_{\delta,p}$ basis (where $m_i$ has the property that $z_i=y_{m_i}$).    Now let us choose $1=k_1<k_2<\ldots$ such that $m_{k_{i+1}}>n_{k_i}$ for all $i\in\nn$ and let $u_i=z_{k_i}$. Then $(u_i)_{i=1}^\infty$ is dominated by some subsequence $(x_{r_i})_{i=1}^\infty$ of the $X_{\delta, p}$ basis and dominates some subsequence $(x_{s_i})_{i=1}^\infty$ of the $X_{\delta,p}$ basis, where $s_1\leqslant r_1<s_2\leqslant r_2<\ldots$. This is seen by taking $s_i=m_{k_i}$ and $r_i=n_{k_i}$. But it is observed in \cite{CR} that two such subsequences of the $X_{\delta,p}$ basis must be $2$-equivalent, whence $(u_i)_{i=1}^\infty$ is equivalent to $(e_{r_i})_{i=1}^\infty$ (and to $(e_{s_i})_{i=1}^\infty$).

$(iv)$ If $\Gamma\subset \{\xi\}$, then by Proposition \ref{phenomenal}$(ii)$ applied with $\gamma=\xi$ and $\delta=0$, every subsequence of $(x_n)_{n=1}^\infty$ has a further subsequence which is dominated by the $X_\delta=c_0$ basis. Conversely, if $\xi>\gamma\in \Gamma$, then with $\delta=\xi-\gamma>0$, $(x_n)_{n=1}^\infty$ has a subsequence which is an $\ell_1^\delta$-spreading model. No subsequence of this sequence can be WUC.

$(v)$ Note that both conditions are satisfied if $(x_n)_{n=1}^\infty$ is norm null, so assume $(x_n)_{n=1}^\infty$ is not norm null. If $\Gamma=\varnothing$, then $p<\infty$, and every subsequence of $(x_n)_{n=1}^\infty$ has a further subsequence which is equivalent to the $\ell_p$ basis, which means $(x_n)_{n=1}^\infty$ is $1$-weakly null, and therefore $\xi$-weakly null.  Thus both conditions are satisfied in this case as well.

It remains to consider the case  $\Gamma\neq \varnothing$. Let $\gamma=\min \Gamma$.    Let us write $$\xi=\omega^{\ee_1}+\ldots +\omega^{\ee_k},$$ where $\ee_1\geqslant \ldots \geqslant \ee_k$.    Note that $\lambda(\xi)=\omega^{\ee_1}$.  First assume that $\lim_n \|x_n\|_\beta=0$ for all $\beta<\lambda(\xi)$, which means $\gamma\geqslant \lambda(\xi)$.   Then if $\gamma+\delta=\xi$, $\delta \leqslant \omega^{\ee_2}+\ldots +\omega^{\ee_k}$.   By $(iii)$, $(x_n)_{n=1}^\infty$ is $\delta+1$-weakly null, and $$\delta+1\leqslant \omega^{\ee_2}+\ldots +\omega^{\ee_k}+1 \leqslant \omega^{\ee_2}+\ldots +\omega^{\ee_k}+\omega^{\ee_1}\leqslant \omega^{\ee_1}+\ldots +\omega^{\ee_k}=\xi$$ yields that $(x_n)_{n=1}^\infty$ is $\xi$-weakly null.   Conversely, assume there exists $\beta<\lambda(\xi)$ such that $\lim\sup_n \|x_n\|_\beta>0$.   Then $\gamma<\lambda(\xi)$. If $\gamma+\delta=\xi$, then $\delta=\xi$.   By $(iii)$, $(x_n)_{n=1}^\infty$ is not $\xi$-weakly null.

\end{proof}

\begin{corollary} For any $0<\xi<\omega_1$ and any seminormalized, weakly null sequence $(x_n)_{n=1}^\infty$ in $X_{\omega^\xi}$, $(x_n)_{n=1}^\infty$ has a subsequence $(y_n)_{n=1}^\infty$  which is either equivalent to the canonical $c_0$ basis or to a subsequence of the $X_{\omega^\xi}$ basis.

\end{corollary}

\begin{proof} By Theorem \ref{shinsuke}$(iv)$, every subsequence of $(x_n)_{n=1}^\infty$ has a further WUC (and therefore equivalent to the $c_0$ basis) subsequence if and only if $\lim_n \|x_n\|_\beta=0$ for every $\beta<\xi=\lambda(\xi)$.  If this condition fails, then there exists a minimum $\gamma<\omega^\xi$ such that $\lim\sup_n \|x_n\|_\gamma>0$. Then if $\gamma+\delta=\omega^\xi$, $\delta=\omega^\xi$, and $(x_n)_{n=1}^\infty$ has a subsequence equivalent to a subsequence of the $X_{\omega^\xi}$ basis.

\end{proof}

\begin{corollary} Fix $0<\xi<\omega_1$, $1<p\leqslant \infty$,  and let $(x_n)_{n=1}^\infty \subset X_{\xi,p}$ be weakly null.  Then $(x_n)_{n=1}^\infty$ is $\xi$-weakly null in $X_{\xi,p}$ if and only if for every $\gamma<\lambda(\xi)$, $\lim_n \|x_n\|_\gamma=0$ if and only if every subsequence of $(x_n)_{n=1}^\infty$ has a further subsequence which is WUC in $X_{\lambda(\xi)}$. 

\label{AC}

\end{corollary}

\begin{proof} This follows from combining Theorem \ref{shinsuke}$(iv)$ and $(v)$.

\end{proof}

\begin{lemma} Fix $0<\xi<\omega_1$ and $1<p\leqslant \infty$. \begin{enumerate}[(i)]\item If $(x_n^{**})_{n=1}^\infty\subset X^{**}_{\xi,p}$ is $\xi$-weakly null, then for every $\gamma<\lambda(\xi)$, $\lim_n\|x^{**}_n\|_\gamma=0$. \item If $(x^{**}_n)_{n=1}^\infty \subset X^{**}_{\xi,p}$ is $\xi$-weakly null and $(x^*_n)_{n=1}^\infty \subset X^*_{\lambda(\xi)}$ is weakly null, then $\lim_n x^{**}_n(x^*_n)=0$. \end{enumerate}

\label{double dual}
\end{lemma}

\begin{proof}$(i)$ Suppose not. Then for some $\gamma<\lambda(\xi)$ and $\ee>0$, we may pass to a subsequence and assume $\inf_n \|x^{**}_n\|_\gamma>\ee$.    We may choose a sequence $(x^*_n)_{n=1}^\infty\subset B_{X^*_\gamma}\cap c_{00}$ such that $\inf_n |x^{**}_n(x_n)|>\ee$.    Since $\lim_n x^{**}_n(e^*_i)=0$ for all $i\in\nn$, we may, by passing to a subsequence and replacing the functionals $x^*_n$ by tail projections thereof, assume that $(x^*_n)_{n=1}^\infty$ is a block sequence in $B_{X^*_\gamma}\cap c_{00}$.  Then by standard properties of ordinals, if $\delta$ is such that $\gamma+\delta=\xi$, $\delta=\xi$. By Proposition \ref{phenomenal}$(i)$, we may pass to a subsequence once more and asssume $(x^*_n)_{n=1}^\infty$ is a $c_0^\xi$-spreading model in $X^*_\xi$, and therefore weakly null in $X_\xi^*$.  By passing to a subsequence one final time, we may also assume $\sum_{n=1}^\infty \sum_{n\neq m=1}^\infty |x^{**}_n(x_m)|<\ee/2$.   Then by a standard duality argument, it follows that $(x^{**}_n)_{n=1}^\infty$ is not $\xi$-weakly null.   This contradiction finishes $(i)$. 

$(ii)$  Also by contradiction.  Assume $(x^{**}_n)_{n=1}^\infty\subset X_{\xi,p}^{**}$ is $\xi$-weakly null, $(x^*_n)_{n=1}^\infty \subset X^*_{\lambda(\xi)}$ is weakly null, and $\inf_n |x^{**}_n(x_n^*)|>\ee>0$.  By perturbing, we may assume $(x^*_n)_{n=1}^\infty$ is a block sequence and there exist $I_1<I_2<\ldots$ such that $I_nx^*_n=x^*_n$ for all $n\in\nn$.  Let $(\gamma_k)_{k=1}^\infty\subset [0, \lambda(\xi))$ be a sequence (possibly with repitition) such that $[0, \lambda(\xi))=\{\gamma_k: k\in\nn\}$.    By $(i)$, $\lim_n \|x^{**}_n\|_{\gamma_k}=0$ for all $k\in\nn$. By passing to a subsequence and relabeling, we may assume that for each $1\leqslant k\leqslant n$, $\|x^{**}_n\|_{\gamma_k}<1/n$.    Let $x_n=I_nx_n^{**}\in X_\xi$ and note that for each $\gamma<\lambda(\xi)$, $\lim_n \|x_n\|_\gamma=0$. Indeed, if $\gamma=\gamma_k$, then for all $n\geqslant k$, $$\|x_n\|_\gamma \leqslant \|x^{**}_n\|_{\gamma_k} \leqslant 1/n.$$   Since $I_nx^*_n=x_n^*$, $|x^*_n(x_n)|=|x^{**}_n(x^*_n)|>\ee$.    But by Corollary \ref{AC}, some subsequence of $(x_n)_{n=1}^\infty$, which we may assume is the entire sequence after relabeling, is WUC in $X_{\lambda(\xi)}$. But now we reach a contradiction by combining the facts that $(x_n)_{n=1}^\infty$ is WUC in $X_{\lambda(\xi)}$, $(x^*_n)_{n=1}^\infty\subset X_{\lambda(\xi)}^*$ is weakly null, and $\inf_n |x^*_n(x_n)|>0$.

\end{proof}

\section{Ideals of interest}

\subsection{Basic definitions}

We recall that $\textbf{Ban}$ is the class of all Banach spaces and $\mathfrak{L}$ denotes the class of all operators between Banach spaces.   For each pair $X,Y\in \textbf{Ban}$, $\mathfrak{L}(X,Y)$ is the class of all operators from $X$ into $Y$. Given a subclass $\mathfrak{I}$ of $\mathfrak{L}$, we let $\mathfrak{I}(X,Y)=\mathfrak{I}\cap \mathfrak{L}(X,Y)$.

We recall that a class $\mathfrak{I}$ is said to have the \emph{ideal property} provided that for any $W,X,Y,Z\in \textbf{Ban}$, $C\in \mathfrak{L}(W,X)$, $B\in \mathfrak{I}(X,Y)$, and $A\in \mathfrak{L}(Y,Z)$, $ABC\in \mathfrak{I}(W,Z)$. 

We say that $\mathfrak{I}$ is an \emph{operator ideal} (or just \emph{ideal}) provided that \begin{enumerate}[(i)]\item $\mathfrak{I}$ has the ideal property, \item $I_\mathbb{K}\in \mathfrak{I}$, \item for each $X,Y\in \textbf{Ban}$, $\mathfrak{I}(X,Y)$ is a vector subspace of $\mathfrak{L}(X,Y)$.  \end{enumerate}

Given an operator ideal $\mathfrak{I}$, we define the \begin{enumerate}[(i)]\item \emph{closure} $\overline{\mathfrak{I}}$ of $\mathfrak{I}$ to be the class of operators such that for every $X,Y\in \textbf{Ban}$, $\overline{I}(X,Y)=\overline{\mathfrak{I}(X,Y)}$, \item \emph{injective hull} $\mathfrak{I}^\text{inj}$ of $\mathfrak{I}$ to be the class of all operators $A:X\to Y$ such that if there exists $Z\in \textbf{Ban}$ and an isometric (equivalently, isomorphic) embedding $j:Y\to Z$ such that $jA\in \mathfrak{I}(X,Z)$,  \item \emph{surjective hull} $\mathfrak{I}^{\text{sur}}$ of $\mathfrak{I}$ to be the class of all operators $A:X\to Y$ such that there exist $W\in \textbf{Ban}$ and a quotient map (equivalently, a surjection) $q:W\to X$ such that $Aq\in \mathfrak{I}(W,Y)$, \item \emph{dual} $\mathfrak{I}^\text{dual}$ to be the class of all operators $A:X\to Y$ such that $A^*\in \mathfrak{I}(Y^*, X^*)$.   \end{enumerate}

Each of $\mathfrak{I}$, $\mathfrak{I}^\text{inj}$, $\mathfrak{I}^\text{sur}$ is also an ideal.

Given two ideals $\mathfrak{I}, \mathfrak{J}$, we let \begin{enumerate}[(i)]\item $\mathfrak{I}\circ \mathfrak{J}^{-1}$ denote the class of all operators $A:X\to Y$ such that for all $W\in \textbf{Ban}$ and $R\in \mathfrak{J}(W,X)$, $AR\in \mathfrak{I}(W,Y)$,  \item $\mathfrak{I}^{-1}\circ \mathfrak{J}$ denote the class of all operators $A:X\to Y$ such that for all $Z\in \textbf{Ban}$ and all $L\in \mathfrak{I}(Y,Z)$, $LA\in \mathfrak{J}(X,Z)$. \end{enumerate} We remark that for any three ideals $\mathfrak{I}_1, \mathfrak{I}_2, \mathfrak{J}$, $$(\mathfrak{I}_1^{-1}\circ \mathfrak{J})\circ \mathfrak{I}_2^{-1}=\mathfrak{I}_1^{-1}\circ (\mathfrak{J}\circ \mathfrak{I}_2^{-1}),$$ so that the symbol $\mathfrak{I}_1^{-1}\circ \mathfrak{J}\circ \mathfrak{I}_2^{-1}$ is unambiguous.

We say an operator ideal is \begin{enumerate}[(i)]\item \emph{closed} if $\overline{\mathfrak{I}}=\mathfrak{I}$, \item \emph{injective} if $\mathfrak{I}=\mathfrak{I}^\text{inj}$, \item \emph{surjective} if $\mathfrak{I}=\mathfrak{I}^\text{sur}$, \item \emph{symmetric} if $\mathfrak{I}=\mathfrak{I}^\text{dual}$. \end{enumerate}

With each ideal, we will associate the class of Banach spaces the identity of which lies in the given ideal. Our ideals will be denoted by fraktur lettering ($\mathfrak{A}, \mathfrak{B}, \mathfrak{I}, \ldots$) and the associated space ideal will be denoted by the same sans serif letter ($\textsf{A}, \textsf{B}, \textsf{I}, \ldots$).

We next list some ideals of interest. We let $\mathfrak{K}$, $\mathfrak{W}$, and $\mathfrak{V}$ denote the class of compact, weakly compact, and completely continuous operators, respectively. 

For the remaining paragraphs in this subsection, $\xi$ will be a fixed ordinal in $[0, \omega_1]$. We let $\mathfrak{W}_\xi$ denote the class of operators $A:X\to Y$ such that any bounded sequence in $X$ has a subsequence whose image under $A$ is $\xi$-convergent in $Y$ (let us recall that a sequence $(y_n)_{n=1}^\infty\subset Y$ is said to be $\xi$-\emph{convergent to} $y\in Y$ if $(y_n-y)_{n=1}^\infty$ is $\xi$-weakly null). Note that $\mathfrak{W}_0=\mathfrak{K}$ and $\mathfrak{W}_{\omega_1}=\mathfrak{W}$.   Furthermore, $\mathfrak{W}_1$ coincides with the class of Banach-Saks operators. For this reason, we refer to $\mathfrak{W}_\xi$ as the class of $\xi$-\emph{Banach-Saks operators}. This class was introduced in this generalilty in \cite{BF}. 

We let $\mathfrak{wBS}_\xi$ denote the class of operators $A:X\to Y$such that for any weakly null sequence $(x_n)_{n=1}^\infty$, $(Ax_n)_{n=1}^\infty$ is $\xi$-weakly convergent to $0$ in $Y$.  Note that $\mathfrak{wBS}_0=\mathfrak{V}$, $\mathfrak{wBS}_{\omega_1}=\mathfrak{L}$, and $\mathfrak{wBS}_1$ is the class of weak Banach-Saks operators.  For this reason, we refer to $\mathfrak{wBS}_\xi$ as the class of $\xi$-\emph{weak Banach-Saks operators}. These classes were introduced in this generality in \cite{BC2}.

We let $\mathfrak{V}_\xi$ denote the class of operators $A:X\to Y$ such that for any $\xi$-weakly null sequence $(x_n)_{n=1}^\infty$, $(Ax_n)_{n=1}^\infty$ is norm nul.  It is evident that $\mathfrak{V}_{\omega_1}=\mathfrak{V}$ and $\mathfrak{V}_0=\mathfrak{L}$. These classes were introduced in this generality in \cite{CN}.

For $0\leqslant \zeta\leqslant \omega_1$, we let $\mathfrak{G}_{\xi, \zeta}$ denote the class of operators $A:X\to Y$ such that whenever $(x_n)_{n=1}^\infty$ is $\xi$-weakly null, $(Ax_n)_{n=1}^\infty$ is $\zeta$-weakly null.  We isolate this class because it is a simultaneous generalization of the two previous paragraphs. Indeed, $\mathfrak{V}_\xi=\mathfrak{G}_{\xi, 0}$, while $\mathfrak{wBS}_\xi= \mathfrak{G}_{\omega_1, \xi}$. It is evident that $\mathfrak{G}_{\xi, \zeta}=\mathfrak{L}$ whenever $\xi\leqslant \zeta$.  These classes are newly introduced here.

For $0\leqslant \zeta\leqslant \omega_1$, we let $\mathfrak{M}_{\xi, \zeta}$ denote the class of all operators $A:X\to Y$ such that for any $\xi$-weakly null $(x_n)_{n=1}^\infty\subset X$ and any $\zeta$-weakly null $(y^*_n)_{n=1}^\infty\subset Y^*$, $\lim_n y^*_n(Ax_n)=0$.  The class $\mathfrak{M}_{\omega_1, \omega_1}$ (sometimes denoted by $\mathfrak{DP}$) is a previously defined class of significant interest, most notably because the associated space ideal $\textsf{M}_{\omega_1, \omega_1}$ is the class of Banach spaces with the Dunford-Pettis property.  As a class of operators,  $\mathfrak{M}_{\xi, \zeta}$ has not previously been investigated, but the space ideals $\textsf{M}_{1, \omega_1}$ and $\textsf{M}_{\omega_1, \xi}$ have been investigated in \cite{GG} and \cite{AG}, respectively.

\begin{rem}\upshape Let us recall that the image of a $\xi$-weakly null sequence under a continuous, linear operator is also $\xi$-weakly null,  for any $0\leqslant \xi\leqslant \zeta\leqslant \omega_1$, any sequence which is $\xi$-weakly null is also $\zeta$-weakly null, and the $0$-weakly null sequences are the norm null sequences.   From this we deduce the following.   \begin{enumerate}[(i)]\item $\mathfrak{G}_{\xi, \zeta}=\mathfrak{L}$ for any $\xi\leqslant \zeta\leqslant \omega_1$. \item $\mathfrak{M}_{\xi, \zeta}=\mathfrak{L}$ if $\min\{\xi, \zeta\}=0$. \item For $\zeta\leqslant \alpha\leqslant \omega_1$ and $\beta\leqslant \xi\leqslant \omega_1$, $\mathfrak{G}_{\xi, \zeta}\subset \mathfrak{G}_{\beta, \alpha}$. \item If $\alpha\leqslant \zeta\leqslant \omega_1$ and $\beta\leqslant \xi\leqslant \omega_1$, then $\mathfrak{M}_{\xi, \zeta}\subset \mathfrak{M}_{\beta, \alpha}$. \end{enumerate}

\label{avp}
\end{rem}

\begin{corollary} For any $0\leqslant \zeta,\xi\leqslant \omega_1$, $$\mathfrak{G}_{\xi, \zeta}\subset \bigcap_{\alpha<\omega_1} \mathfrak{G}_{\alpha+\xi, \alpha+\zeta}.$$

\label{addition}
\end{corollary}

\begin{proof} Suppose $X,Y$ are Banach spaces, $A:X\to Y$ is an operator,  $\alpha<\omega_1$, and $0\leqslant \zeta, \xi\leqslant\omega_1$ are such that $A\in \complement \mathfrak{G}_{\alpha+\xi, \alpha+\zeta}$.  Then there exists a sequence $(x_n)_{n=1}^\infty\subset X$ which is $\alpha+\xi$-weakly null and such that $(Ax_n)_{n=1}^\infty$ is not $\alpha+\zeta$-weakly null. Note that $\zeta<\omega_1$, since otherwise $\alpha+\zeta=\alpha+\omega_1=\omega_1$, and $(Ax_n)_{n=1}^\infty$ would be a non-weakly null image of a weakly null sequence.   If $\xi=\omega_1$, we deduce that $A\in\complement \mathfrak{G}_{\xi, \zeta}$, since $(x_n)_{n=1}^\infty$ is a $\xi$-weakly null sequence the image of which under $A$ is not $\alpha+\zeta$-weakly null, and therefore not $\zeta$-weakly null.   If $\xi<\omega_1$, we use Corollary \ref{block of block} to deduce the existence of some convex blocking $(z_n)_{n=1}^\infty$ of $(x_n)_{n=1}^\infty$ which is $\xi$-weakly null and the image of which under $A$ is an $\ell_1^\zeta$-spreading model. Thus $A\in \complement \mathfrak{G}_{\xi, \zeta}$.  Therefore $\complement \mathfrak{G}_{\alpha+\xi, \alpha+\zeta}\subset \complement \mathfrak{G}_{\xi, \zeta}$.   Taking complements and noting that $\alpha<\omega_1$ was arbitrary, we are done.

\end{proof}

\begin{rem}\upshape We remark that adding $\alpha$ on the left in the previous corollary is necessary. The analogous statement fails if we try to add $\alpha$ on the right.  For example, for any $0\leqslant \xi<\omega_1$ and $\zeta<\omega^\xi$, the formal identity $I:X_{\omega^\xi}\to X_\zeta$ lies in $\textsf{G}_{\omega^\xi, 0}\cap \complement \textsf{G}_{\omega^\xi+1, \zeta}$.

\end{rem}

\subsection{Examples}

In this subsection, we provide examples to show the richness of the classes of interest, $\mathfrak{wBS}_\xi$, $\mathfrak{G}_{\xi, \zeta}$, and $\mathfrak{M}_{\xi, \zeta}$.  We note that $\mathfrak{wBS}_0=\mathfrak{V}$, $\mathfrak{G}_{\xi, \zeta}=\mathfrak{L}$ whenever $\xi\leqslant \zeta$, and $\mathfrak{M}_{\xi, \zeta}=\mathfrak{L}$ whenever $\min\{\xi, \zeta\}=0$. We typically omit reference to these trivial cases.

\begin{proposition} Fix $0<\xi<\omega_1$.  Then for any subset $S$ of $[0, \xi)$ with $\sup S=\xi$,  $(\oplus_{\zeta\in S}X_\zeta)_{\ell_1(S)}\in \mathfrak{wBS}_\xi\cap \cup_{\zeta<\xi}\complement\mathfrak{wBS}_\zeta$. 
\label{wbsx}
\end{proposition}

\begin{proof} By Theorem \ref{sbfacts}$(v)$, if $\zeta<\xi$, $X_\zeta\in \textsf{wBS}_\xi$. We will prove in Proposition \ref{bt} that the $\ell_1$ direct sum of members of $\textsf{wBS}_\xi$ also lies in $\textsf{wBS}_\xi$.

\end{proof}

\begin{theorem} For $0\leqslant \zeta<\xi<\omega_1$, the formal inclusion $I:X_\xi\to X_\zeta$ lies in $\mathfrak{G}_{\xi, \zeta}\cap \complement \mathfrak{G}_{\xi+1, \zeta}$. 

\label{gx}
\end{theorem}

\begin{proof} Fix $(x_n)_{n=1}^\infty \subset X_\xi$ $\xi$-weakly null.   Then by Theorem \ref{shinsuke}$(v)$, $\lim_n \|x_n\|_\beta=0$ for every $\beta<\lambda(\xi)$. If $\zeta=0$, then $\zeta<\lambda(\xi)$ and $\lim_n \|x_n\|_\zeta=0$.  Therefore $(Ix_n)_{n=1}^\infty$ is $\zeta$-weakly null.  If $\zeta>0$, then since $\lambda(\zeta)\leqslant \lambda(\xi)$, $\lim_n \|Ix_n\|_\beta=0$ for every $\beta<\lambda(\zeta)$, and Theorem \ref{shinsuke}$(v)$ yields that $(Ix_n)_{n=1}^\infty$ is $\zeta$-weakly null in this case.  In either case, $(Ix_n)_{n=1}^\infty$ is $\zeta$-weakly null, and $I\in \mathfrak{G}_{\xi, \zeta}$. However, the canonical basis is $\xi+1$-weakly null in $X_\xi$ and not $\zeta$-weakly null in $X_\zeta$, so $I\in \complement \mathfrak{G}_{\xi+1, \zeta}$.

\end{proof}

It is well-known and obvious that every Schur space and every space whose dual is a Schur space has the Dunford-Pettis property.  The generalization of this fact to operators is $\mathfrak{V}, \mathfrak{V}^\text{dual}\subset \mathfrak{DP}$.  The ordinal analogues are also obvious: For any $0<\xi\leqslant \omega_1$, $\mathfrak{V}_\xi\subset \mathfrak{M}_{\xi, \omega_1}$ and $\mathfrak{V}_\xi^\text{dual}\subset \mathfrak{M}_{\omega_1, \xi}$.  Thus it is of interest to come up with examples of members of $\mathfrak{M}_{\xi, \omega_1}$, or more generally $\mathfrak{M}_{\xi, \zeta}$, which do not come from $\mathfrak{V}_\xi$ or $\mathfrak{V}_\zeta^\text{dual}$.

\begin{theorem} For $0<\xi<\omega_1$ and $1<p\leqslant \infty$,  the formal inclusion $I:X_{\xi,p}\to X_{\lambda(\xi)}$ lies in $\mathfrak{M}_{\xi, \omega_1}\cap \complement \mathfrak{M}_{\xi+1, 1}\cap \complement \mathfrak{V}_\xi$ and the formal inclusion $J:X_{\lambda(\xi)}^*\to X_{\xi,p}^*$ lies in $\mathfrak{M}_{\omega_1, \xi}\cap \complement \mathfrak{M}_{1, \xi+1}\cap \complement \mathfrak{V}_\xi^\text{\emph{dual}}$.  

\label{dpx}
\end{theorem}

\begin{proof} It follows from Lemma \ref{double dual}$(ii)$ that $I\in \mathfrak{M}_{\xi, \omega_1}$ and $J\in \mathfrak{M}_{\omega_1, \xi}$.  Since the canonical basis of $X_{\xi,p}\subset X_{\xi,p}^{**}$ is $\xi+1$-weakly null and the canonical basis of $X^*_{\lambda(\xi)}$ is a $c_0^1$-spreading model, and therefore $1$-weakly null, $I\in \complement \mathfrak{M}_{\xi+1, 1}$ and $J\in \complement \mathfrak{M}_{1, \xi+1}$.   Now if $(\gamma_k)_{k=1}^\infty \subset [0, \lambda(\xi))$ is such that $[0, \lambda(\xi))=\{\gamma_k:k\in\nn\}$, we may select $F_1<F_2<\ldots$, $F_i\in \mathcal{S}_{\lambda(\xi)}$,  and positive scalars $(a_i)_{i\in \cup_{n=1}^\infty F_i}$ such that for each $1\leqslant k\leqslant n$, $\sum_{i\in F_n}a_i=1$ and $\|\sum_{i\in F_n}a_ie_i\|_{\gamma_k}<1/n$.    Then with $x_n=\sum_{i\in F_n}a_ie_i$, Theorem \ref{shinsuke}$(v)$ yields that $(x_n)_{n=1}^\infty$ is $\xi$-weakly null in $X_{\xi,p}\subset X_{\xi,p}^{**}$. Evidently $(x_n)_{n=1}^\infty$ is normalized in $X_{\lambda(\xi)}$, whence $I\in \complement \mathfrak{V}_\xi$ and $J\in \complement \mathfrak{V}_\xi^\text{dual}$.

\end{proof}

\begin{corollary}  For any $0\leqslant \alpha, \beta, \zeta, \xi\leqslant \omega_1$, $\mathfrak{G}_{\beta, \alpha}= \mathfrak{G}_{\xi, \zeta}$ if and only if one of the two exclusive conditions holds: \begin{enumerate}[(i)]\item $\xi\leqslant \zeta$ and $\beta\leqslant \alpha$ (in which case $\mathfrak{G}_{\beta, \alpha}=\mathfrak{L}=\mathfrak{G}_{\xi, \zeta}$).  \item $\alpha= \zeta<\xi= \beta$.  \end{enumerate}

\label{gcor}
\end{corollary}

\begin{proof} It is obvious that $(i)$ and $(ii)$ are exclusive and either implies equality. Now suppose that neither $(i)$ nor $(ii)$ holds. Suppose $\xi\leqslant \zeta$ and $\beta>\alpha$. Then $I_{X_\alpha}\in \mathfrak{L}\cap \complement \mathfrak{G}_{\beta, \alpha}=\mathfrak{G}_{\xi, \zeta}\cap \complement \mathfrak{G}_{\beta, \alpha}$, and $\mathfrak{G}_{\xi, \zeta}\neq \mathfrak{G}_{\beta, \alpha}$.  Similarly, $\mathfrak{G}_{\xi, \zeta}\neq \mathfrak{G}_{\beta, \alpha}$ if $\beta\leqslant \alpha$ and $\zeta<\xi$.    

For the remainder of the proof, suppose that $\alpha<\beta$ and $\zeta<\xi$.     Now suppose $\alpha<\zeta$.  Then $$I_{X_\alpha}\in \mathfrak{wBS}_{\alpha+1}\cap \complement \mathfrak{wBS}_\alpha\subset \mathfrak{G}_{\xi, \zeta}\cap \complement \mathfrak{G}_{\beta, \alpha}.$$   Similarly, $\mathfrak{G}_{\xi, \zeta}\neq \mathfrak{G}_{\beta, \alpha}$ if $\zeta<\alpha$.  Next assume $\zeta=\alpha< \xi<\beta$. Then if $I:X_\xi\to X_\zeta$ is the formal inclusion, $I\in \mathfrak{G}_{\xi, \zeta}\cap \complement \mathfrak{G}_{\beta, \alpha}$.   If $\zeta=\alpha<\beta<\xi$, we argue similarly with the inclusion $I:X_\beta\to X_\alpha$.  Since this is a complete list of the possible ways for $(i)$ and $(ii)$ to simultaneously fail, we are done.

\end{proof}

\begin{corollary} For any $0\leqslant \alpha, \beta, \zeta, \xi\leqslant \omega_1$, $\mathfrak{M}_{\beta, \alpha}\subset \mathfrak{M}_{\xi, \zeta}$ if and only if one of the two exclusive conditions holds: \begin{enumerate}[(i)]\item $0=\min\{\zeta, \xi\}$ (in which case $\mathfrak{M}_{\beta, \alpha}=\mathfrak{L}=\mathfrak{M}_{\xi, \zeta}$). \item $0<\zeta\leqslant \alpha$ and $0<\xi\leqslant \beta$. \end{enumerate}

In particular, $\mathfrak{M}_{\beta, \alpha}=\mathfrak{M}_{\xi, \zeta}$ if and only if $\min \{\beta, \alpha\}=0=\min \{\xi, \zeta\}$ or $0<\alpha=\zeta$ and $0<\beta=\xi$. 

\label{dpcor}
\end{corollary}

\begin{proof} It is obvious that $(i)$ and $(ii)$ are exclusive, and either implies that $\mathfrak{M}_{\beta, \alpha}\subset \mathfrak{M}_{\xi, \zeta}$.

Now assume that $\min \{\zeta, \xi\}>0$.   If $\min \{\alpha, \beta\}=0$, $\mathfrak{M}_{\beta, \alpha}=\mathfrak{L}\not\subset \mathfrak{M}_{\xi, \zeta}$, since $I_{\ell_2}\in \complement \mathfrak{M}_{1,1}\subset \complement \mathfrak{M}_{\xi, \zeta}$.    If $0<\alpha, \beta$ and $\beta<\xi$, then let $I:X_\beta\to X_{\lambda(\beta)}$ be the formal inclusion.  Then $$I\in \mathfrak{M}_{\beta, \omega_1}\cap \complement \mathfrak{M}_{\beta+1, 1}\subset \mathfrak{M}_{\beta, \alpha}\cap \complement \mathfrak{M}_{\xi, \zeta}.$$   Now if $0<\alpha, \beta$ and $\alpha<\zeta$, let $J:X^*_{\lambda(\alpha)}\to X^*_\alpha$ be the formal inclusion.  Then $$J\in \mathfrak{M}_{\omega_1, \alpha}\cap \complement \mathfrak{M}_{1, \alpha+1}\subset \mathfrak{M}_{\beta, \alpha}\cap \complement \mathfrak{M}_{\xi, \zeta}.$$  

The last statement follows from the fact that if $\mathfrak{M}_{\beta, \alpha}=\mathfrak{M}_{\xi, \zeta}$, then either both classes must equal $\mathfrak{L}$, which happens if and only if $\min \{\beta, \alpha\}=0=\min \{\xi, \zeta\}$, or neither class is $\mathfrak{L}$, in which case $\min \{\beta, \alpha\}, \min \{\xi, \zeta\}>0$.   In the latter case, using the previous paragraph and symmetry, $\alpha=\zeta$ and $\beta=\xi$.

\end{proof}

\subsection{General properties}

We will need the following fact, shown in \cite{CN}.

\begin{proposition} If $X$  is a Banach space and $(x_n)_{n=1}^\infty\subset X$ is $\xi$-weakly null, then there exists a subsequence $(x_{n_i})_{i=1}^\infty$ of $(x_n)_{n=1}^\infty$ such that the operator $\Phi:\ell_1\to X$ given by $\Phi\sum_{i=1}^\infty a_ie_i=\sum_{i=1}^\infty a_ix_{n_i}$ lies in $\mathfrak{W}_\xi(\ell_1, X)$. 

\label{CN}
\end{proposition}

\begin{rem}\upshape It follows that if $(y^*_n)_{n=1}^\infty$ is $\xi$-weakly null, there exist a subsequence $(y^*_{n_i})_{i=1}^\infty$ of $(y^*_n)_{n=1}^\infty$ such that the operator given by $\Phi_*:Y\to c_0$ given by $\Phi_* y=(y^*_{n_i}(y))_{i=1}^\infty$ lies in $\mathfrak{W}_\xi^\text{dual}(Y, c_0)$. This follows immediately from Proposition \ref{CN}, since $\Phi_*^*:\ell_1\to Y^*$ is given by $\Phi_*^*\sum_{i=1}^\infty a_ie_i=\sum_{i=1}^\infty a_i y^*_{n_i}$.

\end{rem}

\begin{rem}\upshape In the following results, we will repeatedly use the fact that a weakly null $\ell_1^\zeta$-spreading model can have no $\zeta$-convergent subsequence.

\end{rem}

\begin{theorem}  Fix $0\leqslant \zeta<\xi\leqslant \omega_1$. Then $$\mathfrak{G}_{\xi, \zeta}= \mathfrak{W}_\zeta\circ \mathfrak{W}_\xi^{-1}$$ and $$\mathfrak{G}_{\xi, \zeta}^{\text{\emph{dual}}}= (\mathfrak{W}_\xi^{\text{\emph{dual}}})^{-1}\circ \mathfrak{W}_\zeta^\text{\emph{dual}}.$$    Consequently, $\mathfrak{G}_{\xi, \zeta}$ is a closed, two-sided ideal containing all compact operators. Moreover, $\mathfrak{G}_{\xi, \zeta}$ is injective but not surjective. Finally, $$\mathfrak{G}_{\xi, \zeta}^{\text{\emph{dual\ dual}}}\subsetneq \mathfrak{G}_{\xi, \zeta},$$ while neither of $\mathfrak{G}_{\xi, \zeta}$, $\mathfrak{G}_{\xi, \zeta}^\text{\emph{dual}}$ is contained in the other. 

\label{downshift theorem}
\end{theorem}

\begin{proof} Fix $X,Y\in \textbf{Ban}$ and $A\in \mathfrak{L}(X,Y)$.   First suppose that $A\in \mathfrak{G}_{\xi, \zeta}(X,Y)$.   Fix a Banach space $W$ and $R\in \mathfrak{W}_\xi(W,X)$.   Fix a bounded sequence $(w_n)_{n=1}^\infty$. By passing to a subsequence, we may assume there exists $x\in X$ such that $(x-Rw_n)_{n=1}^\infty$ is $\xi$-weakly null, whence $(Ax-ARw_n)_{n=1}^\infty$ is $\zeta$-weakly null.  Since this holds for an arbitrary bounded sequence in $(w_n)_{n=1}^\infty$, $AR\in \mathfrak{W}_\zeta$. Since $W\in \textbf{Ban}$ and $R\in \mathfrak{W}_\xi(W,X)$ were arbitrary, $A\in \mathfrak{W}_\zeta \circ \mathfrak{W}_\xi^{-1}(X,Y)$.

Now suppose that $A\in\complement \mathfrak{G}_{\xi, \zeta}$. Then there exists a $\xi$-weakly null sequence $(x_n)_{n=1}^\infty$ in $X$ such that $(Ax_n)_{n=1}^\infty$ is an $\ell_1^\zeta$-spreading model.     By Proposition \ref{CN}, after passing to a subsequence and relabeling, we may assume the operator $R:\ell_1\to X$ given by $R\sum_{i=1}^\infty a_ie_i=\sum_{i=1}^\infty a_ix_i$ lies in $\mathfrak{W}_\xi(\ell_1, X)$. But since $(ARe_i)_{i=1}^\infty=(Ax_i)_{i=1}^\infty$ has no $\zeta$-convergent subsequence, $A\in \complement \mathfrak{W}_\zeta\circ \mathfrak{W}_\xi^{-1}(X,Y)$.

Next, suppose that $A\in \mathfrak{G}_{\xi, \zeta}^\text{dual}(X,Y)$.     Fix $Z\in \textbf{Ban}$ and an operator $L\in \mathfrak{W}_\xi^\text{dual}(Y,Z)$.  Then $A^*\in \mathfrak{G}_{\xi, \zeta}(Y^*, X^*)= \mathfrak{W}_\zeta\circ \mathfrak{W}_\xi^{-1}(Y^*, X^*)$ and $L^*\in \mathfrak{W}_\xi(Z^*, Y^*)$, whence $(LA)^*=A^*L^*\in \mathfrak{W}_\zeta(Z^*, X^*)$. Thus $LA\in \mathfrak{W}_\zeta^\text{dual}(X,Z)$.    Since this holds for any $Z\in \textbf{Ban}$ and $L\in \mathfrak{W}_\xi^\text{dual}(Y,Z)$, $A\in (\mathfrak{W}_\xi^\text{dual})^{-1}\circ \mathfrak{W}_\zeta^\text{dual}(X,Y)$.

Now if $A\in \complement \mathfrak{G}_{\xi, \zeta}^\text{dual}(X,Y)$, there exists $(y^*_n)_{n=1}^\infty \subset Y^*$ which is $\xi$-weakly null and $(A^*y^*_n)_{n=1}^\infty$ is an $\ell_1^\zeta$-spreading model. By the remarks preceding the theorem, by passing to a subsequence and relabeling,  we may assume the operator $L:Y\to c_0$ given by $Ly=(y^*_n(y))_{n=1}^\infty$ lies in $\mathfrak{W}_\xi^\text{dual}(Y, c_0)$.   But since $(A^*L^*e_i)_{i=1}^\infty=(A^*y^*_i)_{i=1}^\infty$ is a weakly null $\ell_1^\zeta$-spreading model,  $(LA)^*=A^*L^*\in \complement \mathfrak{W}_\zeta(\ell_1, X^*)$. Thus $LA\in \complement ((\mathfrak{W}_\xi^\text{dual})^{-1}\circ \mathfrak{W}_\zeta^\text{dual})(X,Y)$.

This yields the first two equalities.  It follows from the fact that $\mathfrak{W}_\zeta, \mathfrak{W}_\xi$ are closed, two-sided ideals containing the compact operators that $\mathfrak{G}_{\xi, \zeta}$ is also.

It is evident that $\mathfrak{G}_{\xi, \zeta}$ is injective, since a given sequence is $\zeta$-weakly null if and only if its image under some (equivalently, every) isomorphic image of that sequence is $\zeta$-weakly null.    The ideal $\mathfrak{G}_{\xi, \zeta}$ is not surjective, since $X_\zeta\in \complement \textsf{G}_{\xi, \zeta}$, while $X_\zeta$ is a quotient of $\ell_1\in \textsf{V}\subset \textsf{G}_{\xi, \zeta}$.

It is also easy to see that if $A^{**}\in \mathfrak{G}_{\xi, \zeta}$, then $A\in \mathfrak{G}_{\xi, \zeta}$, whence $\mathfrak{G}_{\xi, \zeta}^{\text{dual\ dual}}\subset \mathfrak{G}_{\xi, \zeta}$.   If $\zeta=0$, note that $\ell_1\in \textsf{V}\subset \textsf{G}_{\xi, \zeta}$, but $\ell_1^{**}$ contains an isomorphic copy of $\ell_2$, whence $\ell_1^{**}\in \complement \textsf{G}_{\xi, 0}$. This yields that $\mathfrak{G}_{\xi, 0}^\text{dual\ dual}\neq \mathfrak{G}_{\xi,0}$.   Now if $\zeta>0$, $c_0\in \textsf{wBS}_1\subset \textsf{G}_{\xi, \zeta}$. But $\ell_\infty=c_0^{**}\in \complement \textsf{G}_{\xi, \zeta}$.  In order to see that $\ell_\infty\in \complement \textsf{G}_{\xi, \zeta}$, simply note that $\ell_\infty$ contains a sequence equivalent to the $X_\zeta$ basis, which is $\xi$-weakly null and not $\zeta$-weakly null.

Finally, let us note that if $\zeta=0$, $\ell_1\in \textsf{V}\subset \textsf{G}_{\xi, \zeta}$, while $c_0, \ell_\infty\in \complement \textsf{G}_{\xi, 0}$.   Thus neither of $\mathfrak{G}_{\xi,0}, \mathfrak{G}_{\xi, 0}^\text{dual}$ is contained in the other. Now suppose that $\zeta>0$.     Then since $X^*_{\zeta, 2}\in \textsf{wBS}_1\subset \textsf{G}_{\xi, \zeta}$, $$X_{\zeta, 2}\in \textsf{G}_{\xi, \zeta}^\text{dual}\cap \complement\textsf{G}_{\xi, \zeta}$$ and $$X_{\zeta, 2}^*\in \textsf{G}_{\xi, \zeta}\cap \complement \textsf{G}_{\xi, \zeta}^\text{dual}.$$  Here we recall that $X_{\zeta, 2}$ is reflexive.  This yields that if $0<\zeta<\xi\leqslant \omega_1$, neither of $\mathfrak{G}_{\xi, \zeta}, \mathfrak{G}_{\xi, \zeta}^\text{dual}$ is contained in the other.

\end{proof}

\begin{theorem}  Fix $0<\zeta,\xi\leqslant \omega_1$. Then $$\mathfrak{M}_{\xi, \zeta}=(\mathfrak{W}_\zeta^\text{\emph{dual}})^{-1}\circ \mathfrak{V}_\xi=(\mathfrak{W}_\zeta^\text{\emph{dual}})^{-1}\circ \mathfrak{K}\circ \mathfrak{W}_\xi^{-1}.$$   Consequently, $\mathfrak{M}_{\xi, \zeta}$ is a closed, two-sided ideal containing all compact operators.  Moreover, $\mathfrak{M}_{\xi, \zeta}$ is neither injective nor surjective. Finally, $$\mathfrak{M}_{\xi, \zeta}^\text{\emph{dual}}\subsetneq \mathfrak{M}_{\zeta, \xi}$$ and   $$\mathfrak{M}_{\xi, \zeta}^\text{\emph{dual\ dual}}\subsetneq \mathfrak{M}_{\xi, \zeta}.$$   

\label{dp theorem}
\end{theorem}

\begin{proof} It follows from the fact that $\mathfrak{V}_\xi=\mathfrak{K}\circ \mathfrak{W}_\xi^{-1}$, which was shown in \cite{CN}, that $(\mathfrak{W}_\zeta^\text{dual})^{-1}\circ \mathfrak{V}_\xi= (\mathfrak{W}_\zeta^\text{dual})^{-1}\circ \mathfrak{K}\circ \mathfrak{W}_\xi^{-1}$.     We will show that $\mathfrak{M}_{\xi, \zeta}=(\mathfrak{W}_\zeta^\text{dual})^{-1}\circ \mathfrak{K}\circ \mathfrak{W}_\xi^{-1}$.   To that end, fix Banach spaces $X,Y$ and $A\in \mathfrak{L}(X,Y)$.

Suppose that $A\in \mathfrak{L}(X,Y)$. Fix Banach spaces $W,Z$ and operators $R\in \mathfrak{W}_\xi(W,X)$ and $L\in \mathfrak{W}_\zeta^\text{dual}(Y,Z)$.    We will show that $LAR\in \mathfrak{K}(W,Z)$.  Seeking a contradiction, suppose $LAR\in \complement \mathfrak{K}$. Note that there exists a bounded sequence $(w_n)_{n=1}^\infty\subset W$ such that $\inf_{m\neq n} \|LAR w_m-LARw_n\|\geqslant 4$.    By passing to a subsequence, we may assume there exist $x\in X$ such that $(x-Rw_n)_{n=1}^\infty$ is $\xi$-weakly null.   Since $\|LARw_m-LARw_n\|\geqslant 4$ for all $m\neq n$, there is at most one $n\in\nn$ such that $\|LAx-LARw_n\|<2$.   By passing to a subsequence, we may assume $\|LAx-LARw_n\|\geqslant 2$ for all $n\in\nn$. For each $n\in\nn$, fix $z^*_n\in B_{Z^*}$ such that $|z^*_n(LAx-LARw_n)|\geqslant 2$.  By passing to a subsequence one final time, we may assume there exists $y^*\in Y^*$ such that $(y^*-L^*z^*_n)_{n=1}^\infty$ is $\zeta$-weakly null and, since $(Ax-ARw_n)_{n=1}^\infty$ is weakly null, $|y^*(Ax-ARw_n)|<1$ for all $n\in\nn$.     Then $(y^*-L^*z^*_n)_{n=1}^\infty\subset Y^*$ is $\zeta$-weakly null, $(x-Rw_n)_{n=1}^\infty$ is $\xi$-weakly null, and $$\inf_n |(y^*-L^*z^*_n)(Ax-ARw_n)| \geqslant \inf_n |L^*z^*_n(Ax-ARw_n)|- 1 =\inf_n |z^*_n(LAx-LARw_n)|-1 \geqslant 1.$$   This contradiction yields that  $\mathfrak{M}_{\xi, \zeta}\subset (\mathfrak{W}_\zeta^\text{dual})^{-1}\circ \mathfrak{K}\circ \mathfrak{W}_\xi^{-1}$.    

Now suppose that $A\in \complement \mathfrak{M}_{\xi, \zeta}(X,Y)$.   Then there exist a $\xi$-weakly null sequence $(x_n)_{n=1}^\infty\subset X$ and a $\zeta$-weakly null sequence $(y^*_n)_{n=1}^\infty\subset Y^*$ such that $\inf_n |y^*_n(Ax_n)|=1$.   Using Proposition \ref{CN} and the remark following it, after passing to subsequences twice and relabling, we may assume the operators $R:\ell_1\to X$ given by $R\sum_{i=1}^\infty a_ie_i=\sum_{i=1}^\infty a_ix_i$ and $L:Y\to c_0$ given by $Ly=(y^*_n(y))_{n=1}^\infty$ lie in $\mathfrak{W}_\xi(\ell_1, X)$ and $\mathfrak{W}_\zeta^\text{dual}(Y, c_0)$, respectively.  But $LAR:\ell_1\to c_0$ is not compact, since $|e^*_n(LARe_n)|= |y^*_n(Ax_n)|\geqslant 1$ for all $n\in\nn$.   This yields that $\mathfrak{M}_{\xi, \zeta}= (\mathfrak{W}_\zeta^\text{dual})^{-1}\circ \mathfrak{K}\circ \mathfrak{W}_\xi^{-1}$.

Since $\ell_2\in \complement \textsf{M}_{1,1}\subset \complement \textsf{M}_{\xi, \zeta}$ is a subspace of $\ell_\infty\in \textsf{M}_{\omega_1, \omega_1}\subset \textsf{M}_{\xi, \zeta}$ and a quotient of $\ell_1\in \textsf{M}_{\omega_1, \omega_1}\subset \textsf{M}_{\xi, \zeta}$, $\mathfrak{M}_{\xi, \zeta}$ is neither injective nor surjective. 

Now suppose $A\in \mathfrak{M}_{\xi, \zeta}^\text{dual}(X,Y)$.  Now if $(x_n)_{n=1}^\infty \subset X$ is $\zeta$-weakly null,  $(y^*_n)_{n=1}^\infty$ is $\xi$-weakly null, and $j:X\to X^{**}$ is the canonical embedding, then $(jx_n)_{n=1}^\infty\subset X^{**}$ is $\zeta$-weakly null. Since $A\in \mathfrak{M}_{\xi, \zeta}^\text{dual}(X,Y)$, $$\lim_n y^*_n(Ax_n)=\lim_n A^*y^*_n(x_n)= \lim_n jx_n(A^*y^*_n)=0.$$   Thus $A\in \mathfrak{M}_{\zeta, \xi}(X,Y)$. This yields that $\mathfrak{M}_{\xi, \zeta}^\text{dual}\subset \mathfrak{M}_{\zeta, \xi}$.   To see that $\mathfrak{M}_{\xi, \zeta}^\text{dual}\neq \mathfrak{M}_{\zeta, \xi}$, we cite Stegall's example \cite{Stegall}, $X=\ell_1(\ell_2^n)$.  This space has the Schur property, and therefore lies in $\textsf{M}_{\omega_1, \omega_1}\subset \textsf{M}_{\zeta, \xi}$, while $X^*$ contains a complemented copy of $\ell_2$.  Thus $X\in \complement \textsf{M}_{1,1}^\text{dual}\subset \complement \textsf{M}_{\xi, \zeta}^\text{dual}$.

Next, we note that $$\mathfrak{M}_{\xi, \zeta}^\text{dual\ dual} =(\mathfrak{M}_{\xi, \zeta}^\text{dual})^\text{dual} \subset \mathfrak{M}_{\zeta, \xi}^\text{dual}\subset \mathfrak{M}_{\xi, \zeta}.$$  To see that $\mathfrak{M}_{\xi, \zeta}^\text{dual\ dual}\neq \mathfrak{M}_{\xi, \zeta}$, we make yet another appeal to Stegall's example and let $Y=c_0(\ell_2^n)$.  Then $Y^*=X$ has the Schur property, and therefore $Y\in \textsf{M}_{\omega_1, \omega_1}\subset \textsf{M}_{\xi, \zeta}$. But  $Y^{**}=X^*\in \complement \textsf{M}_{1,1}\subset \complement \textsf{M}_{\xi, \zeta}$.    Therefore $Y\in \complement \textsf{M}_{\xi, \zeta}^\text{dual\ dual}$.

\end{proof}

\subsection{Direct sums}

For $1\leqslant p\leqslant \infty$ and classes $\mathfrak{I}, \mathfrak{J}$, we say $\mathfrak{J}$ is \emph{closed under} $\mathfrak{I}$-$\ell_p$ \emph{sums} provided that for any set $I$ and any collection $(A_i:X_i\to Y_i)_{i\in I}\subset \mathfrak{I}$ such that $\sup_{i\in I}\|A_i\|<\infty$, the operator $A:(\oplus_{i\in I}X_i)_{\ell_p(I)}\to (\oplus_{i\in I} Y_i)_{\ell_p(I)}$ lies in $\mathfrak{J}$.  The notion of an ideal being closed under $\mathfrak{I}$-$c_0$ sums is defined similarly.

We will use the following well-known fact about weakly null sequences in $\ell_1$ sums of Banach spaces. 

\begin{fact} Let $I$ be a set, $(X_i)_{i\in I}$ a collection of Banach spaces, and $(x_n)_{n=1}^\infty=\bigl((x_{i,n})_{i\in I}\bigr)_{n=1}^\infty$ a weakly null sequence in $(\oplus_{i\in I}X_i)_{\ell_1(I)}$.  Then for any $\ee>0$, there exists a  subset $J\subset I$ such that $|I\setminus J|<\infty$ and  for all $n\in\nn$, $\sum_{i\in J}\|x_{i,n}\|<\ee$. 

\label{fact}

\end{fact}

\begin{proposition} Fix $0\leqslant \zeta<\xi\leqslant \omega_1$. \begin{enumerate}[(i)]\item The class $\mathfrak{G}_{\xi, \zeta}$ is closed under $\mathfrak{G}_{\xi, \zeta}$-$\ell_1$ sums. \item  The class $\mathfrak{G}_{\xi, \zeta}$ is closed under $\mathfrak{G}_{\xi, \zeta}$-$\ell_p$ sums for $1<p<\infty$ if and only if $\zeta>0$. \item  The class $\mathfrak{G}_{\xi, \zeta+1}$ is  closed under $\mathfrak{G}_{\xi, \zeta}$-$c_0$ sums.   \item The class $\mathfrak{G}_{\xi, \zeta}$ is not closed under $\mathfrak{G}_{\xi, \zeta}$-$c_0$ sums. \item The class $\mathfrak{G}_{\xi, \zeta}$ is not closed under $\mathfrak{V}$ sums.  \end{enumerate}

\label{bt}

\end{proposition}

\begin{proof} Throughout, let $I$ be a set, $(A_i:X_i\to Y_i)_{i\in I}$ a collection of operators such that $\sup_{i\in I}\|A_i\|=1$. Let $X_p=(\oplus_{i\in I}X_i)_{\ell_p(I)}$, $Y_p=(\oplus_{i\in I}Y_i)_{\ell_p(I)}$, and $A_p:X_p\to Y_p$ the operator such that $A_p|_{X_i}=A_i$.   As usual, $p=0$ will correspond to the $c_0$ direct sum.

$(i)$ Assume $A_i\in \mathfrak{G}_{\xi, \zeta}$ for all $i\in I$. Fix $(x_n)_{n=1}^\infty \subset X_1$ $\xi$-weakly null.  Write $x_n=(x_{i,n})_{i\in I}$ and note that for each $i\in I$, $(x_{i,n})_{n=1}^\infty$ is $\xi$-weakly null, so $(A_ix_{i,n})_{i=1}^\infty$ is $\zeta$-weakly null.   Fix $\ee>0$ and $M\in[\nn]$.    Using Fact \ref{fact}, there exists a subset $J$ of $I$ such that $|I\setminus J|<\infty$ and $\sup_n \sum_{i\in J}\|x_{i,n}\|<\ee/2$.   Since $(A_ix_{i,n})_{n=1}^\infty$ is $\zeta$-weakly null, then there exists $F\in \mathcal{S}_\zeta\cap [M]^{<\nn}$ and positive scalars $(a_n)_{n\in F}$ summing to $1$ such that for each $i\in I\setminus J$, $\|\sum_{n\in F} a_nx_{i,n}\|_{Y_i}< \frac{\ee/2}{1+|I\setminus J|}$.   Then $$\|A_1 \sum_{n\in F}a_nx_n\| \leqslant \sum_{i\in I\setminus J} \|\sum_{n\in F} a_nA_ix_{i,n}\|_{Y_i} + \sum_{n\in F}a_n\sum_{i\in I\setminus J}\|x_{i,n}\|_{X_i}<\ee/2+\ee/2=\ee.$$  Since $\ee>0$ and $M\in[\nn]$ were arbitrary, $(A_1x_n)_{n=1}^\infty$ is $\zeta$-weakly null.

$(ii)$ Fix $1<p<\infty$. Since $\ell_p\in \complement \textsf{G}_{\xi, 0}$ and $\mathbb{K}\in \textsf{G}_{\xi,0}$, $\mathfrak{G}_{\xi,0}$ is not closed under $\ell_p$ sums. It follows by an inessential modification of work from \cite{BC} that for $0<\zeta<\omega_1$, $\mathfrak{G}_{\xi, \zeta}$ is closed under $\mathfrak{G}_{\xi, \zeta}$-$\ell_p$ sums.  More specifically, let $(x_n)_{n=1}^\infty \subset B_{X_p}$ be $\xi$-weakly null and let $v_n=(\|x_{i,n}\|_{X_i})_{i\in I}\in B_{\ell_p(I)}$. Assume $(A_px_n)_{n=1}^\infty$ satisfies $$0<\ee\leqslant \inf\{\|A_px\|: F\in \mathcal{S}_\zeta, x\in \text{co}(x_n: n\in F)\}.$$   By passing to a subsequence, we may assume $v_n\to v=(v_i)_{i\in I}\in B_{\ell_p(I)}$ weakly, and that $v_n$ is a small perturbation of $v+b_n$, where the sequence $(b_n)_{n=1}^\infty$ consists of disjointly supported vectors in $B_{X_p}$.  We may fix a subset $J$ of $I$ such that $|I\setminus J|<\infty$ and $\bigl(\sum_{i\in J}v_i^p\bigr)_{1/p}<\ee/3$.    For $k\in\nn$, we may first choose $M=(m_i)_{i=1}^\infty\in[\nn]$ such that $\mathcal{S}_\zeta[\mathcal{A}_k](M)\subset \mathcal{S}_\zeta$ and let $$u_n=\frac{1}{k}\sum_{j=nk+1}^{(n+1)k} x_{m_j}.$$  If $k$ was chosen sufficiently large, then $$\sup_n \bigl(\sum_{i\in J} \|u_{i,n}\|_{X_i}^p\bigr)^{1/p}<\ee/2.$$  By our choice of $M$, $(A_pu_n)_{n=1}^\infty$ also satisfies $$\ee\leqslant \inf\{\|Au_n \|:  F\in \mathcal{S}_\zeta, x\in \text{co}(x_n: n\in F)\}.$$     Since $(A_ix_{i,n})_{n=1}^\infty$ is $\zeta$-weakly null, there exist $F\in \mathcal{S}_\zeta$ and positive scalars $(a_n)_{n\in F}$ summing to $1$ such that for each $i\in I\setminus J$, $\|\sum_{n\in F}a_nA_ix_{i,n}\|_{Y_i}<\frac{\ee/2}{1+|I\setminus J|}$.  We reach a contradiction as in $(i)$.

$(iii)$ Fix $(x_n)_{n=1}^\infty =((x_{i,n})_{i\in I})_{n=1}^\infty\subset B_{X_0}$ $\xi$-weakly null.   Fix $(\ee_n)_{n=1}^\infty$ such that $\sum_{n=1}^\infty \ee_n<1$.   Since for each $i\in I$, $(A_i x_{i,n})_{n=1}^\infty$ is $\zeta$-weakly null, we may recursively select  $F_1<F_2<\ldots$, $F_n\in \mathcal{S}_\zeta$, positive scalars $(a_j)_{j\in \cup_{n=1}^\infty F_n}$, and finite subsets $\varnothing=I_0\subset  I_2\subset \ldots$ of $I$ such that for each $n\in\nn$, $$\sum_{j\in F_n}a_j=1,$$ $$\max_{i\in I_{n-1}} \|A_i \sum_{j\in F_n} a_jx_{i,j}\|<\ee_n,$$ and $$\max_{i\in I\setminus I_n} \|\sum_{j\in F_n}a_j x_{i,j}\|<\ee_n.$$    Then since for each $n\in\nn$, $\cup_{m=n+1}^{2n} F_m\in \mathcal{S}_{\zeta+1}$ for each $n\in\nn$, we deduce that \begin{align*} \sup_{i\in I} \|A_0\frac{1}{n}\sum_{m=n+1}^{2n} \sum_{j\in F_m} a_j x_{i,j}\| & \leqslant \max\Biggl\{\max_{i\in I\setminus I_{2n}} \sum_{m=n+1}^{2n} \|\sum_{j\in F_n}a_j x_{i,j}\|, \\ & \max_{n<m\leqslant 2n} \Bigl\{\max_{i\in I_m\setminus I_{m-1}} \frac{1}{n}\|A_i\sum_{j\in F_m} a_jx_{i,j}\|+\sum_{m\neq l=n+1}^{2n} \|A_i\sum_{j\in F_l} a_jx_{i,j}\|\Bigr\}\Biggr\} \\ & \leqslant  \frac{1}{n}+\sum_{m=n+1}^\infty \ee_m \underset{n\to \infty}{\to}0. \end{align*}

$(iv)$ For the $\zeta=0$ case, $c_0=c_0(\mathbb{K})$ yields that $\mathfrak{G}_{\xi, 0}$ is not closed under $\mathfrak{G}_{\xi, 0}$-$c_0$ sums.   If $\zeta=\mu+1$, let $\mathcal{F}_n=\mathcal{A}_n[\mathcal{S}_\mu]$ and note that $X_{\mathcal{F}_n}$ is isomorhpic to $X_\mu$.  If $\zeta$ is a limit ordinal, let $(\zeta_n)_{n=1}^\infty$ be the sequence defining $\mathcal{S}_\zeta$ and let $\mathcal{F}_n=\mathcal{S}_{\zeta_n+1}$. In either the successor or limit case, $\mathcal{S}_\zeta=\{E: \exists n\leqslant E\in \mathcal{F}_n\}$.  Also, in both cases, $X_{\mathcal{F}_n}\in \textsf{wBS}_\zeta\subset \textsf{G}_{\xi, \zeta}$ for all $n\in\nn$.   Let $x_n=(e_n, e_n, e_n, \ldots, e_n, 0, 0, \ldots)$, where  $(e_i)_{i=1}^\infty$ simultaneously denotes the basis of each $X_{\mathcal{F}_n}$ and $e_n$ appears $n$ times.   Now fix $\varnothing\neq G \in \mathcal{S}_\zeta$,  let $m=\min G$, and note that $G\in \mathcal{F}_m$.  Fix $(a_n)_{n\in G}$ and note that the $m^{th}$ term of the sequence $\sum_{n\in G}a_nx_n$ is $\sum_{n\in G}a_ne_n$, which has norm $\sum_{n\in G}|a_n|$ in $X_{\mathcal{F}_m}$.  Thus $(x_n)_{n=1}^\infty$ is a weakly null,  isometric $\ell_1^\zeta$-spreading model. By $(iii)$, $(x_n)_{n=1}^\infty$ is $\xi$-weakly null (more precisely, we are using the fact that $\textsf{wBS}_{\zeta+1}$, and therefore $\textsf{wBS}_\xi$,  is closed under $\textsf{wBS}_\zeta$-$c_0$ sums).

$(v)$ Let $E_n=[e_i:i\leqslant n]\subset X_{\zeta, 2}$, which lies in $\textsf{V}$.    But, analogously to Stegall's example,  $\ell_\infty(E_n)$ contains a complemented copy of $X_{\zeta, 2}$.   More precisely, let $Z$ denote the subspace of $\ell_\infty(E_n)$ consisting of those $z=(\sum_{i=1}^n a_{i,n}e_i)_{n=1}^\infty$ such that for all $m<n\in\nn$ and $1\leqslant i\leqslant m$, $a_{i,m}=a_{i,n}$ (that is, the sequences $(a_{i,n})_{i=1}^\infty$ are each initial segments of a single scalar sequence $(a_i)_{i=1}^\infty$).  For $x=\sum_{i=1}^\infty a_ie_i\in X_{\xi,2}$, let $j(x)=(\sum_{i=1}^n a_ie_i)_{n=1}^\infty$, which is an isometric embedding of $X_{\xi,2}$ into $Z$. Moreover, $j$ is onto. Indeed, since the basis of $X_{\xi,2}$ is boundedly-complete and if $z=(\sum_{i=1}^n a_i e_i)_{n=1}^\infty\in Z$, then $$\sup_n \|\sum_{i=1}^n a_ie_i\|_{\xi,2}=\|z\|_{\ell_\infty(E_n)}<\infty,$$ and $x:=\sum_{i=1}^\infty a_ie_i\in X_{\xi,2}$ is such that $j(x)=z$.  Thus $Z$ is isometrically isomorphic to $X_{\xi,2}$. Let $\mathcal{U}$ be a free ultrafilter on $\nn$ and for $z=(\sum_{i=1}^n a_{i,n}e_i)_{n=1}^\infty \in \ell_\infty(E_n)$, let $$Pz=\underset{n\in \mathcal{U}}{\text{weak lim}} \sum_{i=1}^n a_{i,n}e_i\in X_{\xi,2}.$$ This limit is well-defined, since $(\sum_{i=1}^n a_{i,n}e_i)_{n=1}^\infty$ is bounded in the reflexive space $X_{\xi,2}$.    Then $Z$ is an isometric copy of $X_{\xi,2}$ which is $1$-complemented in $\ell_\infty(E_n)$ via the map $jP$.

\end{proof}

\begin{proposition} Fix $0<\zeta, \xi\leqslant \omega_1$.  \begin{enumerate}[(i)]\item  The class $\mathfrak{M}_{\xi, \zeta}$ is closed under $c_0$ and $\ell_1$ sums. \item The class $\mathfrak{M}_{\xi, \zeta}$ is not closed under $\ell_p$ sums for any $1<p<\infty$. \item The class $\mathfrak{M}_{\xi, \zeta}$ is not closed under $\ell_\infty$ sums. \end{enumerate}

\end{proposition}

\begin{proof} Item $(i)$ follows from inessential modifications of the fact that the class of spaces with the Dunford-Pettis property are closed under $c_0$ and $\ell_1$ sums, using Fact \ref{fact}.

Item $(ii)$ follows from the fact that $\ell_p=\ell_p(\mathbb{K})$, $1<p<\infty$, does not lie in $\textsf{M}_{1, 1}$, while     $\mathbb{K}\in \textsf{V}$.

Item $(iii)$ again follows from Stegall's example, which is an $\ell_\infty$ sum of Schur spaces which contains a complemented copy of $\ell_2$, and therefore does not lie in $\textsf{M}_{1,1}$.

\end{proof}

\section{Space ideals}

\subsection{Hereditary properties}

Let us say a Banach space $X$ is \emph{hereditarily} $\textsf{M}_{\xi, \zeta}$ provided that any subspace $Y$ of $X$ lies in $\textsf{M}_{\xi, \zeta}$.     For convenience, let us say a sequence $(x_n)_{n=1}^\infty$ in a Banach space is a $c_0^{\omega_1}$-\emph{spreading model} provided that it is equivalent to the canonical $c_0$ basis.

\begin{proposition} For $0<\xi, \zeta\leqslant \omega_1$, $X$ is hereditarily $\textsf{\emph{M}}_{\xi, \zeta}$ if and only if every seminormalized, $\xi$-weakly null sequence in $X$ has a subsequence which is a $c_0^\zeta$-spreading model.

\end{proposition}

\begin{proof} Suppose that every normalized, $\xi$-weakly null sequence in $X$ has a subsequence which is a $c_0^\zeta$-spreading model. Let $Y$ be any subspace of $X$. Suppose that $(y_n)_{n=1}^\infty\subset Y$ is $\xi$-weakly null,  $(y^*_n)_{n=1}^\infty\subset Y^*$ is weakly null, and $\inf_n |y^*_n(y_n)|=\ee>0$.  By passing to a subsequence, we may assume $(y_n)_{n=1}^\infty$ is a $c_0^\zeta$-spreading model and $\sum_{n=1}^\infty \sum_{n\neq m=1}^\infty |y^*_n(y_m)|<\ee/2$.  Then if $C=\sup \{\|\sum_{n\in F}\ee_n y_n\|: F\in \mathcal{S}_\zeta, |\ee_n|=1\}$, then $$\inf\{\|\sum_{n\in F} a_ny_n^*\|: F\in \mathcal{S}_\zeta, \sum_{n\in F}|a_n|=1\} \geqslant \frac{\ee}{2C}.$$   This yields that $(y^*_n)_{n=1}^\infty$ is not $\zeta$-weakly null, and $Y\in \textsf{M}_{\xi, \zeta}$.

For the converse in  the $\zeta<\omega_1$ case, suppose that $(x_n)_{n=1}^\infty$ is a  seminormalized, $\xi$-weakly null sequence in $X$ having no subsequence which is a $c_0^\zeta$-spreading model.  Assume that $(x_n)_{n=1}^\infty$ is a basis for $Y=[x_n:n\in\nn]$ and let $(x^*_n)_{n=1}^\infty \subset Y^*$ denote the coordinate functionals. For $M=(m_n)_{n=1}^\infty\in [\nn]$, let $Y_M=[x_{m_n}:n\in\nn]$.  By hypothesis, there does not exist $L\in[\nn]$ such that $(x_n)_{n\in L}$ is a $c_0^\zeta$-spreading model. By \cite[Theorem $3.9$]{AG}, there exists $M\in [\nn]$ such that for each $L\in [M]$, $(x^*_n|_{Y_M})_{n\in L}$ is not an $\ell_1^\zeta$-spreading model.   Then $(x^*_n|_{Y_M})_{n\in M}$ is $\zeta$-weakly null in $Y_M^*$.  Since $(x_n)_{n\in M}$ is $\xi$-weakly null in $Y_M$ and $x^*_n(x_n)=1$ for all $n\in M$, $Y_M\in \complement \textsf{M}_{\xi, \zeta}$. 

For the $\zeta=\omega_1$ case of the converse, this is an inessential modification of Elton's characterization of the hereditary Dunford-Pettis property.  For the sake of completenesss, we record the argument as given in \cite[Page 28]{Diestel}.    Suppose that $(x_n)_{n=1}^\infty\subset X$ is $\xi$-weakly null having no subsequence equivalent to the $c_0$ basis. By passing to a subsequence, we may assume $(x_n)_{n=1}^\infty$ is basic with coordinate functionals $(x^*_n)_{n=1}^\infty$ and for any subsequence $(y_n)_{n=1}^\infty$ of $(x_n)_{n=1}^\infty$ and $(a_n)_{n=1}^\infty\in \ell_\infty\setminus c_0$, $\lim_n \|\sum_{i=1}^n a_iy_i\|=\infty$. Now if $P_k:[x_n:n\in\nn]\to [P_n: n\leqslant k]$ denotes the basis projections, for any $x^{**}\in X^{**}$, $$\sup_n \|\sum_{i=1}^n x^{**}(x_i^*)x_i\| \leqslant \|x^{**}\|\sup_n \|P_n\|<\infty.$$  Therefore $(x^{**}(x^*_n))_{n=1}^\infty\in c_0$, and $(x^*_n)_{n=1}^\infty$ is $\omega_1$-weakly null.   Since $x^*_n(x_n)=1$ for all $n\in\nn$, $[x_n: n\in\nn]\in \complement \textsf{M}_{\xi, \omega_1}$.

\end{proof}

\begin{rem}\upshape For each $0\leqslant \xi<\omega_1$, $X_{\omega^\xi}$ is hereditarily $\textsf{M}_{\omega^\xi, \omega_1}$, since every seminormalized, weakly null sequence in $X_{\omega^\xi}$ has either a subsequence which is an $\ell_1^{\omega^\xi}$-spreading model or a subsequence equivalent to the canonical $c_0$ basis.    

In \cite{BCM}, for each $0\leqslant \xi<\omega_1$, a reflexive Banach space $\mathfrak{X}^{\omega^\xi}_{0,1}$ with $1$-unconditional basis was defined such that every seminormalized, weakly null sequence has a subsequence which is either an $\ell_1^{\omega^\xi}$-spreading model or a $c_0^1$-spreading model, and both alternatives occur in every infinite dimensional subspace.   Thus $\mathfrak{X}^{\omega^\xi}_{0,1}$ furnish reflexive examples of members of $\textsf{M}_{\omega^\xi, 1}$.

\end{rem}

For $0\leqslant \zeta, \xi\leqslant \omega_1$.    Let us say that $X$ is \emph{hereditary by quotients} $\textsf{M}_{\xi, \zeta}$ if every quotient of $X$ is a member of $\textsf{M}_{\xi, \zeta}$.

\begin{theorem} Fix $0<\gamma\leqslant \omega_1$. For a Banach space $X$, the following are equivalent. \begin{enumerate}[(i)]\item $X^*\in \textsf{\emph{V}}_\gamma$. \item $X^*$ is hereditarily $\textsf{\emph{M}}_{\gamma, \omega_1}$. \item $X$ is hereditary by quotients $\textsf{\emph{M}}_{\omega_1, \gamma}$. \item $X\in \textsf{\emph{M}}_{\omega_1, \gamma}$ and $\ell_1\not\hookrightarrow X$.  \end{enumerate}

\end{theorem}

\begin{proof}$(i)\Rightarrow (ii)$ Assume $(i)$ holds. If $(x^*_n)_{n=1}^\infty\subset X^*$ is $\gamma$-weakly null, it is norm null.   Thus for any subspace $Y$ of $X^*$, any $\gamma$-weakly null $(y_n)_{n=1}^\infty\subset Y$, and any bounded sequence $(y^*_n)_{n=1}^\infty\subset Y^*$, $\lim_n y^{*}_n(y_n)=0$.  


$(ii)\Rightarrow (iii)$ Assume $(ii)$ holds.   For any quotient $X/N$ of $X$, $(X/N)^*=N^\perp\leqslant X^*$, so $X/N\in \textsf{M}_{\gamma, \omega_1}^\text{dual}\subset \textsf{M}_{\omega_1, \gamma}$. 


$(iii)\Rightarrow (iv)$ Assume $(iii)$ holds. If $\ell_1\hookrightarrow X$, then $\ell_2$ is a quotient of $X$, which is a contradiction.  Thus $\ell_1\not\hookrightarrow X$.    Since $X$ is a quotient of itself, $X\in \textsf{M}_{\omega_1, \gamma}$. 

$(iv)\Rightarrow (i)$ Assume $(iv)$ and $\neg (i)$.   Since $X^*\in \complement \textsf{V}_\gamma$, there exists a seminormalized, $\gamma$-weakly null sequence $(x^*_n)_{n=1}^\infty$ in $X^*$.      Fix  $0<\ee<\frac{1}{2}\inf_n \|x^*_n\|$. For each $n\in\nn$, we may fix $x_n\in B_X$ such that $x^*_n(x_n)>2\ee$.   By passing to a subsequence and relabeling, we may assume that for all $m<n$, $|x^*_n(x_m)|<\ee$.   Since $\ell_1\not\hookrightarrow X$, we may also assume that $(x_n)_{n=1}^\infty$ is weakly Cauchy.  Then with $y^*_n=x_{2n}^*$ and $y_n=x_{2n}-x_{2n-1}$, $(y^*_n)_{n=1}^\infty$ is $\gamma$-weakly null, $(y_n)_{n=1}^\infty$ is weakly null, and $\inf_n |y^*_n(y_n)|\geqslant \ee$.

\end{proof}

\subsection{Distinctness of space ideals}

We showed in Section $3$ that for any $0\leqslant \zeta<\xi\leqslant \omega_1$ and $0\leqslant \alpha<\beta\leqslant \omega_1$, $\mathfrak{G}_{\xi, \zeta}=\mathfrak{G}_{\beta, \alpha}$ if and only if $\zeta=\alpha$ and $\xi=\zeta$. Our next goal is to show that this is not true for the space ideals, due to the idempotence of  identity operators.  We recall the result from \cite{CN} that a Banach space $X$ lies in $\textsf{V}_\zeta$ for some $\omega^\xi<\zeta<\omega^{\xi+1}$ if and only if $X$ lies in $\textsf{V}_\zeta$ for every $\omega^\xi<\zeta<\omega^{\xi+1}$, which is a consequence of considering blocks of blocks. We prove analogous results below.  We need the following result for blocks of blocks.

\begin{proposition} Let $X,Y,Z$ be operators, $\alpha, \beta, \zeta$ countable ordinals, and assume $B\in \mathfrak{G}_{\beta+\zeta, \zeta}$ and $A\in \mathfrak{G}_{\alpha+\zeta, \zeta}$.  Then $AB\in \mathfrak{G}_{\alpha+\beta+\zeta, \zeta}$. 
\label{gis}
\end{proposition}

\begin{proof} By Corollary \ref{avp}, $B\in \mathfrak{G}_{\alpha+\beta+\zeta, \alpha+\zeta}$. Thus if $(x_n)_{n=1}^\infty$ is $\alpha+\beta+\zeta$-weakly null, it is sent by $B$ to a sequence which is $\alpha+\zeta$-weakly null, which is sent by $A$ to a sequence which is $\zeta$-weakly null. 

\end{proof}

\begin{corollary} For a Banach space $X$ and $\zeta<\omega_1$, let $\textsf{\emph{g}}_\zeta(X)=\omega_1$ if $X\in \textsf{\emph{G}}_{\omega_1, \zeta}$, and otherwise let $\textsf{\emph{g}}_\zeta(X)$ be the minimum ordinal $\xi<\omega_1$ such that $X\in \complement \textsf{\emph{G}}_{\xi+\zeta, \zeta}$ (noting that such a $\xi$ must exist). Then there exists $\gamma\leqslant \omega_1$ such that $\textsf{\emph{g}}_\zeta(X)=\omega^\gamma$. 

\label{nn}
\end{corollary}

\begin{proof} Note that $\textsf{g}_\zeta(X)>0$. Fix $\alpha, \beta<\textsf{g}_\zeta(X)$.    Then $I_X\in \mathfrak{G}_{\beta+\zeta, \zeta}$ and $I_X\in \mathfrak{G}_{\alpha+\zeta, \zeta}$. By Proposition \ref{gis}, $I_X\in \mathfrak{G}_{\alpha+\beta+\zeta, \zeta}$.  Thus we have shown that if $\alpha, \beta<\textsf{g}_\zeta(X)$, $\alpha+\beta<\textsf{g}_\zeta(X)$. Since $0<\textsf{g}_\zeta(X)\leqslant \omega_1$, standard facts about ordinals yield that there exists $\gamma\leqslant \omega_1$ such that $\textsf{g}_\zeta(X)=\omega^\gamma$.

\end{proof}

For the following theorem, note that if $\omega^\xi<\lambda(\zeta)$, then $\omega^\xi+\zeta=\zeta$, so $\mathfrak{G}_{\omega^\xi+\zeta, \zeta}=\mathfrak{L}$. This is the reason for the omission of this trivial case. 

\begin{theorem} Fix $0\leqslant \zeta<\omega_1$ and $\xi< \omega_1$ such that $\omega^\xi\geqslant \lambda(\zeta)$.  Then $$\varnothing\neq \complement\textsf{\emph{G}}_{\omega^\xi+\zeta, \zeta} \cap \bigcap_{\eta<\omega^\xi} \textsf{\emph{G}}_{\eta+\zeta, \zeta}.$$

\label{dreachi}
\end{theorem}

\begin{proof}  It was shown in \cite{CN} that for any Banach space $Y$ with a normalized, bimonotone basis and $0<\xi<\omega_1$,  there exists a Banach space $Z$ (there denoted by $Z_\xi(E_Y)$) such that $Z$ has a normalized, bimonotone basis, $Y$ is a quotient of $Z$,  $Z\in \cap_{\eta<\omega^\xi} \textsf{V}_\eta$, and if $(y_n)_{n=1}^\infty$ is an $\omega^\xi$-weakly null sequence in $Y$, then there exists an $\omega^\xi$-weakly null sequence $(z_n)_{n=1}^\infty$ in $Z$ such that $qz_n=y_n$ for all $n\in\nn$.

If $\zeta=0$, we consider $Z$ as above with $Y=c_0$. This space lies in $$\complement \textsf{V}_{\omega^\xi} \cap \bigcap_{\eta<\omega^\xi} \textsf{V}_\eta = \complement \textsf{G}_{\omega^\xi,0} \cap \bigcap_{\eta<\omega^\xi} \textsf{G}_{\eta,0}.$$ This completes the $\zeta=0$ case. For the remainder of the proof, we consider $\zeta>0$. 

Suppose that $\xi=0$.   Then since $1=\omega^\xi\geqslant \lambda(\zeta)\geqslant 1$, $\zeta$ is finite. Futhermore, $\eta+\zeta=\zeta$ for any $\eta<\lambda(\zeta)$, since the only such $\eta$ is $0$.  Then $X=X_\zeta$ is easily seen to satisfy the conclusions. For the remainder of the proof, we assume $0<\xi<\omega_1$.  

If $\lambda(\zeta)=\omega^\xi$, then for every $\eta<\omega^\xi$, $\eta+\zeta=\zeta$. In this case, membership in $\bigcap_{\eta<\omega^\xi} \textsf{G}_{\eta+\zeta, \zeta}=\textbf{Ban}$ is automatic.    In this case, $X_\zeta\in \complement \textsf{G}_{\omega^\xi+\zeta, \zeta}$ is the example we seek.

  We consider the remaining case, $0<\zeta, \xi$ and  $\lambda(\zeta)<\omega^\xi$. Note that this implies that $\zeta<\omega^\xi$.    We use a technique of Ostrovskii from \cite{Ostrovskii}. If $\lambda(\zeta)$ is finite, then it is equal to $1$.  In this case, let $Y=c_0$. If $\lambda(\zeta)$ is infinite, then it is a limit ordinal. In this case, let $(\lambda_n)_{n=1}^\infty$ be the sequence used to define $\mathcal{S}_{\lambda(\zeta)}$.    Let $Y$ be the completion of $c_{00}$ with respect to the norm $$\|x\|=\sup_{n\in\nn} 2^{-n} \|x\|_{\lambda_n}.$$   Note that the formal inclusions $I_1:X_\zeta\to X_{\lambda(\zeta)}$, $I_2:X_{\lambda(\zeta)}\to Y$ are bounded. The first is bounded by the almost monotone property.  For $n\in\nn$ and $E\in \mathcal{S}_{\lambda_n}$, $F=[n, \infty)\in \mathcal{S}_{\lambda(\zeta)}$. Therefore for $x\in c_{00}$, $$2^{-n}\|Ex\|_{\ell_1}\leqslant 2^{-n}((n-1)\|x\|_{c_0}+\|Fx\|_{\ell_1}) \leqslant n2^{-n} \|x\|_{\lambda(\zeta)}\leqslant 2^{-1}\|x\|_{\lambda(\zeta)}.$$ Let us also note that a bounded block sequence $(x_n)_{n=1}^\infty$ in $X_\zeta$ is $\zeta$-weakly null if and only if $\lim_n \|x\|_\beta=0$ for every $\beta<\lambda(\zeta)$ if and only if $(I_2I_1x_n)_{n=1}^\infty$ is norm null in $Y$. We have already established the equivalence of the first two properties. Let us explain the equivalence of the last two properties. First, if $(I_2I_1x_n)_{n=1}^\infty$ is norm null in $Y$, then for any $\beta<\lambda(\zeta)$, we can fix $k$ such that $\beta<\lambda_k$ and note that $$\lim_n \|x_n\|_\beta \leqslant c \lim_n \|x_n\|_{\lambda_k} \leqslant c 2^k \lim_n \|x_n\|_Y=0.$$  Here, $c$ is the norm of the formal inclusion of $X_{\lambda_k}$ into $X_\beta$.   For the reverse direction, suppose $(x_n)_{n=1}^\infty\subset B_{X_\zeta}$ and $\lim_n \|x_n\|_\beta=0$ for all $\beta<\lambda(\zeta)$.    Then \begin{align*} \lim\sup_n \|I_2I_1y\| & \leqslant \inf \Bigl\{\max\bigl\{\lim\sup_n \sum_{m=1}^k \|x_n\|_{\lambda_m}, \sup_{n>k} 2^{-n}\|I_2\|\|I_1\|\bigr\} :  k\in\nn\Bigr\} =0. \end{align*}

Let $i=I_2I_1$ and let $Z$ be as described in the first paragraph with this choice of $Y$. Let $q:Z\to Y$ be the described quotient map.    Let $W=Z\oplus_1 X_\zeta$ and let $T:W\to Y$ be given by $T(z,x)=ix- qz$.    Let $X=\ker(T)$.    Now for $\eta<\omega^\xi$, since in this case $\zeta<\omega^\xi$, $\eta+\zeta<\omega^\xi$.    Suppose that $(z_n, x_n)_{n=1}^\infty\subset X$ is $\eta+\zeta$-weakly null.   Then since $Z\in \textsf{V}_\eta$, $\|z_n\|\to 0$.  From this it follows that $(ix_n)_{n=1}^\infty=(qz_n)_{n=1}^\infty$ is norm null. Therefore $(ix_n)_{n=1}^\infty$ is norm null, and  $(x_n)_{n=1}^\infty$ is $\zeta$-weakly null in $X_\zeta$. Therefore $(z_n, x_n)_{n=1}^\infty$ is $\zeta$-weakly null in $X$.   We last show that $X\in \complement \textsf{G}_{\omega^\xi+\zeta, \zeta}$.  To that end, let us first note that the basis of $Y$ is $\lambda(\zeta)$-weakly null. This is obvious if $\lambda(\zeta)=1$ and $Y=c_0$.  For the case in which $\lambda(\zeta)$ is infinite, the space $Y$ is a mixed Schreier space as defined in \cite{CN}, where it was shown that the basis of $Y$ is $\lambda(\zeta)$-weakly null.  By the properties of $Z$ and $q$, since $\lambda(\zeta) \leqslant \omega^\xi<\omega^\xi+\zeta$, there exists an $\omega^\xi$-weakly null sequence $(z_n)_{n=1}^\infty$ in $Z$ such that $qz_n=e_n$, where $(e_n)_{n=1}^\infty$ simultaneously denotes the bases of $Y$ and $X_\zeta$.  Also note that $(e_n)_{n=1}^\infty$ is $\zeta+1$-weakly null in $X_\zeta$. Since $$\omega^\xi+\zeta\geqslant \zeta+\omega^\xi\geqslant \zeta+1,$$ $(e_n)_{n=1}^\infty$ is $\omega^\xi+\zeta$-weakly null in $X_\zeta$.    Therefore $(z_n, e_n)_{n=1}^\infty$ is $\omega^\xi+\zeta$-weakly null in $X$. However, since $(e_n)_{n=1}^\infty$ is not $\zeta$-weakly null in $X_\zeta$, $(z_n, e_n)_{n=1}^\infty$ is not $\zeta$-weakly null in $X$.   Therefore $X\in \complement \textsf{G}_{\omega^\xi+\zeta, \zeta}$.

\end{proof}

\begin{corollary} The classes $\textsf{\emph{wBS}}_\zeta, \textsf{\emph{G}}_{\omega^\gamma+\zeta, \zeta}$, $ \zeta,\gamma<\omega_1, \omega^\gamma\geqslant \lambda(\zeta)$, are distinct.

\end{corollary}

\begin{theorem} The classes $\textsf{\emph{G}}_{\zeta+\omega^\gamma, \zeta}$, $0\leqslant \zeta<\omega_1$, $0\leqslant \gamma\leqslant \omega_1$, are distinct. 
\label{sok}
\end{theorem}

\begin{proof}  We first recall that if $\zeta<\omega_1$ and $\gamma\leqslant \gamma_1\leqslant \omega_1$, $\textsf{G}_{\zeta+\omega^{\gamma_1}, \zeta}\subset \textsf{G}_{\zeta+\omega^\gamma, \zeta}$.    Thus the statement that these two classes are distinct is equivalent to saying that the former is a proper subset of the latter.

   We will show that the classes are distinct.   Fix $0\leqslant \zeta, \zeta_1<\omega_1$ and $0\leqslant \gamma, \gamma_1\leqslant \omega_1$.    If $\zeta<\zeta_1$, $$X_\zeta\in \textsf{wBS}_{\zeta_1}\cap \complement \textsf{G}_{\zeta+\omega^\gamma, \zeta}\subset \textsf{G}_{\zeta_1+\omega^{\gamma_1}, \zeta_1}\cap \complement \textsf{G}_{\zeta+\omega^\gamma, \zeta}.$$   By symmetry, if $\zeta_1<\zeta$, $\textsf{G}_{\zeta+\omega^\gamma, \zeta}\neq \textsf{G}_{\zeta_1+\omega^{\gamma_1}, \zeta_1}$.   Thus if $\zeta\neq \zeta_1$,  $\textsf{G}_{\zeta+\omega^\gamma, \zeta}\neq \textsf{G}_{\zeta_1+\omega^{\gamma_1}, \zeta_1}$.    

In order to complete the proof that the classes are distinct, it suffices to assume that $\gamma_1<\gamma\leqslant \omega_1$ and exhibit some Banach space $Z\in \textsf{G}_{\zeta+\omega^{\gamma_1}, \zeta}\cap \complement \textsf{G}_{\zeta+\omega^{\gamma}, \zeta}$.  We first claim that it is sufficient to prove the case $\gamma<\omega_1$.   This is because if we prove that $\textsf{G}_{\zeta+\omega^{\gamma}, \zeta}\subsetneq \textsf{G}_{\zeta+\omega^{\gamma_1}, \zeta}$ whenever $0\leqslant \gamma_1<\gamma<\omega_1$, then for any $0\leqslant \gamma_1<\omega_1$, $$\textsf{G}_{\zeta+\omega^{\omega_1}, \zeta}=\textsf{G}_{\omega_1, \zeta}\subset \textsf{G}_{\zeta+\omega^{\gamma_1+1}, \zeta} \subsetneq \textsf{G}_{\zeta+\omega^{\gamma_1}, \zeta}.$$

Fix $0<\gamma<\omega_1$ and let $(\gamma_n)_{n=1}^\infty$ be the sequence defining  $\mathcal{S}_{\omega^\gamma}$.    Fix a sequence $(\vartheta_n)_{n=1}^\infty$ such that $\vartheta:=\sum_{n=1}^\infty \vartheta_n<1$.     Given a sequence space $E$,  we define norm on $[\cdot]$ on $c_{00}$ by letting $|\cdot|_0=\|\cdot\|_E$, $$|x|_{k+1,n}=\sup\Bigl\{\vartheta_n \sum_{i=1}^d |E_ix|_k: n\in\nn, E_1<\ldots <E_d, (\min E_i)_{i=1}^d\in \mathcal{S}_{\gamma_n}\Bigr\},$$ $$|x|_{k+1}=\max\Bigl\{|x|_k, \bigl(\sum_{n=1}^\infty |x|_{k+1,n}^2\bigr)^{1/2}\Bigr\},$$ $$[x]=\lim_k|x|_k,$$ and $$[x]_n=\lim_k |x|_{k,n}.$$  Let us denote the completion of $c_{00}$ with respect to this norm by $Z_\gamma(E)$.  The norm $[\cdot]$ on $Z_\gamma(E)$ satisfies the following $$[z]= \max\Bigl\{\|z\|_E, \bigl(\sum_{n=1}^\infty [z]_n^2\bigr)^{1/2}\Bigr\}.$$     This construction is a generalization of a construction by Odell and Schlumprecht. We will apply the construction with $E=X_\zeta$.    It is a well known fact of such constructions that, since the basis of $X_\zeta$ is shrinking, so is the basis of $Z_\gamma(X_\zeta)$ (see, for example, \cite{CN}).    It was shown in \cite{CN} that if $(z_n)_{n=1}^\infty$ is any seminormalized block sequence in $Z_\gamma(X_\zeta)$, then \begin{enumerate}[(a)]\item $(z_n)_{n=1}^\infty$ is not $\beta$-weakly null for any $\beta<\omega^\gamma$,  \item $(z_n)_{n=1}^\infty$ is  $\omega^\gamma$-weakly null in $Z_\gamma(X_\zeta)$ if and only if it is $\omega^\gamma$-weakly null in $X_\zeta$. \end{enumerate}

We will show that $Z_\gamma(X_\zeta)\in \cap_{\beta<\omega^\gamma} \textsf{G}_{\zeta+\beta, \zeta}$, and in particular $Z_\gamma(X_\zeta)\in \textsf{G}_{\zeta+\omega^{\gamma_1}, \zeta}$, while $Z_\gamma(X_\zeta)\in \complement \textsf{G}_{\zeta+\omega^\gamma, \zeta}$.  This will complete the proof of the distinctness of the classes.   

We prove that $Z_\gamma(X_\zeta)\in \complement \textsf{G}_{\zeta+\omega^\gamma, \zeta}$.    As remarked above, the basis is shrinking and normalized, and so it is weakly null.  If it were not $\zeta+\omega^\gamma$-weakly null, there would exist some $(m_n)_{n=1}^\infty\in[\nn]$ and $\ee>0$ such that $$\ee\leqslant \inf \{[z]: F\in \mathcal{S}_{\omega^\gamma}[\mathcal{S}_\zeta], z\in \text{co}(e_{m_n}: n\in F)\}.$$   But by Theorem \ref{shinsuke}, we may choose $F_1<F_2<\ldots$, $F_i\in \mathcal{S}_\zeta$, and positive scalars $(a_i)_{i\in \cup_{n=1}^\infty F_n}$ such that $\sum_{i\in F_n}a_i=1$ and the sequence $(z_n)_{n=1}^\infty$ defined by $z_n=\sum_{i\in F_n} a_ie_{m_i}$ is equivalent to the $c_0$ basis in $X_\zeta$. But since $$\ee\leqslant \inf \{[z]: F\in \mathcal{S}_{\omega^\gamma}[\mathcal{S}_\zeta], z\in \text{co}(e_{m_n}: n\in F)\},$$ $(z_n)_{n=1}^\infty$ is an $\ell_1^{\omega^\gamma}$-spreading model in $Z_\gamma(X_\zeta)$, contradicting item $(b)$ above.  Therefore the canonical $Z_\gamma(X_\zeta)$ basis is $\zeta+\omega^\gamma$-weakly null.  But it is evidently not $\zeta$-weakly null, whence $Z_\gamma(X_\zeta)\in \complement \textsf{G}_{\zeta+\omega^\gamma, \zeta}$.     

Now let us show that $Z_\gamma(X_\zeta)\in \cap_{\beta<\omega^\gamma}\textsf{G}_{\zeta+\beta, \zeta}$.     First consider the case $\lambda(\zeta)<\omega^\gamma$, which is equivalent to $\zeta+\beta<\omega^\gamma$ for all $\beta<\omega^\gamma$. In this case, $$\{\zeta+\beta: \beta<\omega^\gamma\}=[0, \omega^\gamma).$$  It therefore follows from property $(a)$ above that $$Z_\gamma(X_\zeta)\in \bigcap_{\beta<\omega^\gamma} \textsf{G}_{\beta, 0} = \bigcap_{\beta<\omega^\gamma} \textsf{G}_{\zeta+\beta, 0} \subset \bigcap_{\beta<\omega^\gamma} \textsf{G}_{\zeta+\beta, \zeta}.$$

   Let us now treat the case $\lambda(\zeta)\geqslant \omega^\gamma$.    Write $$\zeta= \lambda(\zeta)+\mu$$ and note that $$\mu+\omega^\gamma \leqslant \mu+\lambda(\zeta) \leqslant \lambda(\zeta)+\mu=\zeta.$$

	We  claim that if $(z_n)_{n=1}^\infty$ is a seminormalized block sequence in $Z_\gamma(X_\zeta)$ which is not $\zeta$-weakly null in $Z_\gamma(X_\zeta)$, then there exists $\beta<\lambda(\zeta)$ such that $\lim\sup_n \|z_n\|_\beta>0$.  To see this, suppose that for every $\beta<\lambda(\zeta)$, $\lim_n \|z_n\|_\beta=0$, but $(z_n)_{n=1}^\infty$ is not $\zeta$-weakly null in $Z_\gamma(X_\zeta)$.   Then, by Proposition \ref{phenomenal}$(ii)$, by passing to a subsequence and relabeling, we may assume $(z_n)_{n=1}^\infty$, when treated as a sequence in $X_\zeta$,  is dominated by a subsequence $(e_{m_i})_{i=1}^\infty$ of the $X_\mu$ basis, and $(z_n)_{n=1}^\infty$, when treated as a sequence in $Z_\gamma(X_\zeta)$,  is an $\ell_1^\zeta$-spreading model.    Since $\zeta\geqslant \mu+\omega^\gamma$, we may, after passing to a subsequence again, assume $$0<\ee\leqslant \inf \{[z]: F\in \mathcal{S}_{\omega^\gamma}[\mathcal{S}_\mu], z\in \text{co}(z_n: n\in F)\}.$$    We may select $F_1<F_2<\ldots$, $F_i\in \mathcal{S}_\mu$, and positive scalars $(a_i)_{i\in \cup_{n=1}^\infty F_n}$ such that $\sum_{i\in F_n}a_i=1$ and $(\sum_{i\in F_n}a_ie_{m_i})_{n=1}^\infty\subset X_\mu$ is equivalent to the canonical $c_0$ basis (again using Theorem \ref{shinsuke} as in the previous case).   Since $(z_n)_{n=1}^\infty\subset X_\zeta$ is dominated by $(e_{m_i})_{i=1}^\infty\subset X_\mu$, $(\sum_{i\in F_n} a_iz_i)_{n=1}^\infty$ is WUC in $X_\zeta$.    But since $$0<\ee\leqslant \inf \{[z]: F\in \mathcal{S}_{\omega^\gamma}[\mathcal{S}_\mu], z\in \text{co}(z_n: n\in F)\},$$   $(\sum_{i\in F_n}a_i z_i)_{n=1}^\infty$ must be an $\ell_1^{\omega^\gamma}$-spreading model in $Z_\gamma(X_\zeta)$, contradicting $(b)$ above.    This proves the claim from the beginning of the paragraph.    Now suppose that $(z_n)_{n=1}^\infty$ is a weakly null sequence in $Z_\gamma(X_\zeta)$ which is not $\zeta$-weakly null.   Then by the claim  combined with Corollary \ref{AC}, $(z_n)_{n=1}^\infty$ is not $\zeta$-weakly null in $X_\zeta$.  After passing to a subsequence, we may assume $(z_n)_{n=1}^\infty$ is an $\ell_1^\zeta$-spreading model in $X_\zeta$. Assume that $$0<\ee\leqslant \inf \{[z]: F\in \mathcal{S}_\zeta, z\in \text{co}(z_n: n\in F)\}.$$      Now fix $n\in\nn$ and $F\in \mathcal{S}_{\gamma_n}[\mathcal{S}_\zeta]$ and scalars $(a_i)_{i\in F}$.   By definition of $\mathcal{S}_{\gamma_n}[\mathcal{S}_\zeta]$, there exist $F_1<\ldots <F_d$ such that $F=\cup_{j=1}^d F_j$,  $\varnothing\neq F_j\in \mathcal{S}_\zeta$, and $(\min F_j)_{j=1}^d\in \mathcal{S}_{\gamma_n}$.   Let $E_i=\text{supp}(z_i)$ and let $I_j=\cup_{i\in F_j} E_i$.  Since $$\min I_i = \min \text{supp}(z_{\min F_j}) \geqslant \min F_j,$$ $(\min I_j)_{j=1}^d$ is a spread of $(\min F_j)_{j=1}^d$, whence $(\min I_j)_{j=1}^d\in \mathcal{S}_{\gamma_n}$.     Therefore \begin{align*} [\sum_{i\in F} a_iz_i] &  \geqslant \Bigl(\sum_{k=1}^\infty \bigl[\sum_{i\in F}a_iz_i\bigr]_k^2\Bigr)^{1/2} \geqslant \bigl[\sum_{i\in F} a_iz_i\bigr]_n \\ & \geqslant \vartheta_n \sum_{j=1}^d [I_j\sum_{i\in F}a_iz_i]= \vartheta_n\sum_{j=1}^d [\sum_{i\in F_j} a_iz_i] \geqslant \ee \vartheta_n \sum_{j=1}^d \sum_{i\in F_j}|a_i| \\ & = \ee \vartheta_n \sum_{i\in F}|a_i|. \end{align*}   Thus $$0<\inf\{[z]: F\in \mathcal{S}_{\gamma_n}[\mathcal{S}_\zeta], x\in \text{co}(z_n: n\in F)\}.$$  From this it follows that $(z_i)_{i=1}^\infty$ is not $\zeta+\gamma_n$-weakly null.  Since this holds for any $n\in\nn$ and $\sup_n \gamma_n=\omega^\gamma$, $(z_i)_{i=1}^\infty$ is not $\zeta+\beta$-weakly null for any $\beta<\omega^\gamma$.   Thus by contraposition, for any $\beta<\omega^\gamma$, any $\zeta+\beta$-weakly null sequence in $Z_\gamma(X_\zeta)$ is $\zeta$-weakly null, whence $Z_\gamma(X_\zeta)\in \cap_{\beta<\omega^\gamma} \textsf{G}_{\zeta+\beta, \zeta}$.     This completes the proof of the distinctness of these classes.

\end{proof}

\begin{rem}\upshape For $\xi, \eta<\omega_1$ and $\delta, \zeta\leqslant \omega_1$ with $\eta\neq \zeta$, the classes $\textsf{G}_{\omega^\xi+\zeta, \zeta}$, $\textsf{G}_{\eta+\omega^\delta, \eta}$ are not equal. Indeed, if $\eta<\zeta$, $X_\eta\in \textsf{G}_{\omega^\xi+\zeta, \zeta}\cap \complement \textsf{G}_{\eta+\omega^\delta, \eta}$. This is because every sequence in $X_\eta$ is $\eta+1$-weakly null, and therefore $\zeta$-weakly null. However, the basis of $X_\eta$ is $\eta+1$-weakly null, and therefore $\eta+\omega^\delta$-weakly null, but not $\eta$-weakly null.   Now if $\zeta<\eta$, either $\omega^\xi+\zeta>\zeta$ or $\omega^\xi+\zeta=\zeta$. If $\omega^\xi+\zeta>\zeta$, $X_\zeta\in \textsf{G}_{\eta+\omega^\delta, \eta}\cap \complement \textsf{G}_{\omega^\xi+\zeta, \zeta}$. If $\omega^\xi+\zeta=\zeta$, then $\textsf{G}_{\omega^\xi+\zeta, \zeta}=\textbf{Ban}\neq \textsf{G}_{\eta+\omega^\delta, \eta}$.

We next wish to discuss how the classes $\textsf{G}_{\omega^\xi+\zeta, \zeta}$ can be compared to the classes $\textsf{G}_{\zeta+\omega^\delta, \zeta}$. In particular, we will show that they are equal if and only if $\omega^\xi+\zeta=\zeta+\omega^\delta$.  If $\zeta=0$, then  $\textsf{G}_{\omega^\xi+\zeta, \zeta}=\textsf{V}_{\omega^\xi}$ and $\textsf{G}_{\zeta+\omega^\delta, \zeta}=\textsf{V}_{\omega^\delta}$.  Then $\textsf{V}_{\max\{\omega^\xi, \omega^\delta\}}\subset \textsf{V}_{\min \{\omega^\xi, \omega^\delta\}}$, with proper containment if and only if $\xi\neq \delta$.

Now for $0<\zeta<\omega_1$, write $\zeta=\omega^{\alpha_1}n_1+\ldots +\omega^{\alpha_l}n_l$ $l, n_1, \ldots, n_l\in\nn$, $\alpha_1>\ldots >\alpha_l$.   Let us consider several cases.  For convenience, let $\alpha=\alpha_1$ and $n=n_1$.    

Case $1$: $\xi<\alpha$. Then $\omega^\xi+\zeta=\zeta$ and $\textsf{G}_{\omega^\xi+\zeta, \zeta}=\textbf{Ban}\neq \textsf{G}_{\zeta+\omega^\delta, \delta}$.

For the remaining cases, we will assume $\xi\geqslant \alpha$, which implies that $\omega^\xi+\zeta>\zeta$.

Case $2$: $\omega^\xi+\zeta<\zeta+\omega^\delta$.  Then there exists $\beta<\omega^\delta$ such that $\omega^\xi+\zeta= \zeta+\beta$.      Then the space $Z_\delta(X_\zeta)$ from Theorem \ref{sok} lies in $\complement \textsf{G}_{\zeta+\omega^\delta, \zeta}\cap \textsf{G}_{\zeta+\beta, \zeta}= \complement \textsf{G}_{\zeta+\delta, \zeta}\cap \textsf{G}_{\omega^\xi+\zeta, \zeta}$.

Case $3$: $\omega^\xi+\zeta=\zeta+\omega^\delta$.  In this case, of course $\textsf{G}_{\omega^\xi+\zeta, \zeta}=\textsf{G}_{\zeta+\omega^\delta, \zeta}$.  By considering the Cantor normal forms of $\omega^\xi+\zeta$ and $\zeta+\omega^\delta$, it follows that equality can only hold in the case that $\xi=\delta=\alpha$ and $\zeta=\omega^\alpha n$, in which case $\omega^\xi+\zeta=\omega^\alpha(n+1)=\zeta+\omega^\delta$.

For the remaining cases, we will assume $\omega^\xi+\zeta>\zeta+\omega^\delta$. Note that this implies $\delta\leqslant \xi$. Indeed, if $\delta>\xi$, then since we are in the case $\xi\geqslant \alpha$, it follows that $\omega^\delta>\omega^\xi, \zeta$. By standard properties of ordinals, $\omega^\xi>\omega^\xi+\zeta$.   Therefore for the remaining cases, $\omega^\xi+\zeta>\zeta+\omega^\delta$ and $\alpha, \delta\leqslant \xi$.

Case $4$: $\delta=\xi>\alpha$. Then the space $X_{\omega^\delta}$ lies in $\complement \textsf{G}_{\omega^\xi+\zeta, \zeta}\cap \textsf{G}_{\zeta+\omega^\delta, \zeta}$.  To see this, note that since $\delta>\alpha$, $\zeta+\omega^\delta=\omega^\delta$. Moreover, we have already shown that any $\omega^\delta$-weakly null sequence in $X_{\omega^\delta}$ has the property that every subsequence has a further WUC subsequence.  Thus any $\omega^\delta$-weakly null sequence in $X_{\omega^\delta}$ is $1$-weakly null, and $X_{\omega^\delta}\in \textsf{G}_{\zeta+\omega^\delta, \zeta}$. But of course the basis of $X_{\omega^\delta}$ shows that it does not lie in $\textsf{G}_{\omega^\xi+\zeta, \zeta}\subset \textsf{G}_{\omega^\delta+1, \omega^\delta}$. 

Case $5$: $\xi=\alpha>\delta$.  The space $Z_\xi(X_\zeta)$, as defined in Theorem \ref{sok}, lies in $\complement \textsf{G}_{\omega^\xi+\zeta, \zeta}\cap \textsf{G}_{\zeta+\omega^\delta, \zeta}$.  To see this, let us note that $$Z_\xi(X_\zeta)\in \complement \textsf{G}_{\zeta+\omega^\xi, \zeta}\cap \bigcap_{\gamma<\omega^\xi} \textsf{G}_{\zeta+\gamma, \zeta}.$$  Since $\xi\geqslant \alpha$, $\omega^\xi+\zeta \geqslant \zeta+\omega^\xi$, whence $\textsf{G}_{\omega^\xi+\zeta, \zeta}\subset \textsf{G}_{\zeta+\omega^\xi, \zeta}$ and $Z_\xi(X_\zeta)\in \complement \textsf{G}_{\zeta+\omega^\xi, \zeta}\subset \complement \textsf{G}_{\omega^\xi+\zeta, \zeta}$.   Since $\omega^\delta<\omega^\xi$, $Z_\xi(X_\zeta)\in \textsf{G}_{\zeta+\omega^\delta ,\zeta}$.

Case $6$: $\xi>\alpha, \delta$.  Then the space $Z_\xi(c_0)$, as shown in \cite{CN}, lies in $\textsf{wBS}_{\omega^\xi}\cap \bigcap_{\gamma<\omega^\xi} \textsf{V}_\gamma$. Furthermore, the basis of the space is normalized, weakly null.  Therefore the basis is $\omega^\xi$-weakly null but not $\gamma$-weakly null for any $\gamma<\omega^\xi$.  Therefore $Z_\xi(c_0)\in \complement \textsf{G}_{\omega^\xi, \zeta}\subset \complement \textsf{G}_{\omega^\xi+\zeta, \zeta}$.  However, since $\alpha, \delta<\xi$, $\zeta+\omega^\delta<\xi$, and $Z_\xi(c_0)\in \textsf{V}_{\zeta+\omega^\delta}\subset \textsf{G}_{\zeta+\omega^\delta, \zeta}$. Therefore $Z_\xi(c_0)$ lies in $\complement \textsf{G}_{\omega^\xi+\zeta, \zeta}\cap \textsf{G}_{\zeta+\omega^\delta, \zeta}$. 

Case $7$: $\xi=\alpha=\delta$. In this case, we can write $\zeta=\omega^\alpha n+\mu$, where $\mu=\omega^{\alpha_2}n_2+\ldots +\omega^{\alpha_l}n_l$. Note that in this case, $\mu>0$, since otherwise we would be in the case $\omega^\xi+\zeta=\zeta+\omega^\delta$. Then the space $X_{\omega^\alpha(n+1)}$ lies in $\complement \textsf{G}_{\omega^\alpha +\zeta, \zeta}\cap \textsf{G}_{\zeta+\omega^\alpha, \zeta}$.    To see this,  note that  the canonical basis of $X_{\omega^\alpha (n+1)}$ is $$\omega^\alpha(n+1)+1 \leqslant \omega^\alpha(n+1)+\mu= \omega^\alpha +\zeta$$ weakly null, but it is not $\omega^\alpha(n+1)=\omega^\alpha n+\omega^\alpha$-weakly null, and therefore not $\zeta$-weakly null.      Thus $X_{\omega^\alpha(n+1)}\in \complement \textsf{G}_{\omega^\alpha+\zeta, \zeta}$. However, if $(x_n)_{n=1}^\infty$ is $\omega^\alpha (n+1)$-weakly null, then by Theorem \ref{shinsuke}, every subsequence of $(x_n)_{n=1}^\infty$ has a further subsequence which is dominated by a subsequence of the $X_{\omega^\alpha n}$ basis. This means $(x_n)_{n=1}^\infty$ is $\omega^\alpha n+1$-weakly null. Since $\omega^\alpha n+1\leqslant \omega^\alpha n+\mu=\zeta$ and $\zeta+\omega^\alpha = \omega^\alpha(n+1)$,  $X_{\omega^\alpha(n+1)}\in \textsf{G}_{\zeta+\omega^\alpha, \zeta}$.

\end{rem}

Our next goal will be to prove a fact regarding the distinctness of the space ideals $\textsf{M}_{\xi, \zeta}$ analogous to those proved above for the classes $\textsf{G}_{\xi, \zeta}$.  

\begin{rem}\upshape If $\xi, \eta$ are ordinals such that $\omega^\xi+1<\eta<\omega^{\xi+1}$, then there exist ordinals $\alpha, \gamma<\eta$ such that $\gamma>1$ and $\alpha+\gamma=\eta$. This is obvious if $\xi=0$, since since $\eta>2$ is finite and we may take $\eta=1+(\eta-1)$ in this case. Assume $0<\xi$.  Then there exist $n\in\nn$ and $\delta<\omega^\xi$ such that $\eta=\omega^\xi n+\delta$. If $n>1$,  we may take $\alpha=\omega^\xi(n-1)$ and $\gamma=\omega^\xi$.  Now if $n=1$, then $\delta>1$, and we may take $\alpha=\omega^\xi$ and $\gamma=\delta$. 

\label{sidearm}
\end{rem}

\begin{theorem} Fix $0\leqslant \xi<\omega_1$ and $0< \nu\leqslant \omega_1$. Let $X$ be a Banach space.  \begin{enumerate}[(i)]\item $X$ is hereditarily $\textsf{\emph{M}}_{\mu, \nu}$ for some $\omega^\xi<\mu<\omega^{\xi+1}$ if and only if $X$ is hereditarily $\textsf{\emph{M}}_{\mu, \nu}$ for every $\omega^\xi<\mu<\omega^{\xi+1}$.  \item $X$ is hereditarily $\textsf{\emph{M}}_{\nu, \mu}$ for some $\omega^\xi<\mu<\omega^{\xi+1}$ if and only if $X$ is hereditarily $\textsf{\emph{M}}_{\nu, \mu}$ for every $\omega^\xi<\mu<\omega^{\xi+1}$. \end{enumerate}

\end{theorem}

\begin{proof}$(i)$ Seeking a contradiction, suppose that $X$ is hereditarily $\textsf{M}_{\mu, \nu}$ for some but not all $\mu \in (\omega^\xi, \omega^{\xi+1})$.   Let $\eta$ be the minimum ordinal $\mu$ such that $X$ is not hereditarily  $\textsf{M}_{\mu, \nu}$. Note that, since the classes $\textsf{M}_{\mu, \nu}$ are decreasing with $\mu$ and $X$ is hereditarily $\textsf{M}_{\mu, \nu}$ for some $\omega^\xi<\mu<\omega^{\xi+1}$,  it follows that $\omega^\xi+1<\eta$.  We can write $\eta=\alpha+\gamma$ for some $\alpha, \gamma<\eta$ with $\gamma>1$.   Since $X$ is not hereditarily $\textsf{M}_{\eta, \nu}$, there exists a seminormalized, $\eta$-weakly null sequence $(x_n)_{n=1}^\infty$ in $X$ which has no subsequence which is a $c_0^\nu$-spreading model.    Since $\alpha+1<\alpha+\gamma$, the minimality of $\eta$ implies that $X$ is hereditarily $\textsf{M}_{\alpha+1, \nu}$, which means $(x_n)_{n=1}^\infty$ has a subsequence which is an $\ell_1^{\alpha+1}$-spreading model.   By Corollary \ref{block of block}$(i)$, there exists a convex block sequence $(y_n)_{n=1}^\infty$ of $(x_n)_{n=1}^\infty$ which is an $\ell_1^1$-spreading model and which is $\gamma$-weakly null.  But since $(y_n)_{n=1}^\infty$ is an $\ell_1^1$-spreading model, it can have no subsequence which is a $c_0^\nu$-spreading model. Since $\gamma<\eta$, $(y_n)_{n=1}^\infty$ witnesses that $X$ is not hereditarily $\textsf{M}_{\gamma, \nu}$, contradicting the minimality of $\eta$.

$(ii)$ Arguing as in $(i)$, let us suppose we have $\omega^\xi+1<\eta<\omega^{\xi+1}$ such that $X$ is hereditarily $\textsf{M}_{\nu, \mu}$for every $\mu<\eta$ but $X$ is not hereditarily $\textsf{M}_{\nu, \eta}$. Then there exists a $\nu$-weakly null $(x_n)_{n=1}^\infty\subset X$ which has no subsequence which is a $c_0^\eta$-spreading model. Write $\eta=\alpha+\gamma$, $\alpha, \gamma<\eta$, $\gamma>1$.   By passing to a subsequence, we may assume $(x_n)_{n=1}^\infty$ is a $c_0^{\alpha+1}$-spreading model.    By Corollary \ref{block of block}$(ii)$, there exists a blocking $(y_n)_{n=1}^\infty$ of $(x_n)_{n=1}^\infty$ which is a $c_0^1$-spreading model and has no subsequence which is a $c_0^\gamma$-spreading model.    Since $(y_n)_{n=1}^\infty$ is a $c_0^1$-spreading model, it is $1$-weakly null, and therefore $\nu$-weakly null.    But $(y_n)_{n=1}^\infty$ has no subsequence which is a $c_0^\gamma$-spreading model. Since $\gamma<\eta$, this contradicts the minimality of $\eta$.

\end{proof}

\begin{rem}\upshape The previous theorem yields that for a fixed $0<\zeta\leqslant \omega_1$ and $0\leqslant \xi<\omega_1$, a given Banach space $X$ may lie in $\complement \textsf{M}_{\omega^\xi, \zeta}\cap \bigcap_{\eta<\omega^\xi}\textsf{M}_{\eta, \zeta}$. That is, the first ordinal $\eta$ for which $X$ fails to lie in $\textsf{M}_{\eta, \zeta}$ is of the form $\omega^\xi$, $0\leqslant \xi<\omega_1$.  But it also allows for $X$ to lie in $\textsf{M}_{\omega^\xi, \zeta}$ and fail to lie in $\textsf{M}_{\omega^\xi+1, \zeta}$.  Let us make this precise: For $1\leqslant \zeta\leqslant \omega_1$, let  $\textsf{m}_\zeta(X)=\omega_1$ if $X\in \textsf{M}_{\omega_1, \zeta}$ and otherwise let $\textsf{m}_\zeta(X)$ be the minimum $\eta$ such that $X\in \complement \textsf{M}_{\eta, \zeta}$.  Let $\textsf{m}^*_\zeta(X)=\omega_1$ if $X\in \textsf{M}_{\zeta, \omega_1}$, and otherwise let $\textsf{m}^*_\zeta(X)$ be the minimum $\eta$ such that $X\in \complement \textsf{M}_{\zeta, \eta}$.    Then the preceding theorem yields that for any $1\leqslant \zeta\leqslant \omega_1$ and any Banach space $X$,  there exists $0\leqslant \xi\leqslant \omega_1$ such that either $\textsf{m}_\zeta(X)=\omega^\xi$ or $\textsf{m}_\zeta(X)=\omega^\xi+1$, and a similar statement holds for $\textsf{m}_\zeta^*$.

Contrary to the $\textsf{G}_{\xi, \zeta}$ case, both alternatives can occur for both $\textsf{m}_\zeta$ and $\textsf{m}^*_\zeta$.     For example, for $0<\xi<\omega_1$, our spaces $Z_\xi(c_0)$ lie in $\bigcap_{\eta<\omega^\xi}\textsf{V}_\eta$, and therefore lie in $\bigcap_{\eta<\omega^\xi}\textsf{M}_{\eta, \omega_1}\subset \bigcap_{\zeta\leqslant \omega_1}\bigcap_{\eta<\omega^\xi}\textsf{M}_{\eta, \zeta}$.  However, the basis of this space is $\omega^\xi$-weakly null, and the dual basis is $1$-weakly null, so $Z_\xi(c_0)\in \complement \textsf{M}_{\omega^\xi, 1}\subset \bigcap_{1\leqslant \zeta\leqslant \omega_1} \textsf{M}_{\omega^\xi, \zeta}$.    Thus for every $1\leqslant \zeta\leqslant \omega_1$, $\textsf{m}_\zeta(Z_\xi(c_0))=\omega^\xi$.     Since these spaces have a shrinking, asymptotic $\ell_1$ basis, they are reflexive.    From this it follows that for all $1\leqslant \zeta\leqslant \omega_1$, $\textsf{m}_\zeta^*(Z_\xi(c_0)^*)=\omega^\xi$.    For the $\xi=0$ case, $\textsf{m}_\zeta(\ell_2)=\textsf{m}_\zeta^*(\ell_2)=1=\omega^0$ for every $1\leqslant \zeta\leqslant \omega_1$.

However, as we have already seen, for any $0\leqslant \xi<\omega_1$, $\textsf{m}_\zeta(X_{\omega^\xi})=\textsf{m}_\zeta^*(X_{\omega^\xi}^*)=\omega^\xi+1$. This completely elucidates the examples with $\xi<\omega_1$.

For the $\xi=\omega_1$ case, we note that $\textsf{m}_\zeta(X)=\omega_1$ if and only if $X\in \bigcap_{\eta<\omega_1}\textsf{M}_{\eta, \zeta}=\textsf{M}_{\omega_1, \zeta}$, and a similar statement holds for $\textsf{m}^*_\zeta$.

\end{rem}

\subsection{Three-space}

In \cite{Ostrovskii}, a Banach space $X$ with subspace $Y$ was exhibited such that $Y, X/Y$ have the weak Banach-Saks property, while $X$ does not. In \cite{CastilloGonzalez}, it was shown that $Y, X/Y$ have the hereditary Dunford-Pettis property, while $X$ does not.   More precisely, let $q:\ell_1\to c_0$ be a quotient map, $I:X_1\to c_0$  the formal inclusion, $X=\ell_1\oplus_1 X_1$, $T:X\to c_0$ be given by $T(x,y)=qx+Iy$, and $Y=\ker(T)$.    Then $X/Y=c_0$, which has the  hereditarily Dunford-Pettis property. If $(x_n, y_n)_{n=1}^\infty$ is a weakly null sequence in $Y$, then $\lim_n x_n=\lim_n qx_n=\lim_n Iy_n=0$. Since $(y_n)_{n=1}^\infty$ is a weakly null sequence in $X_1$ with $\lim_n \|y_n\|_{c_0}=0$, $(y_n)_{n=1}^\infty$, and therefore $(x_n, y_n)_{n=1}^\infty$, has a WUC subsequence. This yields that $Y$ has the hereditary Dunford-Pettis property.   However, the basis of $X_1$ is $2$-weakly null and can be normed by the basis of $X_1^*$, which is $1$-weakly null.   Thus $X\in \complement \textsf{M}_{2,1}$.  We modify this example to provide a sharp solution to the three space properties of the classes $\textsf{wBS}_\xi$.

\begin{theorem} For any $0\leqslant \zeta, \xi<\omega_1$, any Banach space $X$, and any subspace $Y$ such that $Y\in \textsf{\emph{wBS}}_\xi$, and $X/Y\in \textsf{\emph{wBS}}_\zeta$, $X\in \textsf{\emph{wBS}}_{\zeta+\xi}$.   

For any $0\leqslant \zeta, \xi<\omega_1$, there exist a Banach space $X$ with a subspace $Y$ such that $Y\in \textsf{\emph{wBS}}_\xi$, $X/Y\in \textsf{\emph{wBS}}_\zeta$, and $X\in \cup_{\gamma<\zeta+\xi} \complement \textsf{\emph{wBS}}_\gamma$.

\end{theorem}

\begin{proof}  Assume $Y\in \textsf{wBS}_\xi$ and $X/Y\in \textsf{wBS}_\zeta$. Fix a weakly null sequence $(x_n)_{n=1}^\infty\subset X$ and, seeking a contradiction, assume $$0<\ee=\inf\{\|x\|:F\in \mathcal{S}_{\zeta+\xi}, x\in \text{co}(x_n:n\in F)\}.$$  By passing to a subsequence, we may assume $$\ee\leqslant \inf \{\|x\|: F\in \mathcal{S}_\xi[\mathcal{S}_\zeta], x\in \text{co}(x_n: n\in F)\}.$$

 Since $(x_n+Y)_{n=1}^\infty$ is weakly null in $X/Y$, it is $\zeta$-weakly null.   Thus there exist $F_1<F_2<\ldots$, $F_i\in \mathcal{S}_\zeta$,  and positive scalars $(a_i)_{i\in \cup_{n=1}^\infty F_n}$ such that $\sum_{i\in F_n}a_i=1$ and $\|\sum_{i\in F_n}a_ix_i +Y\|<\min \{\ee/2, 1/n\}$.  For each $n\in\nn$, we fix $y_n\in Y$ such that $\|y_n-\sum_{i\in F_n}a_ix_i\|<\min \{\ee/2, 1/n\}$.   Since $(x_n)_{n=1}^\infty$ is weakly null, so are $(\sum_{i\in F_n}a_ix_i)_{n=1}^\infty$ and $(y_n)_{n=1}^\infty$.   Since $Y\in \textsf{wBS}_\xi$, there exist $G\in \mathcal{S}_\xi$ and positive scalars $(b_n)_{n\in G}$ such that $\sum_{n\in G}b_n=1$ and $\|\sum_{n\in G}b_ny_n\|<\ee/2$. Since $\cup_{n\in G}F_n\in \mathcal{S}_\xi[\mathcal{S}_\zeta]$, $$\ee\leqslant \|\sum_{n\in G}\sum_{i\in F_n}b_na_ix_i\|\leqslant \|\sum_{n\in G}b_ny_n\|+\sum_{n\in G}b_n\|y_n-\sum_{i\in F_n}a_ix_i\|<\ee/2+\ee/2=\ee,$$ and this contradiction finishes the first statement.

Now if $\zeta=0=\xi$, let $X$ be any finite dimensional space and let $Y=X$.  If $\zeta=0$ and $\xi>0$, let $(\xi_n)_{n=1}^\infty$ be any sequence such that $\sup_n \xi_n+1=\xi$. Let $X=(\oplus_{n=1}^\infty X_{\xi_n})_{\ell_1}$ and let $Y=X$. If $\xi=0$ and $\zeta>0$, let $(\zeta_n)_{n=1}^\infty$ be any sequence such that $\sup_n \zeta_n+1=\zeta$. Let $X=(\oplus_{n=1}^\infty X_{\zeta_n})_{\ell_1}$ and let $Y=\{0\}$. Each of these choices is easily seen to be the example we seek in these trivial cases.

We now turn to the non-trivial case, $\xi, \zeta>0$.  Fix $(\xi_n)_{n=1}^\infty$ such that if $\xi$ is a successor, $\xi_n+1=\xi$ for all $n\in\nn$. Otherwise let $(\xi_n)_{n=1}^\infty$ be the sequence such that $$\mathcal{S}_\xi=\{E\in [\nn]^{<\nn}: \exists n\leqslant E\in \mathcal{S}_{\xi_n}\}.$$   Let $(\zeta_n)_{n=1}^\infty$ be chosen similarly.  Let $I_{m,n}:X_{\zeta+\xi_m}\to X_{\zeta_n}$ be the canonical inclusion, which is bounded, since $\zeta+\xi_m\geqslant \zeta>\zeta_n$.  Let $a_{m,n}=\|I_{m,n}\|^{-1}$.   For each $m\in\nn$, let  $Z_m=(\oplus_{n=1}^\infty X_{\zeta_n})_{\ell_1}$ and let $Z=(\oplus_{m=1}^\infty Z_m)_{\ell_1}$.  Define $J_m:X_{\zeta+\xi_m}\to Z_m$ by $J_m(w)=(2^{-n} a_{m,n}I_{m,n}w)_{n=1}^\infty$. Note that $\|J_m\|\leqslant 1$.  Now let $W=(\oplus_{m=1}^\infty X_{\zeta+\xi_m})_{\ell_1}$ and define $S:W\to Z$ by letting $S|_{X_{\zeta+\xi_m}}=J_m$.  Note that $\|S\|\leqslant 1$.   Let $q:\ell_1\to Z$ be a quotient map. Let $X=\ell_1\oplus_1 W$ and define $T:X\to Z$ by $T(x,w)=qx+Sw$.  Then $T$ is also a quotient map, and, with $Y=\ker(T)$, $X/Y=Z$.   Since $\zeta_n<\zeta$, $X_{\zeta_n}\in \textsf{wBS}_\zeta$. Since $\textsf{wBS}_\zeta$ is closed under $\ell_1$ sums, $Z_m$ and $Z$ lie in $\textsf{wBS}_\zeta$.    Fix $\gamma<\zeta+\xi$ and note that there exists $m\in\nn$ such that $\gamma\leqslant \zeta+\xi_m$.  Since $X$ contains an isomorph of $X_{\zeta+\xi_m}$, the basis of which is not $\zeta+\xi_m$-weakly null, $X\in \complement \textsf{wBS}_\gamma$. It remains to show that $Y\in \textsf{wBS}_\xi$.  To that end, fix a weakly null sequence $((x_n, w_n))_{n=1}^\infty\subset B_{\ker(T)}$.   Then $x_n\to 0$, and $Tx_n\to 0$. From this it follows that $Sw_n\to 0$.    Seeking a contradiction, assume that $$0<\ee = \inf\{\|z\|: F\in \mathcal{S}_\xi, z\in \text{co}((x_n, w_n): n\in F)\}.$$  By passing to a subsequence, we may assume $\|x_n\|<\ee/2$ for all $n$, whence $$\ee/2\leqslant \inf\{\|w\|: F\in \mathcal{S}_\xi, w \in \text{co}(w_n: n\in F)\}.$$    Since $(w_n)_{n=1}^\infty \subset W$ is weakly null, there exists $k\in\nn$ such that for all $n\in\nn$, $$\sum_{m=k+1}^\infty \|w_{n,m}\|_{X_{\zeta+\xi_m}}<\ee/4,$$ where $w_n=(w_{n,m})_{m=1}^\infty$.  Since $Sw_n\to 0$, it follows that for all $m\in\nn$,  $J_mw_{n,m}\underset{n}{\to}0$.  In particular, for every $\beta<\zeta$ and $m\in\nn$, $\lim_n \|w_{n,m}\|_\beta=0$.    By passing to a subsequence $k$ times, once for each $1\leqslant m\leqslant k$, we may assume $(w_{n,m})_{n=1}^\infty$ is dominated by a subsequence of the $X_{\xi_m}$ basis. For this we are using Proposition \ref{phenomenal}$(ii)$.       Since $\xi_m<\xi$, $(w_{n,m})_{n=1}^\infty$ is $\xi$-weakly null for each $1\leqslant m\leqslant k$.   From this it follows that there exist $F\in \mathcal{S}_\xi$ and positive scalars $(a_n)_{n\in F}$ such that $\sum_{n\in F}a_n=1$ and for each $1\leqslant m\leqslant k$, $\|\sum_{n\in F}a_n w_{m,n}\|_{\zeta+\xi_m}<\ee/4k$.    Then $$\ee/2 \leqslant \|\sum_{n\in F}a_nw_n\|\leqslant \sum_{m=1}^k \|\sum_{n\in F}a_nw_{m,n}\|_{\zeta+\xi_m} + \sum_{n\in F}a_n \sum_{m=k+1}^\infty \|w_{m,n}\|_{\zeta+\xi_m} <\ee/4+\ee/4=\ee/2,$$ a contradiction.

\end{proof}

\end{document}